\documentclass{amsart}

\newtheorem*{theorem}{Main Theorem}
\newtheorem*{proposition*}{Proposition}
\newtheorem{proposition}{Proposition}[subsection]
\newtheorem{lemma}[proposition]{Lemma}

\theoremstyle{definition}
\newtheorem{definition}[proposition]{Definition}
\newtheorem{remark}[proposition]{Remark}
\newtheorem{example}[proposition]{Example}
\newtheorem*{notation}{Notation}
\newtheorem*{acknowledgment}{Acknowledgments}

\numberwithin{equation}{subsection}

\usepackage{amssymb, latexsym}
\usepackage[mathscr]{eucal}
\usepackage{graphicx}
\usepackage{color, calc}
\usepackage{setspace}
\usepackage{verbatim, enumerate}
\usepackage[backref, colorlinks]{hyperref}

\DeclareMathOperator{\bdry}{\mathrm{Bdry}}
\DeclareMathOperator{\dist}{\mathrm{dist}}

\DeclareMathOperator{\diag}{\mathrm{diag}}
\DeclareMathOperator{\derv}{\mathrm{d}\!}
\DeclareMathOperator{\divg}{\mathrm{div}}

\DeclareMathOperator{\tng}{\mathrm{T}}

\newcommand{\crci}{\ensuremath{\text{\textcircled{\raisebox{-0.2ex}{\textbf{1}}}}}}
\newcommand{\crcii}{\ensuremath{\text{\textcircled{\raisebox{-0.3ex}{\textbf{2}}}}}}
\newcommand{\crciii}{\ensuremath{\text{\textcircled{\raisebox{-0.2ex}{\textbf{3}}}}}}
\newcommand{\crciv}{\ensuremath{\text{\textcircled{\raisebox{-0.2ex}{\textbf{4}}}}}}
\newcommand{\crcv}{\ensuremath{\text{\textcircled{\raisebox{-0.2ex}{\textbf{5}}}}}}
\newcommand{\crcvi}{\ensuremath{\text{\textcircled{\raisebox{-0.2ex}{\textbf{6}}}}}}
\newcommand{\crcvii}{\ensuremath{\text{\textcircled{\raisebox{-0.2ex}{\textbf{7}}}}}}
\newcommand{\crcviii}{\ensuremath{\text{\textcircled{\raisebox{-0.2ex}{\textbf{8}}}}}}

\begin{document}

\title{Hypersurfaces with central convex cross-sections}
\date{\today}

\author{Metin Alper Gur}

\address{Department of Mathematics\\
	      Indiana University\\
	      Bloomington, IN 47405}

\subjclass[2010]{52A20, 53A07, 53A15}
\keywords{Quadric hypersurfaces, ovaloids, central symmetry}

\begin{abstract}
	A hypersurface $M \subseteq \mathbb{R}^{n}$, $n \geq 4$, has the central ovaloid property, or \emph{cop}, if
	\begin{itemize}
		\item $M$ meets some hyperplane transversally along an ovaloid, and
		\item Every such ovaloid on $M$ has central symmetry.
	\end{itemize}
	Generalizing work of  B. Solomon to higher dimensions, we show that a complete, connected, smooth hypersurface with 		\emph{cop} must either be a cylinder over a central ovaloid, 	or else quadric.
\end{abstract}

\maketitle

\section{Introduction and Overview}

\subsection{Introduction and Main Theorem}

A central set in a euclidean space has symmetry with respect to reflection through a point, called its center. A closed embedded smooth hypersurface of $\mathbb{R}^{n}$ is called an ovaloid if all its principal curvatures, with respect to the outer unit normal, are positive everywhere. An ovaloid of dimension one is also called an oval.

\begin{figure}[h]
	\centering
	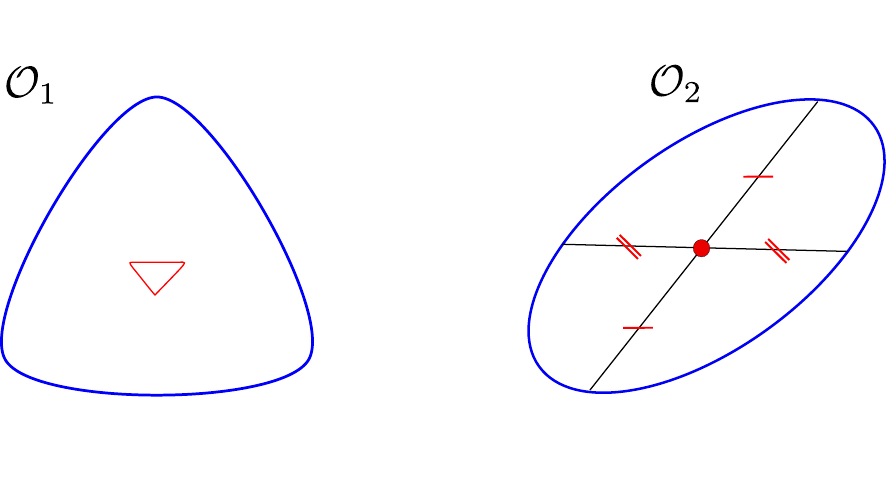
	\caption[Centrices of ovals]{Ovals $\mathcal{O}_{1}$, $\mathcal{O}_{2}$, and their centrices. Only $\mathcal{O}_{2}$ 						is central}
\end{figure}

The compact transverse cross-sections of a cylinder over a central ovaloid in $\mathbb{R}^{n}$, $n \geq 3$, with hyperplanes are central ovaloids. A similar result holds also for quadrics, which are the level sets of quadratic polynomials in $\mathbb{R}^{n}$, $n \geq 3$. Their compact transverse cross-sections with hyperplanes are ellipsoids, which are central ovaloids.

Following Solomon who showed that these two kinds of examples provide the only complete smooth hypersurfaces in $\mathbb{R}^{3}$, whose ovaloid cross-sections are central, we show that the same conclusion also holds in $\mathbb{R}^{n}$, $n \geq 4$. Roughly, we will say that a smooth hypersurface $M^{n-1} \subseteq \mathbb{R}^{n}$, $n \geq 4$, has the \emph{central ovaloid property}, or \emph{cop}, if
\begin{itemize}
	\item $M$ intersects at least one hyperplane transversally along an ovaloid, and
	\item Every such ovaloid is central.
\end{itemize}

Given this set-theoretic definition of \emph{cop}, our main result can be stated as follows:

{\em A complete, connected smooth hypersurface $M^{n-1}$ in $\mathbb{R}^{n}$, $n \geq 4$, with \emph{cop} is either a 	       cylinder over a central ovaloid, or a quadric.}

More precisely, the \emph{central ovaloid property} and our main theorem can be stated in terms of mappings as follows:

\begin{definition}[cross-cut] \label{DS1.1:ccut}
	If $M^{n-1}$ is a smooth manifold, $F \colon M^{n-1} \to \mathbb{R}^{n}$, $n \geq 4$, is an immersion, and $H 			\subseteq \mathbb{R}^{n}$ is a hyperplane transversal to $F$, a cross-cut of $F$ relative to $H$ is a compact connected 		component $\Gamma \subseteq F^{-1}(H)$.
\end{definition}

\begin{definition}[\emph{cop}] \label{DS1.1:cop}
	Let $M^{n-1}$ be a smooth manifold, and $F \colon M \to \mathbb{R}^{n}$, be an immersion, $n \geq 4$, then $F$  is 		said to have the \emph{central ovaloid property}, or, \emph{cop}, if
	\begin{itemize}
		\item There exists at least one cross-cut whose image is an ovaloid, and
		\item Every such ovaloid is central.
	\end{itemize}
\end{definition}

Given this precise terminology, the main theorem can be stated as follows:

\begin{theorem}
	Let $M^{n-1}$ be a smooth, connected manifold, and $F \colon M \to \mathbb{R}^{n}$, $n \geq 4$, be a proper, 			complete immersion with \emph{cop}, then $F(M)$ is either a cylinder over a central ovaloid or a quadric.
\end{theorem}

\subsection{Historical Background}

Historically, the first theorem of this sort was proved by W. Blaschke.

\begin{proposition*} \cite{wB18}
	Suppose every plane transverse, and nearly tangent to, a smooth convex surface $S \subseteq \mathbb{R}^{3}$ intersects 		$S$ along a central loop. Then $S$ is a quadric.
\end{proposition*}

B. Solomon, in 2009, removed the convexity assumption and first considered hypersurfaces of revolution as follows:

\begin{proposition*} \cite{bS09}
	Let $M^{n-1} \subseteq \mathbb{R}^{n}$, $n \geq 3$, be a hypersurface of revolution. If $M$ intersects every hyperplane 	nearly perpendicular to its axis of rotation in a central set, then $M$ is a quadric.
\end{proposition*}

Using the rotationally symmetric case, Solomon then proved the $n = 3$ analogue of our main theorem in 2012.

\begin{proposition*} \cite{bS12}
	Let $M \subseteq \mathbb{R}^{3}$ have the \emph{central oval property}, i.e., $M$ intersects some plane along an oval, 		and every such oval is central. Then a complete, connected, smooth surface $M$ with \emph{central oval property} must 		either be a cylinder over a central oval or a quadric.
\end{proposition*}

\subsection{Overview of the proof of main theorem}

Our main result is the generalization of Solomon's result \cite{bS12} to all dimensions $n \geq 3$. The proof is not inductive, and although it is, at certain times, a direct generalization of Solomon's method, some essentially new difficulties arose. Moreover, the technical tools required to carry on the argument are more sophisticated and the calculations are more elaborate.

As in Solomon's paper, we obtain our main result by first working the local version. If an immersed hypersurface $M$ intersects some hyperplane along an ovaloid as our definition of \emph{cop} requires, then some neighborhood, in $M$, of that ovaloid embeds into $\mathbb{R}^{n}$ as a tube with \emph{cop}. The proof that these tubes are either cylindrical or quadric constitute the local version of our work. Unless noted otherwise, $I$ denotes the open interval $(-1, 1)$.

\begin{definition}[Transversely convex tube]
	Suppose $X \colon \mathbb{S}^{n-2} \times I \to \mathbb{R}^{n}$, $n \geq 3$, is an embedding of the form
	\begin{equation} \label{E:trcvxtstr}
		X(u, z) = (\, c(z) + \alpha(u, z), z \,)
	\end{equation}
	where $c \colon I \to \mathbb{R}^{n-1}$ and $\alpha \colon \mathbb{S}^{n-2} \times I \to \mathbb{R}^{n-1}$ are 			smooth, and for each fixed $z \in I$, the map $\alpha(\cdot, z) \colon \mathbb{S}^{n-2} \to \mathbb{R}^{n-1}$ 			parameterizes an ovaloid with center of mass at the origin.
\end{definition}

\begin{figure}[h]
	\centering
	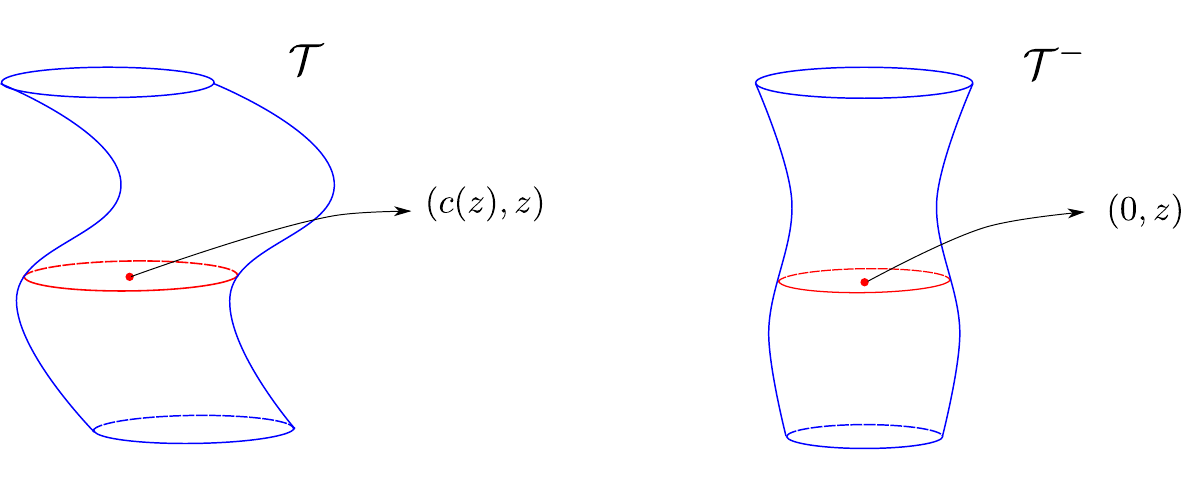
	\caption[Transversely Convex Tubes]{Transversely convex tube $\mathcal{T}$ with its rectification $\mathcal{T^{-}}$}
\end{figure}

A transversely convex tube is any embedded annulus that, after an affine isomorphism, can be parameterized this way. A transversely convex tube in \emph{standard position} is the image of an embedding $X$ of the form \eqref{E:trcvxtstr}. If we discard the  central curve $c$ of a transversely convex tube $\mathcal{T}$ in \emph{standard position}, we get the rectification $\mathcal{T}$, denoted $\mathcal{T^{-}}$, the image of
\[
	X^{-}(u, z) = (\, \alpha(u, z), z \,).
\]

We say that $\mathcal{T^{-}}$ splits when we can separate variables:
\[
	\alpha(u, z) = r(z) \alpha_{0}(u)
\]
for some functions $r \colon I \to (0, \infty)$ and $\alpha_{0} \colon \mathbb{S}^{n-2} \to \mathbb{R}^{n-1}$ parameterizing a fixed ovaloid. As in Solomon's paper we get a Splitting lemma:
\vspace{0.25cm}

\noindent
\textbf{Proposition} \ref{PS5.2:split} (Splitting lemma)\textbf{.} {\itshape If a transversely convex tube $\mathcal{T}$ in \emph{standard position} has \emph{cop}, then its rectification $\mathcal{T}^{-}$ splits.}
\vspace{0.25cm}

To obtain this Splitting lemma we derive a pair of partial differential equations satisfied by the function $h \colon \mathbb{S}^{n-2} \times I \to \mathbb{R}$ which, for each $z \in I$, yields the support function $h(\cdot, z)$ of the ovaloid $\mathcal{O}(z) = \alpha(\mathbb{S}^{n-2}, z)$. Using the second PDE we cook up a strictly elliptic operator with bounded coefficients and use Harnack's inequality to get our Splitting lemma. After we are able to split the function $\alpha$ we use the first PDE to obtain the following local version of our main result for the rectified tube:
\vspace{0.25cm}

\noindent
\textbf{Proposition} \ref{PS5.2:classofrcttube}\textbf{.} {\itshape Suppose $\mathcal{T}$ is a transversely convex tube in \emph{standard position} with \emph{cop}. Then its rectification $\mathcal{T^{-}}$ is either
\begin{enumerate}
	\item a cylinder over a central ovaloid, or
	\item affinely isomorphic to a hypersurface of revolution.
\end{enumerate}}
\vspace{0.25cm}

\noindent
\textbf{Proposition} \ref{PS5.3:axslemm} (Axis lemma)\textbf{.}\! {\itshape Suppose $\mathcal{T}$ is a transversely convex tube with \emph{cop}. Then its central curve is affine and $\mathcal{T}$ is affinely isomorphic to its rectification $\mathcal{T^{-}}$.}
\vspace{0.25cm}

Combining Axis lemma and \cite{bS09} we can conclude that a transversely convex tube with \emph{cop} is either cylindrical or quadric. This local version of our main result can be stated as
\vspace{0.25cm}

\noindent
\textbf{Proposition} \ref{PS5.4:lclvrs} (Local version)\textbf{.} {\itshape A transversely convex tube with \emph{cop} is either a cylinder over a central ovaloid or a quadric.}
\vspace{0.25cm}

The final steps of the proof of our main theorem are as follows:

Given any immersion with \emph{cop} and a cross-cut, Lemma \ref{LS3.3:exttubnbd} (Tubular neighborhood) shows the existence of a neighborhood of this cross-cut foliated by diffeomorphic copies. Then Lemma \ref{LS3.3:exsttrncvxtube} (Existence of transversely convex tube) shows that this neighborhood is embedded onto a transversely convex tube with \emph{cop} about the image of the cross-cut, when that image is an ovaloid. By Proposition \ref{PS5.4:lclvrs} (Local version), the tube is either cylindrical or quadric. But the boundaries of such a tube, in either case, are again images of cross-cuts and therefore, using Lemma \ref{LS3.3:exttrncvxtube} (Extension of transversely convex tube) we can push the boundaries of the annular region in $M$ and the tubular region in $\mathbb{R}^{n}$ a little further. Since $M$ is connected and complete, the extension process stops only when the annular region fills up $M$ and the image is either a complete cylinder or a quadric.

\begin{acknowledgment}
	We would like to thank our advisor Bruce Solomon for giving us this project, and providing guidance. We would also like to 	express our gratitude to the Indiana University Mathematics Department for creating an academically challenging and 		supportive 	environment, and to the Graduate School for the financial support that made our studies possible.
\end{acknowledgment}

\section{Convex Geometry and Support Function} \label{S:cvxgeo}

In this section, we provide the necessary basic definitions and auxiliary lemmas from convex geometry. All definitions and lemmas come from the book of Schneider \cite{rS14}. This section ends with a computation of the support function of an ellipsoid.

\subsection{Convex Sets}

In this subsection, we provide the definition of convex set, convex body, and strictly convex body.

\begin{definition}
	A set $A \subseteq \mathbb{R}^{n}$ is convex if for every $x, \, y \in A$ it contains the segment
	\[
		[x, y] = \big\{\, (1-\lambda) x + \lambda y \colon 0 \leq \lambda \leq 1 \,\big\}.
	\]
	A set $K \subseteq \mathbb{R}^{n}$ is called a convex body if $K \ne \emptyset$, convex, and compact. 	By $\mathcal{K}	^{n}$ we denote the set of all convex bodies in $\mathbb{R}^{n}$. 
\end{definition}

\begin{definition}
	$A \subseteq \mathbb{R}^{n}$, the affine hull of $A$, $\mathrm{aff} A$, is the smallest affine subspace of $\mathbb{R}		^{n}$ containing $A$. The relative boundary, $\mathrm{relBdry}A$, is the boundary of $A$ relative to its affine hull.
\end{definition}

\begin{definition}
	A convex body $K \in \mathcal{K}^{n}$ is called a strictly convex body if $\mathrm{relBdry}K$ does not contain any line 	segment.
\end{definition}

\begin{figure}[h]
	\centering
	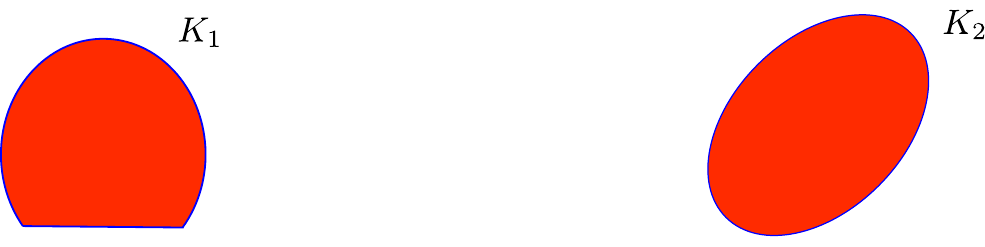
	\caption[Convex bodies]{$K_{1}$ is not strictly convex, $K_{2}$ is strictly convex}\label{FS2.1:cvxbdy}
\end{figure}

\subsection{The Support Function}

In this subsection, we introduce the support function and list its important properties that will be relevant to us in later parts. In particular, we remark how the gradient of a support function determines the boundary points of a convex set. We also give a rigorous definition of the centrix and we end this subsection with a calculation of the support function of an ellipsoid.

\begin{definition}
	Let $K \subseteq \mathbb{R}^{n}$, $\emptyset \ne K \ne \mathbb{R}^{n}$, be a convex body. The support function 			$h(K, \cdot) = h_{K}$ of $K$ is defined on $\mathbb{R}^{n}$ by 
	\[
		h(K, u) = h_{K}(u) = \sup \big\{ x \cdot u \colon x \in K \big\}.
	\]
\end{definition}

\begin{lemma} \label{LS2.2:spfnpr} \cite[p.44, 45] {rS14}
	Let $K \in \mathcal{K}^{n}$, then the support function $h_{K}$ of $K$ has the following properties:
	\begin{enumerate}
		\item $h_{K}$ is sublinear. \label{LS2.2:spfnpr2}
		\item $h(\lambda K, \cdot) = \lambda h(K, \cdot)$ and $h(K, -u) = h(-K, u)$ for all $\lambda \geq 0$	.																			\label{LS2.2:spfnpr3}
		\item Given $T \in \mathrm{O}(n)$, $h_{TK} = h_{K} \circ T^{-1}$.		 \label{LS2.2:spfnpr5}		
		\item $h_{K + v}(u) = h_{K}(u) + u \cdot v$ for all $u \in \mathbb{R}^{n}$ and for all $u$, $v \in \mathbb{R}^{n}$.
																  \label{LS2.2:spfnpr6}
	\end{enumerate}
\end{lemma}

The item \eqref{LS2.2:spfnpr2} of Lemma \ref{LS2.2:spfnpr} shows that a support function is sublinear. However, the converse is also true.

\begin{lemma} \label{LS2.2:subvcxbd} \cite[p.45]{rS14}
	If $f \colon \mathbb{R}^{n} \to \mathbb{R}$ is a sublinear function, then there exists a unique convex body $K \in 			\mathcal{K}^{n}$ with support function $f$.
\end{lemma}

\begin{figure}[h]
	\centering
	\input{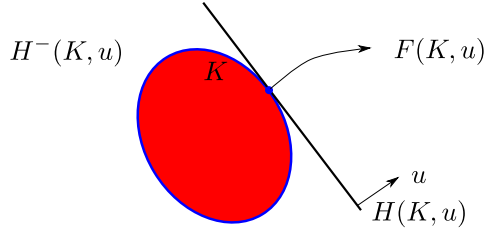}
	\caption{Supporting halfspace, support hyperplane, and support set}
\end{figure}

\begin{definition}	
	Let $K \in \mathcal{K}^{n}$ be a convex body and  $u \in \mathbb{R}^{n}$, then
	\[
		H(K, u) = \big\{ x \in \mathbb{R}^{n} \colon x \cdot u = h_{K}(u) \big\}
	\]
	is called the support hyperplane of $K$ with outer normal $u$,
	\[
		H^{-}(K, u) = \big\{ x \in \mathbb{R}^{n} \colon x \cdot u \leq h_{K}(u) \big\}
	\]
	is called the supporting halfspace of $K$ with outer normal $u$, and 
	\[
		F(K, u) = H(K, u) \cap K
	\]
	is called the support set of $K$ with outer normal $u$.	
\end{definition}

For $u \in \mathbb{S}^{n-1}$, $h_{K}(u)$ is the signed distance of the support hyperplane $H(K, u)$ to the origin.

Note that the support set of convex body does not need to be a singleton. In particular, the convex body $K_{1}$ in \autoref{FS2.1:cvxbdy} has a support set, which is a line segment. However, if the support set is a singleton then we have an important analytical consequence about the support function.

\begin{lemma} \label{LS2.2:diffsptfct} \cite[p.47]{rS14}
	Let $K \in \mathcal{K}^{n}$ and $u \in \mathbb{R}^{n} \setminus \{0\}$. The support function $h_{K}$ is differentiable 	at $u$ if and only if $F(K, u) = \{ x \}$. In this case
	\[
		\nabla_{\mathbb{R}^{n}} \, h_{K} \big\rvert_{u} = x.
	\]
\end{lemma}

\begin{definition}[centrix] \label{DS2.2:centrix}
	Let $K \in \mathcal{K}^{n}$ and $h_{K}$ be differentiable in $\mathbb{R}^{n} \setminus \{0\}$, then the centrix of $K$ 	is defined as
	\begin{align*}
		c_{K}      &\colon \mathbb{R}^{n} \setminus \{0\} \to \mathbb{R}^{n}\\
		c_{K}(u) &= c(K, u) = \frac{\nabla_{\mathbb{R}^{n}} \, h_{K}(u)  + \nabla_{\mathbb{R}^{n}} \, h_{K}(-u)}{2}
	\end{align*}
	Note that the function $u \mapsto \nabla_{\mathbb{R}^{n}} \, h_{K}(u)$ and hence the centrix of $K$ are positively 		homogeneous of degree zero.
\end{definition}

\begin{figure}[h]
	\centering
	\input{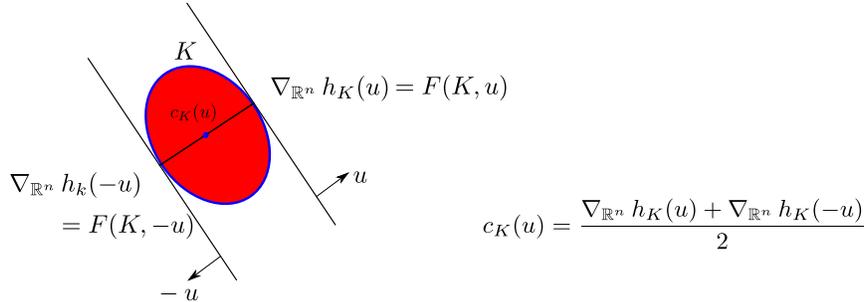}
	\caption{The centrix $c_{K}(u)$}
\end{figure}

We end this subsection with a computation of the support function of an ellipsoid. This support function will play an important rule in the proof of the main theorem.

\begin{example} \label{ES2.5:sptfnctellipsoid} [support function of ellipsoid]
	Given $n$ positive numbers $\lambda_{i} > 0$ $i = 1, \dotsc, n$, the function $h \colon \mathbb{R}^{n} \to \mathbb{R}$ 	defined as
	\[
		h(x) = \left[ \sum_{i=1}^{n} \lambda_{i}^{2} x_{i}^{2} \right]^{1/2}
	\]
	is sublinear and differentiable in $\mathbb{R}^{n} \setminus \{0\}$. By Lemma \ref{LS2.2:subvcxbd} there exists a unique 	convex body $K \in \mathcal{K}^{n}$, with $h_{K} = h$ and 
	\[
		\mathrm{Bdry} \, K \supseteq \left\{\, \nabla_{\mathbb{R}^{n}} \, h_{K} \big\rvert_{x} \colon x \ne 0 \,\right\}
	\]
	We want to show that the set consisting of the gradients of $h_{K}$ equals the complete ellipsoid
	\begin{equation} \label{E:stdellipsoid}
		E = \left\{ x \in \mathbb{R}^{n} \colon \sum_{i=1}^{n} \frac{x_{i}^{2}}{\lambda_{i}^{2}} = 1 \right\}
	\end{equation}
	Let $x \in \mathbb{R}^{n} \setminus \{0\}$, then the gradient of $h_{K}$ at $x$ equals
	\begin{align*}
		\nabla_{\mathbb{R}^{n}} \, h_{K} \big\rvert_{x} 								     &= \frac{1}{h_{K}		(x)} (\lambda_{1}^{2} x_{1}, \dotsc, \lambda_{n}^{2} x_{n}) \text{ and hence}\\
		\sum_{i=1}^{n} \frac{\big[ \lambda_{i}^{2} x_{i} h_{K}^{-1}(x) \big]^{2}}{\lambda_{i}^{2}} &= \frac{1}				{h_{K}^{2}(x)} \sum_{i=1}^{n} \lambda_{i}^{2} x_{i}^{2} = 1. 
	\end{align*}
	On the other hand let $x \in \mathbb{R}^{n}$ satisfy
	\begin{align*}
		\sum_{i=1}^{n} \frac{x_{i}^{2}}{\lambda_{i}^{2}} &= 1 \text{ and choose $z \in \mathbb{R}^{n}$ as}\\
		z &= \left( \frac{x_{1}}{\lambda_{1}^{2}}, \dotsc, \frac{x_{n}}{\lambda_{n}^{2}} \right) \text{ then $h_{K}(z) = 			1$ and}\\  
		\nabla_{\mathbb{R}^{n}} \, h_{K} \big\rvert_{z} &= \left(\lambda_{1}^{2}z_{1}, \dotsc, \lambda_{n}^{2} z_{n} 			\right) = x
	\end{align*}
	So we can conclude that 
	\[
		\mathrm{Bdry} \, K \supseteq \left\{ x \in \mathbb{R}^{n} \colon \sum_{i=1}^{n} \frac{x_{i}^{2}}{\lambda_{i}			^{2}} = 1 \right\}
	\]
	and since $K$ is convex the equality must hold, and $K$ is the convex body with boundary an origin centered ellipsoid. 		From now on $h_{K} = h$ will also be called the support function of the origin centered ellipsoid $E$ given in 				\eqref{E:stdellipsoid}. Given any ellipsoid $\widetilde{E}$ in $\mathbb{R}^{n}$ there exist $T \in \mathrm{O}(n)$ and 		$v \in \mathbb{R}^{n}$ so that $\widetilde{E} = TE + v$, hence the support function $h_{\widetilde{E}}$ of $			\widetilde{E}$ can be calculated using the properties \eqref{LS2.2:spfnpr5} and \eqref{LS2.2:spfnpr6} of the support 		function listed in Lemma \ref{LS2.2:spfnpr} in the following manner:
	\begin{align*}
		h_{\widetilde{E}}(u) &= h_{TE + v}(u) = h_{TE}(u) + u \cdot v\\
					         &= \big( h_{E} \circ T^{-1} \big)(u) + u \cdot v
	\end{align*}
	for every $u \in \mathbb{R}^{n}$. 
\end{example}

\section{Ovaloid and Tubular Neighborhood} \label{S:diffgeo}

Here, we provide the necessary basic definitions and auxiliary lemmas from multilinear algebra and differential geometry. The references for this material are the books of  Federer \cite{hF69}, Hirsch \cite{mH76}. We end this section by sketching proofs of the existence and extension of a transversely convex tube. 

\subsection{Ovaloid and Support Parameterization}

In this subsection, we introduce a principal geometric object of interest, the ovaloid in $\mathbb{R}^{n}$. Using Hadamard's theorem \cite[p.41]{jH98, sKkN69}, in the case $n \geq 3$, we note that every ovaloid has a support parameterization. Writing the Laplacian in polar coordinates and using Lemma \ref{LS2.2:diffsptfct} we give an analytical expression for the support parameterization. 

Lastly, we recall that ovaloids are preserved under affine isomorphisms of $\mathbb{R}^{n}$, find a transversely convex tube about the image of a cross-cut when that image is an ovaloid, and determine sufficient conditions for being able to extend it.  

\begin{definition}[Ovaloid]
	An ovaloid $\mathcal{O}^{n-1} \subseteq \mathbb{R}^{n}$, $n \geq 2$ is a closed embedded smooth hypersurface with 		all principal curvatures positive everywhere with respect to the outer unit normal. 
\end{definition}
In particular, an ovaloid of dimension one, $\mathcal{O} \subseteq \mathbb{R}^{2}$, also called an oval, is an embedded smooth loop in $\mathbb{R}^{2}$ such that, given the outer unit normal vector field $N$ and the unit tangent vector field $T$ along $\mathcal{O}$ 
\[
	\overline{\nabla}_{T} N = \kappa_{g}T
\]
with $\kappa_{g}$, called the geodesic curvature, is positive everywhere.

It is a simple exercise of differential and convex geometry to show that that the Gauss map of an oval is a diffemorphism and that an oval bounds a strictly convex body. The analogous statement holds in higher dimensions by a theorem due to Hadamard \cite{jH98}, with an alternate proof in \cite[p.41]{sKkN69}. Hadamard's theorem implies that any ovaloid $\mathcal{O}^{n-1} \subseteq \mathbb{R}^{n}$, $n \geq 3$ is the boundary of a strictly convex body and that its Gauss map is a diffeomorphism. Using this remark we can make the following definition.

\begin{definition}[Support Parameterization]
	The support parameterization of the ovaloid $\mathcal{O}^{n-1} \subseteq \mathbb{R}^{n}$, $n \geq 2$, that bounds a 		strictly convex body $K$ is defined as the inverse of the Gauss map
	\begin{align*}
	\Gamma = N^{-1} \colon \mathbb{S}^{n-1} &\to \mathcal{O}^{n-1}\\
					      			       u &\mapsto \Gamma(u).
	\end{align*}
	Since $\Gamma$ is bijective and for each $u \in \mathbb{S}^{n-1}$, the outer unit normal vector $u$ to the hyperplane 		$H(K, u)$ at the points of $F(K, u) = K \cap H(K, u) \subseteq \mathcal{O}^{n-1}$ is also normal to the boundary $			\mathcal{O}^{n-1}$, we can conclude that $F(K, u)$ must be a singleton. Now using this observation and Lemma 			\ref{LS2.2:diffsptfct} we can write
	\[
		\bigl\{\, \Gamma(u) \,\bigr\} = F(K, u) = \bigl\{\, \nabla_{\mathbb{R}^{n}} h_{K}(u) \,\bigr\}.
	\]
	So by the sublinearity property of the support function, introduced in item \eqref{LS2.2:spfnpr2} of Lemma 				\eqref{LS2.2:spfnpr}, and the representation of $\nabla_{\mathbb{R}^{n}}$ in polar coordinates we can conclude that the 		support parameterization is also equal to
	\begin{equation}
	\begin{aligned} \label{ES3.2:sptparam}
		\Gamma(u) &= \nabla_{\mathbb{R}^{n}} h_{K}(u)\\
				  &= \nabla_{\mathbb{S}^{n-1}} h_{K}(u) + h_{K}(u)u 
	\end{aligned}
	\end{equation}
	for all $u \in \mathbb{S}^{n-1}$.
\end{definition}

Using the the $\star$-operator defined in \cite[p.34]{hF69} we get, up to a sign, a unique coordinate representation for the unit normal field defined along a smooth orientable hypersurface.

\begin{remark} \label{RS3.2:outrunitnorfld}
 	If $M^{n-1} \subseteq \mathbb{R}^{n}$ is a hypersurface and $x \colon U \subseteq \mathbb{R}^{n-1} \to \mathbb{R}		^{n}$ is a local parameterization of $M$ then up to a sign the unique unit normal field along $x(U)$ is defined as
	\begin{equation} \label{DS3.2:outrunitnorfld}
		u \in U \mapsto \frac{\star \left( \bigwedge \nolimits_{n-1} \derv x_{u}\right)  \left(\frac{\partial}{\partial u_{1}} 			\wedge \cdots \wedge 	\frac{\partial}{\partial u_{n-1}}\right)}{\left \lvert  \star \left( \bigwedge \nolimits_{n-1} \derv 		x_{u}\right)  \left(\frac{\partial}{\partial u_{1}} \wedge \cdots \wedge 	\frac{\partial}{\partial u_{n-1}}\right)\right 			\rvert}.
	\end{equation}	
	When $M^{n-1} \subseteq \mathbb{R}^{n}$ is a closed embedded hypersurface we let the mapping in 					\eqref{DS3.2:outrunitnorfld} define the unit normal field pointing into the unbounded component of $\mathbb{R}^{n} 		\setminus M$.	 
\end{remark}

During the course of our proof of the main theorem we want to use the fact that the hypotheses and the conclusion are invariant under affine isomorphisms of $\mathbb{R}^{n}$. To achieve this, we use the fact that ovaloids are preserved under affine isomorphisms of $\mathbb{R}^{n}$. 

\begin{lemma} \label{LS3.2:Glnprsvovld}
	If $n \geq 2$, $G$ is an affine isomorphism of $\mathbb{R}^{n}$, and $\mathcal{O}^{n-1} \subseteq \mathbb{R}^{n}$ is 	an ovaloid, then $G \mathcal{O}^{n-1}$ is again an ovaloid.
\end{lemma}

\begin{proof}[Sketch of proof]
	Normal curvatures are preserved under isometries and hence ovaloids are preserved under orthogonal maps and translations. 	Since each invertible linear map can be written as the composition of a symmetric linear map and an orthogonal 	map, it 		suffices to show that ovaloids are preserved under invertible symmetric maps. The last claim can be easily shown by 			writing the ovaloid and its image locally as a graph.  
\end{proof}

\subsection{Tubular Neighborhood}

In this subsection, our main goal is to construct a transversely convex tube $\mathcal{T}$ about the image of a cross-cut when that image is an ovaloid. In order to achieve this goal, we first need Lemma \ref{LS3.3:exttubnbd} (Tubular neighborhood), which provides a neighborhood for a cross-cut foliated by diffeomorphic copies. Since the cross-sections of $\mathcal{T}$ that are close enough to the image of the cross-cut must also be ovaloids we can use Lemma \ref{LS3.3:exsttrncvxtube} (Existence of transversely convex tube) to get the required tube.  

Another important goal is to give sufficient condition for extensibility of a transversely convex tube. In particular, we point out in Lemma \ref{LS3.3:exttrncvxtube} (Extension of transversely convex tube) that whenever the boundary of the tube is the image of a cross-cut then one can extend the tube a little further.

\subsubsection{Existence of tubular neighborhood}

\begin{definition} \
	Given $u \in \mathbb{R}^{n} \setminus \{0\}$, define the linear function $u^{*} \colon \mathbb{R}^{n} \to \mathbb{R}$ by 	$u^{*}(x) = u \cdot x$. For any given $u \in \mathbb{R}^{n} \setminus \{0\}$ and $\alpha \in \mathbb{R}$, the level 		hyperplane $H_{u, \alpha}$ is defined by
	\[
		H_{u, \alpha} = \{x \colon u^{*}(x) = \alpha\}.
	\]
\end{definition}

\begin{lemma}[Tubular neighborhood] \label{LS3.3:exttubnbd}
	Suppose $M$ is a smooth manifold, and $\Sigma$ is a compact connected embedded hypersurface of $M$.  Assume that for 	some $u \in \mathbb{R}^{n} \setminus \{0\}$ and $\alpha \in \mathbb{R}$, $F$ is transversal to $H_{u, \alpha}$ 			along $\Sigma$, then there exist $a > 0$ and an embedding $\psi \colon \Sigma \times [-a, a] \to M$ with the following 		properties:
	\begin{enumerate}
	\item $\psi(x, 0) = x$ for all $x \in \Sigma$,
	\item $\bigl( u^{*} \circ F \circ \psi)(x, h) = h$ for all $x \in \Sigma$, $\lvert h \rvert \leq a$,
	\item $\derv \, \bigl( u^{*} \circ F \circ \psi \bigr) (x, h) \ne 0$ for all $x \in \Sigma$, $\lvert h \rvert \leq a$.
	\end{enumerate}
\end{lemma}

\begin{remark} \label{RS3.3:ccutfoliate}
	According to Lemma \ref{LS3.3:exttubnbd} (Tubular neighborhood), $\Sigma$ is a cross-cut of $F$ relative to $H_{u, 		\alpha}$ and the neighborhood $U = \psi\bigl(\, \Sigma \times [-a, a] \,\bigr)$ of $\Sigma$ in $M$ is foliated by cross-cuts $	\psi(\Sigma \times \{t\})$ diffeomorphic to $\Sigma$, each a level set of $u^{*} \circ F$.
\end{remark}

\begin{proof}[Sketch of proof]
	Since the gradient $\nabla (u^{*} \circ F)$ does not vanish on the compact submanifold $\Sigma$ we can define the vectorfield
	\[
		X = \frac{\nabla (u^{*} \circ F)}{\lvert \nabla (u^{*} \circ F) \rvert^{2}}
	\]
	in an open neighborhood of $\Sigma$. Using a standard flow box argument for the vectorfield $X$ we then get the required 	embedding $\psi$.
\end{proof}

\subsubsection{Existence and extension of transversely convex tube}

\begin{lemma}[Existence of transversely convex tube] \label{LS3.3:exsttrncvxtube}
	Suppose $n \geq 4$, the map $F \colon M^{n-1} \to \mathbb{R}^{n}$ is an immersion, $H = H_{u, \alpha}$ is a hyperplane, 	and $\Sigma \subseteq F^{-1}(H)$ is a compact connected embedded hypersurface of $M$. If  $F(\Sigma)$ is an ovaloid and 	the function $u^{*}\circ F$ has no critical point on $\Sigma$ then $F$ maps some neighborhood of $\Sigma$ in $M$ onto a 	transversely convex tube.
\end{lemma}

\begin{proof}[Sketch of proof]
	We can compose the immersion $F$ with the embedding $\psi$, obtained in Lemma \ref{LS3.3:exsttrncvxtube}, to get an 		embedding because $F$ embeds $\Sigma$ onto a simply connected ovaloid and the set of embeddings is  open in the strong 	topology \cite[p.38]{mH76}. Since the principal curvatures are continuous there exists an open neighborhood of  $F(\Sigma)$ 	on the image of $F \circ \psi$ so that each horizontal cross section is again an ovaloid.
\end{proof}

\begin{remark}
	Now suppose $F$, $\Sigma$ are as above and $\mathcal{T}_{a}$ is the transversely convex tube whose existence is 			guaranteed  by Lemma \ref{LS3.3:exsttrncvxtube}. Note that if $F$ has  \emph{cop} so does $\mathcal{T}_{a}$. Let $U_{a} 	= \psi \bigl( \Sigma \times [-a, a] \bigr)$ be the neighborhood of $\Sigma$ that maps under $F$ onto $\mathcal{T}_{a}$ as described in 	Lemma \ref{LS3.3:exttubnbd} and Lemma \ref{LS3.3:exsttrncvxtube}.
\end{remark}

\begin{lemma}[Extension of transversely convex tube] \label{LS3.3:exttrncvxtube}
	Let $F \colon M^{n-1} \to \mathbb{R}^{n}$ be an immersion with \emph{cop}, and $\Sigma \subseteq M$ a cross-cut of $F$ 	relative to $H = 	H_{u, \alpha}$, then given a transversely convex tube $\mathcal{T}_{a}$, with \emph{cop}, about 			$F(\Sigma)$, where $u^{*} \circ F$ has no 	critical point in $U_{a}$, there exists $d > 0$ so that $\mathcal{T}_{a+d} 		\supseteq \mathcal{T}_{a}$ is a transversely convex tube, with \emph{cop}, about $F(\Sigma)$, where $u^{*} \circ F$ has no 	critical point in a larger neighborhood $U_{a+d}$ of $\Sigma$. 
\end{lemma}

\begin{proof}[Sketch of proof]
	Since each boundary of the transversely convex tube $\mathcal{T}_{a}$ is the image of a cross-cut we can use Lemma 		\ref{LS3.3:exsttrncvxtube} to construct a transversely convex tube about each boundary. Then it is an easy exercise to glue 		these two transversely convex tubes about each boundary to the tube $\mathcal{T}_{a}$ to get the required extension. Finally, 	the \emph{cop} of $F$ implies that the extended tube has also \emph{cop}.
\end{proof}

\section{Proof of Main Theorem}

Here, we give the proof of the main theorem. We first show that the conclusion holds locally and then use standard differential topology arguments to get the global result. The two papers that are cited, \cite{bS09} and \cite{bS12}, were written by Solomon. Unless stated otherwise, we assume that $n \geq 3$ throughout this section.

\subsection{Ellipsoid and Centrix}

In this subsection, we classify the support function of an origin centered ellipsoid as a function whose square is the solution of some differential equation. We define the centrix of an ovaloid and characterize central symmetry in terms of the centrix.

\subsubsection{Origin centered ellipsoid}

\begin{lemma} \label{LS5.1:orgcntellipsoid}
	Let $h \in C^{2}(\mathbb{S}^{n-2})$, $n \geq 3$, be a non vanishing function on $\mathbb{S}^{n-2}$, then
	\[
		\nabla_{\mathbb{S}^{n-2}} \bigl[ \, \Delta_{\mathbb{S}^{n-2}} h^{2} + 2 (n - 1)h^{2} \,\bigr](u)  = 0
	\]
	for each $u \in \mathbb{S}^{n-2}$ if and only if $h$ or $-h$ is the support function of an origin centered ellipsoid.
\end{lemma}

\begin{proof}
	$(\Rightarrow)$: Define the linear differential operator $L$ as $L = \Delta_{\mathbb{S}^{n-2}} + 2 (n-1)$, then given $h 		\in C^{2}(\mathbb{S}^{n-2})$ the equality $\nabla_{\mathbb{S}^{n-2}} (L h^{2}) \equiv 0$ implies $Lh^{2} \equiv c$, 		for some constant $c \in \mathbb{R}$. The function $h^{2}$ can be written as $h^{2} = h_{1} + h_{2}$,  where $Lh_{1} 		\equiv c$, $h_{1}$ is the particular solution given by 
	\[
		h_{1}(x) = \frac{c}{2(n-1)} \sum_{i=1}^{n-1} x_{i}^{2} \equiv \frac{c}{2(n-1)} \text{ for $x \in \mathbb{S}				^{n-2}$ and}
	\]
	$Lh_{2} = 0$. The homogeneous solution $h_{2} \colon \mathbb{S}^{n-2} \to \mathbb{R}$, $h_{2} \in C^{2}(S^{n-2})$ 	satisfies
	\begin{align*}
		\Delta_{\mathbb{S}^{n-2}} h_{2} + 2 (n - 1) h_{2} &= 0\\
		- \Delta_{\mathbb{S}^{n-2}} h_{2}     		     &= 2 (n - 1)h_{2}
	\end{align*}
	and using the fact that the spherical harmonics are the eigenspaces of the Laplace-Beltrami operator, $\Delta_{\mathbb{S}		^{n-2}}$ \cite[p.74]{hG96}, we can conclude that $h_{2}$ is a spherical harmonic of degree $2$ in $\mathbb{R}^{n-1}$. 		Therefore, $h_{2}$ must be of the form $h_{2}(x) = x^{\tng}Ax$, $x \in \mathbb{S}^{n-2}$ for some trace-free symmetric 	matrix $A$. Then the function $h^{2}$ equals
	\begin{align*}
		h^{2}(x) &= h_{1}(x) + h_{2}(x) = \frac{c}{2(n-1)} \sum_{i=1}^{n-1} x_{i}^{2} + x^{\tng}Ax\\
			      &= x^{\tng} \left( \frac{c}{2(n-1)} I_{n-1} + A \right) x > 0 \quad \text{for all $x \in \mathbb{S}^{n-2}$ 					   and hence}\\
			   B&:= \frac{c}{2(n-1)} I_{n-1} + A
	\end{align*}
	is a positive definite symmetric matrix. The set $E :=\bigl\{x \in \mathbb{R}^{n-1} \colon x^{\tng}B^{-1}x = 1\bigr\}$ is 		therefore an origin centered ellipsoid. There exists $P \in \mathrm{O}(n-1)$ so that
	\[
		PBP^{-1} = \begin{pmatrix}
					\lambda_{1}^{2} & 0        & 0\\
					0                           & \ddots & 0\\
					0			  & 0         &\lambda_{n-1}^{2}
				\end{pmatrix} \quad B^{-1} = P^{-1} \begin{pmatrix}
												\lambda_{1}^{-2} & 0         & 0\\
												0                             & \ddots & 0\\ 
												0			    & 0         & \lambda_{n-1}^{-2}
										\end{pmatrix} P
	\]
	where $\lambda_{i}^{2} > 0$ for each $i = 1, \dotsc, n-1$ and the set $E$ equals
	\begin{align*}
		E &= \bigl\{ x \in \mathbb{R}^{n-1} \colon x^{\tng} P^{-1} \diag \bigl[\, \lambda_{1}^{-2}, \dotsc, \lambda_{n-1}				 ^{-2} \,\bigr] Px = 1\bigr\}\\
		   &= \bigl\{ x \in \mathbb{R}^{n-1} \colon (Px)^{\tng} \diag \bigl[\, \lambda_{1}^{-2}, \dotsc, \lambda_{n-1}					 ^{-2} \,\bigr] Px = 1\bigr\}\\
		   &= P^{-1} \bigl\{ x \in \mathbb{R}^{n-1} \colon x^{\tng} \diag \bigl[\, \lambda_{1}^{-2}, \dotsc, \lambda_{n-1}				 ^{-2} \,\bigr] x = 1\bigr\}\\
		   &= P^{-1} \left\{ x \in \mathbb{R}^{n-1} \colon \sum_{i=1}^{n-1} \frac{x_{i}^{2}}{\lambda_{i}^{2}} = 1 \right				\}\\
	      PE &= \left\{ x \in \mathbb{R}^{n-1} \colon \sum_{i=1}^{n-1} \frac{x_{i}^{2}}{\lambda_{i}^{2}} = 1 \right\}.
	\end{align*} 
	The support function of $PE$ is calculated in Example \ref{ES2.5:sptfnctellipsoid} as
	\[
		h_{PE}(x) = \left\{ \sum_{i=1}^{n-1} \lambda_{i}^{2}x_{i}^{2} \right\}
	\]
	and its square satisfies for all $x \in \mathbb{R}^{n-1}$
	\begin{align*}
		\bigl[ h_{PE}(x) \bigr]^{2} = \sum_{i=1}^{n-1} \lambda_{i}^{2} x_{i}^{2} &= (x^{\tng} P)B(P^{-1}x) = (P^{-1}								      									       x)^{\tng}B(P^{-1}x)\\
		\Rightarrow \bigl[ (h_{PE} \circ P)(x) \bigr]^{2} 						&= x^{\tng}Bx = h^{2}(x).				\end{align*}
	\[
		(h_{PE} \circ P)(x) = \max_{z \in PE} z \cdot Px = \max_{y \in E} Py \cdot Px = \max_{y \in E} y \cdot x = h_{E}			(x).
	\]
	So we can conclude that $h^{2}(x) = \bigl[ h_{E}(x) \bigr]^{2}$, $h_{E}$ is the support function of an origin centered 		ellipsoid and $h = h_{E}$ or $-h = h_{E}$, depending on whether the nonvanishing function $h$ is positive or negative.
	
	$(\Leftarrow)$: For any given $\lambda_{i} > 0$, $i = 1, \dotsc, n-1$, let
	\[
		E = \left\{ x \in \mathbb{R}^{n-1} \colon \sum_{i=1}^{n-1} \frac{x_{i}^{2}}{\lambda_{i}^{2}} = 1 \right\} =       			\bigl\{ x\in \mathbb{R}^{n-1} \colon x^{\tng} \diag \bigl[ \lambda_{1}^{-2}, \dotsc, \lambda_{n-1}^{-2} \bigr] x = 1 		\bigr\}
	\]
	be an origin centered ellipsoid with principal axes lying on the coordinate axes and $h_{E}$ the support function of $E$. 		We need to show that $\Delta_{\mathbb{S}^{n-2}} h_{E}^{2} + 2 (n-1)h_{E}^{2} \equiv c$ on $\mathbb{S}^{n-2}$ for 		some constant $c \in \mathbb{R}$. Since $h_{E}^{2}$ is positively homogeneous of degree $2$, using the expression of 		the Laplacian in polar coordinates we can compute for any $u \in \mathbb{S}^{n-2}$
	\begin{align*}
		\Delta_{\mathbb{S}^{n-2}} h_{E}^{2}(u) &= \Delta_{\mathbb{R}^{n-1}} h_{E}^{2}(u) - (n-2) \left. \frac{\partial}		{\partial r} \right\rvert_{r=1} h_{E}^{2}(ru) - \left. \frac{\partial^{2}}{\partial r^{2}} \right\rvert_{r=1} h_{E}^{2}			(ru)\\
		&= 2 \sum_{i=1}^{n-1} \lambda_{i}^{2} - 2 (n - 2) h_{E}^{2}(u) - 2 h_{E}^{2}(u)\\
		&= 2 \sum_{i=1}^{n-1} \lambda_{i}^{2} - 2 (n - 1) h_{E}^{2}(u)
	\end{align*}
	\[
		\Rightarrow \Delta_{\mathbb{S}^{n-2}} h_{E}^{2}(u) + 2 (n - 1) h_{E}^{2}(u) = 2 \sum_{i=1}^{n-1} \lambda_{i}			^{2}=: c > 0
	\]
	holds for all $u \in \mathbb{S}^{n-2}$. Thus we can conclude that 
	\[
		\nabla_{\mathbb{S}^{n-2}} \bigl[\, \Delta_{\mathbb{S}^{n-2}} h_{E}^{2} + 2(n-1) h_{E}^{2} \,\bigr] \equiv 0 			\quad \text{on $\mathbb{S}^{n-2}$.}
	\]
	A general ellipsoid $\widetilde{E}$ centered at the origin is of  the form $\widetilde{E} = P^{-1}E$, where $P \in 			\mathrm{O}(n-1)$ . According to the item \eqref{LS2.2:spfnpr5} of Lemma \ref{LS2.2:spfnpr}, the support function 			$h_{\widetilde{E}}$ satisfies
	\[
		h_{\widetilde{E}} = h_{P^{-1}E} = h_{E} \circ P.
	\]
	For each $u \in \mathbb{S}^{n-2}$, by applying the representation of $\Delta_{\mathbb{R}^{n-1}}$ in general coordinates  	to the particular case $(\mathbb{R}^{n-1}, P)$ we get 
	\[
		\Delta_{\mathbb{R}^{n-1}} (h_{E}^{2} \circ P)(u) = \Delta_{\mathbb{R}^{n-1}} h_{E}^{2}(Pu)
	\]
	and hence the spherical Laplacian of $h_{\widetilde{E}}^{2}$ equals
	\begin{align*}
		&\Delta_{\mathbb{S}^{n-2}} h_{\widetilde{E}}^{2}(u) = \Delta_{\mathbb{S}^{n-2}} (h_{E}^{2} \circ P)(u)\\
		&= \Delta_{\mathbb{R}^{n-1}}(h_{E}^{2} \circ P)(u) - (n-2) \left. \frac{\partial}{\partial r} \right\rvert_{r=1} 			h_{E}^{2}\bigl[rP(u)] - \left. \frac{\partial^{2}}{\partial r^{2}} \right\rvert_{r=1} h_{E}^{2}\bigl[rP(u)\bigr]\\
		&= \Delta_{\mathbb{R}^{n-1}} h_{E}^{2}(Pu) - 2(n-1) h_{E}^{2} (Pu) = \Delta_{\mathbb{S}^{n-2}} h_{E}^{2}			(Pu)
	\end{align*}
	Therefore, for each $u \in \mathbb{S}^{n-2}$ we can compute
	\begin{align*}
		\Delta_{\mathbb{S}^{n-2}} h_{\widetilde{E}}^{2}(u) + 2 (n-1) h_{\widetilde{E}}^{2}(u) &= \Delta_{\mathbb{S}													^{n-2}} h_{E}^{2} (Pu) + 2(n-1) h_{E}^{2}(Pu)\\
																		    &= 2 \sum_{i=1}^{n-1} 																				  \lambda_{i}^{2}.
	\end{align*}
	So we can conclude that $\nabla_{\mathbb{S}^{n-2}} \bigl[\, \Delta_{\mathbb{S}^{n-2}} h_{\widetilde{E}}^{2} + 2(n-1) 	h_{\widetilde{E}}^{2} \,\bigr] \equiv 0$ on $\mathbb{S}^{n-2}$.
\end{proof}

\subsubsection{The centrix}

\begin{definition}
	The centrix of an ovaloid $\mathcal{O}^{n-1} \subseteq \mathbb{R}^{n}$ is defined as the centrix of the strictly convex 		body $K$ that it bounds. Namely, following Definition \ref{DS2.2:centrix}, if $h$ is the support function of $K$, and 		hence of $\mathcal{O}$ by definition, then the centrix of $\mathcal{O}$ is defined as
	\begin{align*}
		&c \colon \mathbb{R}^{n} \setminus \{0\} \to \mathbb{R}^{n}\\
		&c(u)  = \frac{\nabla_{\mathbb{R}^{n}} h(u) + \nabla_{\mathbb{R}^{n}} h(-u)}{2} = \frac{\nabla_{\mathbb{R}				      ^{n}} h(\bar{u}) + \nabla_{\mathbb{R}^{n}} h(-\bar{u})}{2}\\
		&\phantom{c(u)} = \frac{\Gamma (\bar{u}) + \Gamma (- \bar{u})}{2}	
	\end{align*}
	where $\Gamma \colon \mathbb{S}^{n-1} \to \mathcal{O}$ is the support parameterization of $\mathcal{O}$, $\bar{u} = 		\frac{u}{\lvert u \rvert}$, and the second the equality follows because $\nabla_{\mathbb{R}^{n}} h$ is positively 			homogeneous of degree zero.
\end{definition}

\begin{definition}
	Given the support parameterization $\Gamma$ of an ovaloid $\mathcal{O}$, we call the maps
	\begin{align*}
		\Gamma^{+}(u) &= \frac{\Gamma(u) + \Gamma(- u)}{2}\\
		\Gamma^{-}(u) &= \frac{\Gamma(u) - \Gamma(- u)}{2}
	\end{align*}
	the even and odd parts of $\Gamma$, respectively.
\end{definition}

\begin{lemma} \label{LS5.1:oddsptparam}
	The centrix $c \colon \mathbb{S}^{n-1} \to \mathbb{R}^{n}$, when restricted to $\mathbb{S}^{n-1}$, coincides with $		\Gamma^{+}$. The centrix is constant if and only if $\mathcal{O}$ has central symmetry. In that case, $\Gamma^{-}$ 		support parameterizes the origin centered ovaloid $\mathcal{O} - c_{0}$, where $c \equiv c_{0}$.
\end{lemma}

\begin{proof}
	The first claim follows directly from the definition of $c$ and $\Gamma^{+}$. To prove the second claim, assume first that 	$c = c_{0}$ is constant and consider the map
	\begin{align*}
		R_{c_{0}} \colon \mathbb{R}^{n} &\to \mathbb{R}^{n}\\
		x					             &\mapsto 2c_{0} - x 
	\end{align*}
	which is the reflection through $c_{0}$. Given $u \in \mathbb{S}^{n-1}$
	\[
		R_{c_{0}} \bigl[\, \Gamma(u) \,\bigr] = 2 c_{0} - \Gamma(u) = 2 \frac{\Gamma(u) + \Gamma(-u)}{2} - \Gamma(u) 			=  \Gamma(-u).
	\]
	Since $u \in \mathbb{S}^{n-1}$ is arbitrary, we can conclude that $R_{c_{0}}(\mathcal{O}) = \mathcal{O}$ and $			\mathcal{O}$ has central symmetry with center $c_{0}$. On the other hand, if $\mathcal{O}$ is central with center 			$c_{0}$, then for every $u \in \mathbb{S}^{n-1}$
	\[
		c_{0} = \frac{\Gamma(u) + R_{c_{0}} \bigl[\, \Gamma(u) \,\bigr]}{2} = \frac{\Gamma(u) + \Gamma(-u)}{2} = c(u)
	\]
	end hence the centrix $c$ is constant. In order to show the last claim assume that $\mathcal{O}$ has the center of symmetry 	$c_{0}$ and
	\begin{align*}
		&h_{\mathcal{O}} \text{ is the support function of $\mathcal{O}$,}\\
		&\Gamma_{\mathcal{O}} \text{ is the support parameterization of $\mathcal{O}$,}\\
		&h_{\mathcal{O} - c_{0}} \text{ is the support function of $\mathcal{O} - c_{0}$,}\\
		&\Gamma_{\mathcal{O} - c_{0}} \text{ is the support parameterization of $\mathcal{O} - c_{0}$.}
	\end{align*}
	According to the item \eqref{LS2.2:spfnpr6}of Lemma \ref{LS2.2:spfnpr}, for every $u \in \mathbb{S}^{n-1}$, 			the support parameterization of $h_{\mathcal{O} - c_{0}}$ satisfies $h_{\mathcal{O}- c_{0}}(u) = h_{\mathcal{O}}(u) - 	u \cdot c_{0}$ and hence
	\begin{align*}
		\Gamma_{\mathcal{O} - c_{0}}(u) &= \nabla_{\mathbb{R}^{n}} h_{\mathcal{O} - c_{0}}(u) = 													  \nabla_{\mathbb{R}^{n}} h_{\mathcal{O}(u)} - c_{0}\\
								     &= \Gamma_{\mathcal{O}}(u) - c_{0} = \Gamma_{\mathcal{O}}(u) - 											   \frac{\Gamma_{\mathcal{O}}(u) + \Gamma_{\mathcal{O}}(-u)}{2}\\
								     &= \frac{\Gamma_{\mathcal{O}}(u) - \Gamma_{\mathcal{O}}(-u)}{2} = 										   \Gamma_{\mathcal{O}}^{-}(u).
	\end{align*}
	So we can conclude that the odd part $\Gamma_{\mathcal{O}}^{-}$ of $\Gamma_{\mathcal{O}}$ support parameterizes  	the origin centered ovaloid $\mathcal{O} - c_{0}$.
\end{proof}

\subsection{Splitting}

In this subsection, our goal is to show that the rectification $\mathcal{T}^{-}$ of a transversely convex tube $\mathcal{T}$ with \emph{cop} is either cylindrical or quadric.

\subsubsection{The support map and the height function} 

Our goal is to reparameterize the tube $\mathcal{T}$ with the use of the support map. Using this new parameterization we will  construct the height function $z$ that lets us parameterize nearly horizontal ovaloid cross-sections as graphs over horizontal ovaloid cross-sections.

\begin{definition}
	Given any $\epsilon \in \mathbb{R}$, $\tau \in \mathbb{S}^{n-2}$, the $\epsilon$-tilted hyperplane is given by
	\[
		P_{\tau, z_{0}}(\epsilon) = \bigl\{ (p, z) \in \mathbb{R}^{n-1} \times \mathbb{R} \colon z = \epsilon (p \cdot \tau) + 		z_{0} \bigr\}
	\]
	where $\tau$ is called the tilt direction, $z_{0}$ is the $e_{n}$-intercept, and $\epsilon$ is the slope of this hyperplane 		along $\tau$ direction
\end{definition}

\begin{definition}[The support map of $\mathcal{T}$]
	A transversely convex tube $\mathcal{T}$ in \emph{standard position} can be reparameterized as follows:
	\begin{align*}
		&Y \colon \mathbb{S}^{n-2} \times I \to \mathcal{T}\\
		&Y(u, z) = \bigl(\, \Gamma(u, z), z \,\bigr)
	\end{align*}
	where $\Gamma \colon \mathbb{S}^{n-2} \times I \to \mathbb{R}^{n-1}$ is a smooth map such that for every $z \in I$, the 	map $\Gamma( \cdot, z) $ support parameterizes the ovaloid $\mathcal{O}(z) - (0, z)$, where $\mathcal{O}(z) := 			\mathcal{T} \cap H_{e_{n}, z}$ is the ovaloid with center of mass at $(c(z), z)$. The smooth map $\Gamma$ is called the 		support map of $\mathcal{T}$. In particular, the construction of the map $\Gamma$ is carried out as follows:
	
	Define the maps
	\begin{align*}
		\nu \colon \mathcal{T} \to \mathbb{S}^{n-2}&\\
		p = (p', p_{n})               &\mapsto \nu(p) = \text{outer unit normal in $\mathbb{R}^{n-1}$ to the ovaloid}\\
						  &\phantom{\mapsto \nu(p) = {}} \mathcal{O}(p_{n}) - (c(p_{n}), p_{n}) \subseteq								     \mathbb{R}^{n-1} \times \{0\} \simeq \mathbb{R}^{n-1}\\
						  &\phantom{\mapsto \nu(p) = {}} \text{at } (p' - c(p_{n}), 0) \simeq p' - c(p_{n})
	\end{align*}
	and $\bar{\nu} \colon \mathcal{T} \to \mathbb{S}^{n-2} \times I$, $p \mapsto (\nu(p), p_{n})$. 
	
	\begin{figure}[h]
		\centering
		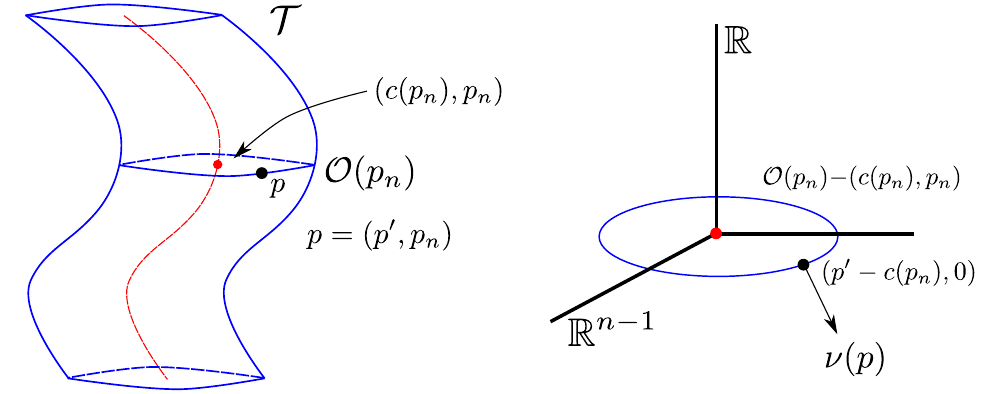
		\caption{Outer unit normal field $\nu$}
	\end{figure}
	
	Using the parameterization $X$ of $\mathcal{T}$, which is defined in the equation \eqref{E:trcvxtstr}, and Remark 			\ref{RS3.2:outrunitnorfld} we can write
	
	\begin{align*}
		(\bar{\nu} \circ X)(u, z) &= \bar{\nu} \bigl[\, X(u, z) \,\bigr] = \bar{\nu} \bigl(\, c(z) + \alpha(u, z), z \,\bigr)\\
						   &= \left ( \parbox{4.5cm}{\centering{outer unit normal in $\mathbb{R}^{n-1}$ to the 									 ovaloid $\mathcal{O}(z) - (c(z), z)$ at $(\alpha(u, z), 0)$}}, \quad z \right)\\
						   &= \left( \frac{\star \Bigl[\, \bigwedge \nolimits_{n-2} \partial_{1} \alpha(u, z) \,\Bigr] 									\Bigl( \frac{\partial}{\partial u_{1}} \wedge \cdots \wedge \frac{\partial}{\partial 									u_{n-2}} \Bigr)}{\Bigl\lvert  \star \Bigl[\, \bigwedge \nolimits_{n-2} \partial_{1} 									\alpha(u, z) \,\Bigr] \Bigl( \frac{\partial}{\partial u_{1}} \wedge \cdots \wedge 									\frac{\partial}{\partial u_{n-2}} \Bigr) \Bigr\rvert}, \quad z \right)
	\end{align*}
	The map
	\[
		(u, z) \mapsto \frac{\star \Bigl[\, \bigwedge \nolimits_{n-2} \partial_{1} \alpha(u, z) \,\Bigr] \Bigl( \frac{\partial}			{\partial u_{1}} \wedge \cdots \wedge \frac{\partial}{\partial u_{n-2}} \Bigr)}{\Bigl\lvert  \star \Bigl[\, \bigwedge 			\nolimits_{n-2} \partial_{1} \alpha(u, z) \,\Bigr] \Bigl( \frac{\partial}{\partial u_{1}} \wedge \cdots \wedge 				\frac{\partial}{\partial u_{n-2}} \Bigr) \Bigr\rvert}
	\]
	is smooth because $\star$ is an $(n-2)$-linear map and $\partial_{1} \alpha$ is smooth. The map $\bar{\nu} \circ X$ is 		bijective because for every $z \in I$ the ovaloid $\alpha(\mathbb{S}^{n-2}, z)$ has a globally defined unit normal field, and 	the differential has a maximal rank because for every $z \in I$
	\[
		u \mapsto \frac{\star \Bigl[\, \bigwedge \nolimits_{n-2} \partial_{1} \alpha(u, z) \,\Bigr] \Bigl( \frac{\partial}{\partial 			u_{1}} \wedge \cdots \wedge \frac{\partial}{\partial u_{n-2}} \Bigr)}{\Bigl\lvert  \star \Bigl[\, \bigwedge 					\nolimits_{n-2} \partial_{1} \alpha(u, z) \,\Bigr] \Bigl( \frac{\partial}{\partial u_{1}} \wedge \cdots \wedge 				\frac{\partial}{\partial u_{n-2}} \Bigr) \Bigr\rvert}	
	\]
	is a diffemorphism, which follows from a simple observation in the case $n = 3$ (oval), and from Hadamard's theorem 		\cite[p.41]{jH98, sKkN69} in the case $n \geq 4$. Using the inverse function theorem we can conclude that $\bar{\nu} \circ 	X$ is a diffeomorphism and we can define the reparameterization $Y$ of $\mathcal{T}$ as follows:
	\begin{align*}
		&Y \colon \mathbb{S}^{n-2} \times I \to \mathcal{T}\\
		&Y = X \circ (\bar{\nu} \circ X)^{-1}
	\end{align*}
	$Y(\bar{u}, z) = X \bigl[ (\bar{\nu} \circ X)^{-1} (\bar{u}, z) \bigr] = X(u, z) = \bigl( c(z) + \alpha(u, z), z \bigr)$ where
	\[
		\bar{u} = \frac{\star \Bigl[\, \bigwedge \nolimits_{n-2} \partial_{1} \alpha(u, z) \,\Bigr] \Bigl( \frac{\partial}{\partial 			u_{1}} \wedge \cdots \wedge \frac{\partial}{\partial u_{n-2}} \Bigr)}{\Bigl\lvert  \star \Bigl[\, \bigwedge 					\nolimits_{n-2} \partial_{1} \alpha(u, z) \,\Bigr] \Bigl( \frac{\partial}{\partial u_{1}} \wedge \cdots \wedge 				\frac{\partial}{\partial u_{n-2}} \Bigr) \Bigr\rvert}
	\]
	is the outer unit normal in $\mathbb{R}^{n-1} \simeq \mathbb{R}^{n-1} \times \{0\}$ to  the ovaloid $\mathcal{O}(z) - 		(c(z), z)$ at $(\alpha(u, z), 0)$. Or equivalently, $\bar{u}$ is the outer unit normal to the ovaloid $\mathcal{O}(z) - (0, z)$ 		at the point $(c(z) + \alpha(u, z), z)$.Therefore, if we define the map $\Gamma \colon \mathbb{S}^{n-2} \times I \to 			\mathbb{R}^{n-1}$ as
	\[
	 	\Gamma(\bar{u}, z) := c(z) + \alpha(u, z) 	
	\]
	 we can conclude that $\Gamma(\cdot, z)$ support parameterizes the ovaloid $ \mathcal{O}(z) -(0, z)$ for each $z \in I$ and 
	\[
		Y(\bar{u}, z) = \bigl(\, \Gamma(\bar{u}, z), z \,\bigr)
	\]
	for each $(\bar{u}, z) \in \mathbb{S}^{n-2} \times I$.
\end{definition}

\begin{definition}[The height function $z$]
	For each $z_{0} \in I$, $\epsilon \in \mathbb{R}$, $\tau \in \mathbb{S}^{n-2}$, we define the cross-section
	\begin{align*}
		\overline{\mathcal{O}}_{\tau}(z_{0}, \epsilon) &= \mathcal{T} \cap P_{\tau, z_{0}}(\epsilon)\\
	\intertext{and its image under the projection $(x_{1}, \dotsc, x_{n-1}, x_{n}) \mapsto (x_{1}, \dotsc, x_{n-1})$}
		\mathcal{O}_{\tau}(z_{0}, \epsilon) &\subseteq \mathbb{R}^{n-1} \times \{0\} \simeq \mathbb{R}^{n-1}.
	\end{align*}
	The horizontal $(\epsilon = 0)$ cross-section, for each $\tau \in \mathbb{S}^{n-2}$, satisfies
	\[
		\mathcal{O}(z_{0}) = \overline{\mathcal{O}}_{\tau}(z_{0}, 0) \equiv \mathcal{O}_{\tau}(z_{0}, 0).
	\]
\end{definition}

Because of the transverse convexity of $\mathcal{T}$ and the proof of Lemma \ref{LS3.3:exsttrncvxtube}, given any 		$z_{0} \in I$, there exists $\delta_{0} > 0$ so that for every $\tau \in \mathbb{S}^{n-2}$, $\lvert \epsilon \rvert < \delta_{0}$, $\overline{\mathcal{O}}_{\tau}(z_{0}, \epsilon)$ is again an ovaloid. With this remark, our objective now is to show that for any given $z_{0} \in I$ there exist $\delta> 0$ and a height function $z(\epsilon, u)$ so that for every $\lvert \epsilon \rvert < \delta$,  $z(\epsilon, \cdot)$ lets us parameterize the ovaloid $\overline{\mathcal{O}}_{\tau}(z_{0}, \epsilon)$ by the map
\begin{align*}
	\mathbb{S}^{n-2} &\to \mathbb{R}^{n}\\
	u 			    &\mapsto \bigl(\, \Gamma \bigl( u, z(\epsilon, u) \bigr), z(\epsilon, u) \,\bigr).
\end{align*}

We define the smooth map $G \colon \mathbb{R} \times \mathbb{S}^{n-2} \times I \to \mathbb{R}$ as
\[
	G(\epsilon, u, z) = z - z_{0} - \epsilon \tau \cdot \Gamma(u, z).
\]
Consider $\lvert \epsilon \rvert < \delta_{0}$, and observe that
\begin{equation} \label{ES5.2:paramrmk}
	\begin{aligned}
		&(u, z) \in Y^{-1} \bigl[\, \overline{\mathcal{O}}_{\tau}(z_{0}, \epsilon) \,\bigr] \iff Y(u, z) \in						    \overline{\mathcal{O}}_{\tau}(z_{0}, \epsilon)\\
		& \iff \bigl(\, \Gamma(u, z), z \,\bigr) \in \mathcal{T} \cap P_{\tau, z_{0}}(\epsilon)\\
		& \iff \bigl(\, \Gamma(u, z), z \,\bigr) \in P_{\tau, z_{0}}(\epsilon)\\
		& \iff z = z_{0} + \epsilon \bigl[ \tau \cdot \Gamma(u, z) \bigr]\\
		& \iff G(\epsilon, u, z) = 0. 
	\end{aligned}
\end{equation}

Note that $G(0, u, z_{0}) = 0$ for each $u \in \mathbb{S}^{n-2}$ and 
\[
	\frac{\partial G}{\partial z}(\epsilon, u, z) = 1 - \epsilon \tau \cdot \frac{\partial \Gamma}{\partial z}(u, z) \quad \text{for all 	$z  \in I$ implies} \quad  \frac{\partial G}{\partial z} (0, u, z_{0}) = 1.
\]

\begin{figure}[h]
	\centering
	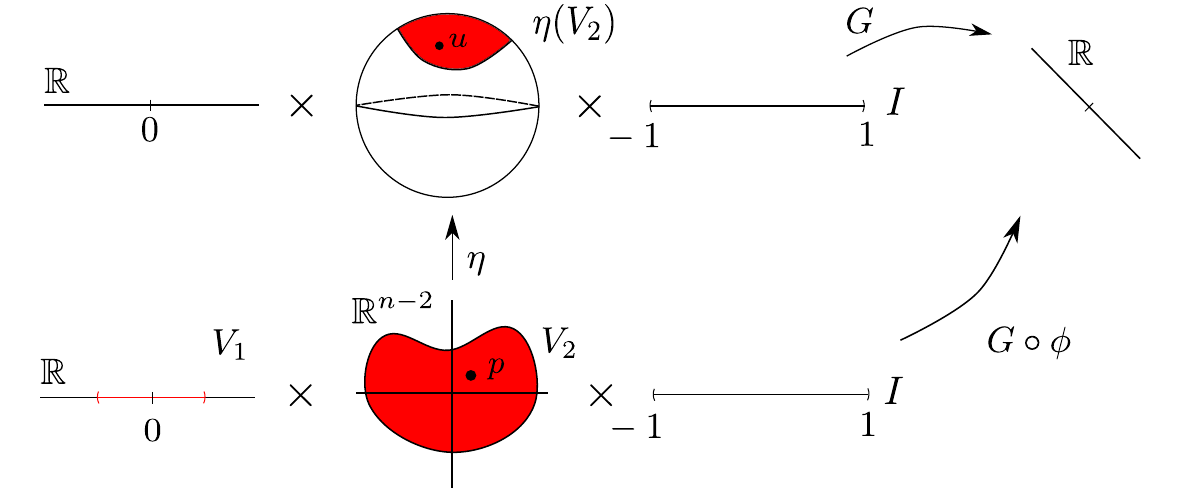
	\caption{Implicit function $G \circ \phi$}
\end{figure}

Fix a local parameterization $(V, \eta)$ of $\mathbb{S}^{n-2}$ and let $\phi = \mathrm{id}_{\mathbb{R}} \times \eta \times \mathrm{id}_{I}$, then the map $G \circ \phi$ satisfies 
\begin{align*}
	&G \circ \phi \colon \mathbb{R} \times V \times I \to \mathbb{R}\\
	&G \circ \phi ( \epsilon, v, z) = z - z_{0} - \epsilon \tau \cdot \Gamma \bigl( \eta(v), z\bigr)\\
	&G \circ \phi (0, v, z_{0}) = 0 \quad \text{for all $v \in V$ and }\\
	&\frac{\partial (G \circ \phi)}{\partial z}(0, v, z_{0}) = 1 \quad \text{for all $v \in V$.}
\end{align*}
Therefore, the implicit function theorem applied to $G \circ \phi$ yields the existence of an open neighborhood of  $(0, v) \in V_{1} \times V_{2} \subseteq \mathbb{R} \times \mathbb{R}^{n-2} = \mathbb{R}^{n-1}$, where $V_{1} \subseteq (-\delta_{0}, \delta_{0})$ is an open interval about $0$, $V_{2}$ is an open subset of $V$, and a unique smooth function $g \colon V_{1} \times V_{2} \to \mathbb{R}$ such that
\begin{align*}
	g(0, v) &= z_{0}\\
	G \circ \phi \bigl[ \epsilon, v, g(\epsilon, v) \bigr] 			&= 0 \quad \text{for all } \epsilon \in V_{1} \text{ and } v \in 												      V_{2}\\
	\Rightarrow G \circ \phi \bigl[ \epsilon, v, g(\epsilon, v) \bigr] &= g( \epsilon, v) - z_{0} - \epsilon \tau \cdot \Gamma 													       \bigl( \eta(v), g(\epsilon, v) \bigr) = 0\\
	G \circ \phi \bigl[ 0, v, g(0, v) \bigr] 					&= g(0, v) - z_{0} = 0\\
	g(0, v) 										&= z_{0} \quad \text{for all $v \in V_{2}$.}
\end{align*}
So there exists a unique smooth function $\bar{g} \colon V_{1} \times \eta(V_{2}) \to \mathbb{R}$ defined as
\begin{align*}
	\bar{g}(\epsilon, u) &= g \circ (\mathrm{id}_{\mathbb{R}} \times \eta^{-1})(\epsilon, u) \quad \text{and satisfies}\\
	\bar{g}(0, u) &= z_{0} \quad \text{for every $u \in \eta(V_{2}) = U_{2}$ and} \\
	G\bigl(\epsilon, u, \bar{g}(\epsilon, u)\bigr) &= 0 \quad \text{for every $\epsilon \in V_{1}$, $u \in \eta(V_{2}) = U_{2}$.}
\end{align*}
Since $\mathbb{S}^{n-2}$ can be covered by finitely many overlapping open sets $U_{2}^{k}$, $k = 1, \dotsc, M$, by taking the intersection $\cap_{k=1}^{M}V_{1}^{k}$ , using the uniqueness of the smooth function
\[
	\bar{g}_{k} \colon V_{1}^{k} \times U_{2}^{k} \to \mathbb{R} \quad k = 1, \dotsc, M
\]
and by letting $(-\delta, \delta) = \cap_{k=1}^{M} V_{1}^{k}$ we can construct a smooth function
\begin{equation} \label{ES5.2:impfctdom}
	z \colon (-\delta, \delta) \times \mathbb{S}^{n-2} \to \mathbb{R}
\end{equation}
that satisfies
\begin{gather}
	z(0, u) = z_{0} \quad \text{for all $u \in \mathbb{S}^{n-2}$}\notag\\
	G(\epsilon, u, z) = 0 \iff z = z(\epsilon, u) \quad \text{for all $\lvert \epsilon \rvert < \delta$ and $u \in \mathbb{S}^{n-2}$}		\notag\\
	z(\epsilon, u) = z_{0} + \epsilon \tau \cdot \Gamma \bigl( u, z(\epsilon, u) \bigr) \quad \text{for all $\lvert \epsilon \rvert <		\delta$ and $u \in \mathbb{S}^{n-2}$.} \label{ES5.2:impfct} 
\end{gather}
Using the observation in \eqref{ES5.2:paramrmk} we get for every $\lvert \epsilon \rvert < \delta$,
\begin{align*}
	(u, z) \in Y^{-1} \bigl[\, \overline{\mathcal{O}}_{\tau}(z_{0}, \epsilon) \,\bigr] &\iff G(\epsilon, u, z) = 0\\
														        &\iff z = z(\epsilon, u)
\end{align*}
and hence the ovaloid $\overline{\mathcal{O}}_{\tau}(z_{0}, \epsilon)$ can be parameterized as
\begin{align*}
	\mathbb{S}^{n-2} &\to \mathbb{R}^{n}\\
	u &\mapsto Y\bigl( u, z(\epsilon, u) \bigr) = \bigl(\, \Gamma \bigl(u, z(\epsilon, u) \bigr), z(\epsilon, u) \,\bigr).
\end{align*}
Therefore, the projected ovaloid $\mathcal{O}_{\tau}(z_{0}, \epsilon)$, for each $\tau \in \mathbb{S}^{n-2}$ and $\lvert \epsilon \rvert < \delta$, can be parameterized as
\begin{align*}
	\mathbb{S}^{n-2} &\to \mathbb{R}^{n-1}\\
	u 			    &\mapsto \Gamma \bigl( u, z(\epsilon, u) \bigr)
\end{align*}
Using the equation \eqref{ES5.2:impfct} we can compute the derivative for each $\lvert \epsilon \rvert < \delta$ and $u \in \mathbb{S}^{n-2}$ as
\begin{align}
		\frac{\partial z}{\partial \epsilon}(\epsilon, u) &= \tau \cdot \Gamma \bigl( u, z(\epsilon, u) \bigr) + \epsilon \left( \tau 		\cdot \frac{\partial}{\partial \epsilon} \Gamma \bigl(u, z(\epsilon, u) \bigr) \right) \notag\\
		\intertext{if we let $\epsilon = 0$ the derivative simplifies to}
		\frac{\partial z}{\partial \epsilon}(0, u) &= \tau \cdot \Gamma \bigl( u, z(0, u) \bigr) = \tau \cdot \Gamma(u, z_{0}) 			\quad \text{for each $u \in \mathbb{S}^{n-2}$.} \label{ES5.2:drvofimpfnc}
\end{align}

\subsubsection{The support reparameterizing map $\theta_{\epsilon}$} 

Using a diffeomorphism of the sphere $\mathbb{S}^{n-2}$ we will obtain the support parameterization of the projected ovaloid $\mathcal{O}_{\tau}(z, \epsilon)$ for each $z \in I$ and small enough $\lvert \epsilon \rvert$.

\begin{definition}
	Given $\tau \in \mathbb{S}^{n-2}$, we define the map $\tau^{\tng} \colon \mathbb{S}^{n-2} \to \mathbb{R}^{n-1}$ as
	\begin{equation} \label{DS5.2:tautng}
		\begin{aligned}
			u \in \mathbb{S}^{n-2} \mapsto \tau^{\tng}(u) &= \text{orthogonal projection of $\tau$ onto $\tng_{u} 					\mathbb{S}^{n-2}$}\\
											     &= \tau - (\tau \cdot u)u.
		\end{aligned}
	\end{equation}
\end{definition}

\begin{proposition} \label{PS5.2:sptreparam}
	Given $z_{0} \in I$ and $\tau \in \mathbb{S}^{n-2}$ there exist $\delta > 0$ and a differentiable $1$-parameter family of 		diffeomorphisms
	\[
		\theta_{\epsilon} \colon \mathbb{S}^{n-2} \to \mathbb{S}^{n-2}, \quad  - \delta < \epsilon < \delta
	\]
	such that given the map
	\begin{align*}
		\Gamma_{\epsilon} \colon \mathbb{S}^{n-2}            &\to \mathbb{R}^{n-1}\\
		u 								   	     &\mapsto \Gamma \bigl( u, z(\epsilon, u) \bigr), \text{the 														composition}\\
		\Gamma_{\epsilon} \bigl( \theta_{\epsilon}(u) \bigr) &= \Gamma \bigl(\, \theta_{\epsilon}(u), z\bigl(\epsilon, 				\theta_{\epsilon}(u)\bigr) \,\bigr)
	\end{align*}
	support parameterizes $\mathcal{O}_{\tau}(z_{0}, \epsilon)$ for each $\lvert \epsilon \rvert < \delta$. The initial map $		\theta_{0}$ is the identity map on $\mathbb{S}^{n-2}$, with the initial $\epsilon$-derivative given by
	\begin{align*}
		\left. \frac{\partial \theta_{\epsilon}}{\partial \epsilon} (u) \right\rvert_{\epsilon = 0} &= \bigl[\, u \cdot 					(\partial_{2}\Gamma)(u, z_{0})(1) \,\bigr]  \tau^{\tng}(u)\\
		&= \left( u \cdot \frac{\partial \Gamma}{\partial z}(u, z_{0}) \right)  \tau^{\tng}(u).
	\end{align*} 
\end{proposition}

\begin{proof}
	Let $(-\delta, \delta)$ be the interval obtained in Definition \ref{ES5.2:impfctdom} of the height function $z$, then for 			each $\lvert \epsilon \rvert  < \delta$ the map
	\begin{align*}
		\mathbb{S}^{n-2} &\to \mathbb{R}^{n-1}\\
		u 			    &\mapsto \Gamma \bigl( u, z(\epsilon, u) \bigr)
	\end{align*}
	parameterizes $\mathcal{O}_{\tau}(z_{0}, \epsilon)$, and
	\begin{align*}
		\psi \colon \mathbb{S}^{n-2} \times (-\delta, \delta) &\to \mathbb{R}^{n-1}\\
		(u, \epsilon) 							   &\mapsto \Gamma \bigl( u, z(\epsilon, u) \bigr)
	\end{align*}
	is smooth. Define the smooth map
	\begin{align*}
	 	\nu \colon \mathbb{S}^{n-2} \times (- \delta, \delta) &\to \mathbb{S}^{n-2} \times (-\delta, \delta)\\
		(u, \epsilon) &\mapsto \left( \frac{\star \Bigl(\bigwedge \nolimits_{n-2} \partial_{1} \psi(u, \epsilon)\Bigr) 				\Bigl( \frac{\partial}{\partial u_{1}} \wedge \cdots \wedge \frac{\partial}{\partial u_{n-2}} \Bigr)}{\Bigl\lvert \star 			\Bigl(\bigwedge \nolimits_{n-2} \partial_{1} \psi(u, \epsilon)\Bigr) \Bigl( \frac{\partial}{\partial u_{1}} \wedge 			\cdots \wedge \frac{\partial}{\partial u_{n-2}} \Bigr) \Bigr\rvert}, \ \epsilon \right)
	\end{align*}
	and for every $\lvert \epsilon \rvert < \delta$ the map
	\begin{equation} \label{DS5.2:outuntnrml}
		u \in \mathbb{S}^{n-2} \mapsto \nu_{\epsilon}(u) = \frac{\star \Bigl(\bigwedge \nolimits_{n-2} \partial_{1} \psi(u, 			\epsilon)\Bigr) \Bigl( \frac{\partial}{\partial u_{1}} \wedge \cdots \wedge \frac{\partial}{\partial u_{n-2}} \Bigr)}			{\Bigl\lvert \star \Bigl(\bigwedge \nolimits_{n-2} \partial_{1} \psi(u, \epsilon)\Bigr) \Bigl( \frac{\partial}{\partial 			u_{1}} \wedge \cdots \wedge \frac{\partial}{\partial u_{n-2}} \Bigr) \Bigr\rvert}
	\end{equation}
	which, according to Remark \ref{RS3.2:outrunitnorfld}, is the outer unit normal field along the closed embedded 			hypersurface $\psi(\mathbb{S}^{n-2}, \epsilon)$. The map $u \mapsto \nu_{\epsilon}(u)$ is a diffeomorphism because $u 		\mapsto \psi(u, \epsilon)$ parameterizes the ovaloid $\mathcal{O}_{\tau}(z_{0}, \epsilon)$ and the outer unit normal field $	\nu_{\epsilon}$, in \eqref{DS5.2:outuntnrml}, is a diffeomorphism as a result of a simple observation in the case $n = 3$ 		(oval), and from Hadamard's theorem \cite[p.41]{jH98, sKkN69} in the case $n \geq 4$. Since $\nu$ is bijective and its 		differential has full rank, by the inverse function theorem, $\nu$ is a diffeomorphism with the smooth inverse
	\begin{align*}
		\theta \colon \mathbb{S}^{n-2} \times (-\delta, \delta) &\to \mathbb{S}^{n-2} \times (-\delta, \delta)\\
		(u, \epsilon)							      &\mapsto \bigl( \theta_{\epsilon}(u), \epsilon \bigr) 
	\end{align*}
	where $\theta_{\epsilon} = \nu_{\epsilon}^{-1}$, $\lvert \epsilon \rvert < \delta$. Since $\theta$ is smooth	, 				$\{ \theta_{\epsilon} \colon \lvert \epsilon \rvert  < \delta \}$ is a $1$-parameter family of diffeomorphisms. 
	
	For each $\lvert 	\epsilon \rvert < \delta$, since $\psi( \cdot, \epsilon)$ parameterizes $\mathcal{O}_{\tau}(z_{0}, \epsilon)		$, the map
	\begin{align*}
		&\mathbb{S}^{n-2} \to \mathbb{R}^{n-1}\\
		&u \mapsto \bigl( \theta_{\epsilon}(u), \epsilon \bigr) \mapsto \psi \bigl( \theta_{\epsilon}(u), \epsilon \bigr) = 				   \Gamma \bigr(\, \theta_{\epsilon}(u), z\bigl( \epsilon, \theta_{\epsilon}(u) \bigr) \,\bigr)
	\end{align*}
	support parameterizes $\mathcal{O}_{\tau}(z_{0}, \epsilon)$ because the unit vector $u \in \mathbb{S}^{n-2}$ is the outer 	unit normal to the ovaloid $\psi(\mathbb{S}^{n-2}, \epsilon) = \mathcal{O}_{\tau}(z_{0}, \epsilon)$ at
	\[
		\psi \bigl( \theta_{\epsilon}(u), \epsilon \bigr) = \Gamma \bigl(\, \theta_{\epsilon}(u), z \bigl( \epsilon, 					\theta_{\epsilon}(u) \bigr) \,\bigr).
	\]
	
	\begin{figure}[h]
		\centering
		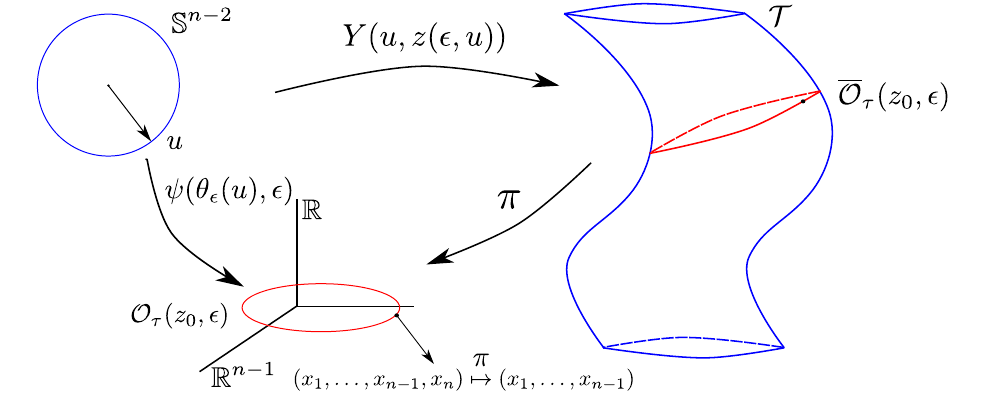
		\caption{Support parameterization of $\mathcal{O}_{\tau}(z_{0}, \epsilon)$}
	\end{figure}
	
	When $\epsilon = 0$,
	\begin{align*}
		&\nu_{0} \colon \mathbb{S}^{n-2} \to \mathbb{S}^{n-2}\\
		&u						  \mapsto \frac{\star \Bigl(\bigwedge \nolimits_{n-2} \partial_{1} \psi(u, 0)\Bigr) 			\Bigl( \frac{\partial}{\partial u_{1}} \wedge \cdots \wedge \frac{\partial}{\partial u_{n-2}} \Bigr)}{\Bigl\lvert \star 			\Bigl(\bigwedge \nolimits_{n-2} \partial_{1} \psi(u, 0)\Bigr) \Bigl( \frac{\partial}{\partial u_{1}} \wedge \cdots 			\wedge \frac{\partial}{\partial u_{n-2}} \Bigr) \Bigr\rvert} = u
	\end{align*}
	because $u \mapsto \Gamma \bigl( u, z(0, u) \bigr) = \Gamma(u, z_{0})$ is already the support parameterization of $			\mathcal{O}(z_{0}) = \mathcal{O}_{\tau}(z_{0}, 0)$. So we can conclude that for each $\lvert \epsilon \rvert < \delta$, $		\Gamma_{\epsilon} \circ \theta_{\epsilon}$ support parameterizes the ovaloid $\mathcal{O}_{\tau}(z_{0}, \epsilon)$.
	
	Fix $u \in \mathbb{S}^{n-2}$, $v \in \tng_{u} \mathbb{S}^{n-2}$, and choose a smooth curve $c \colon J \to \mathbb{S}		^{n-2}$ so that $c(0) = u$ and $c'(0) = v$. Then for every $\lvert \epsilon \rvert  < \delta$
	\begin{align}
		0 &= u \cdot \derv \, (\Gamma_{\epsilon} \circ \theta_{\epsilon})(u) \, v \notag\\
		   &= u \cdot \left. \frac{\partial}{\partial t} \right\rvert_{t=0} (\Gamma_{\epsilon} \circ \theta_{\epsilon}) \bigl[ c(t) 				\bigr] = \left. \frac{\partial}{\partial t} \right\rvert_{t=0} u \cdot (\Gamma_{\epsilon} \circ \theta_{\epsilon}) 				\bigr[ c(t) \bigr]. \notag\\
		\intertext{Taking the $\epsilon$-derivative at $\epsilon = 0$ and using the smoothness of $\Gamma$ we obtain}
		0 &= \left. \frac{\partial}{\partial \epsilon} \right\rvert_{\epsilon=0} \left. \frac{\partial}{\partial t} \right\rvert_{t=0} 			\Bigl\{ u \cdot \Gamma \Bigl(\, \theta_{\epsilon}(c(t)), z\bigl( \epsilon, \theta_{\epsilon}(c(t)) \bigr) \,\Bigr) 				\Bigr	\} \notag\\
		   &= \left. \frac{\partial}{\partial t} \right\rvert_{t=0} \left\{ u \cdot \left. \frac{\partial}{\partial \epsilon} \right					\rvert_{\epsilon=0} \Gamma \Bigl(\, \theta_{\epsilon}(c(t)), z\bigl( \epsilon, \theta_{\epsilon}(c(t)) \bigr) \, 				\Bigr)\right\}. \label{ES5.2:drvcompsptfnc1}
	\end{align}
	$\epsilon \mapsto \Bigl(\, \theta_{\epsilon} \bigl( c(t) \bigr), z \bigl( \epsilon, \theta_{\epsilon}(c(t)) \bigr) \,\Bigr)$ is a 		smooth curve on $\mathbb{S}^{n-2} \times (-\delta, \delta)$ whose $\epsilon$-derivative at $\epsilon = 0$ equals 
	\begin{align*}
		& \left. \frac{\partial}{\partial \epsilon} \right\rvert_{\epsilon=0} \Bigl(\, \theta_{\epsilon} \bigl( c(t) \bigr), z 				     \bigl( \epsilon, \theta_{\epsilon}(c(t)) \bigr) \,\Bigr)\\
		 &= \left(\, \left. \frac{\partial}{\partial \epsilon} \right\rvert_{\epsilon=0} \theta_{\epsilon} \bigl( c(t) \bigr),\  					(\partial_{1}z) \bigl(\, 0, c(t) \,\bigr) (1) + (\partial_{2}z) \bigl(\, 0, c(t) \,\bigr) \left( \left. \frac{\partial}{\partial 			\epsilon} \right\rvert_{\epsilon=0} \theta_{\epsilon} \bigl( c(t) \bigr) \right) \,\right).
	\end{align*}
	Since $z( 0, \cdot ) \equiv z_{0}$, $(\partial_{2}z) \bigl(\, 0, c(t) \,\bigr) = 0$, using the equation 						\eqref{ES5.2:drvofimpfnc} we can compute 			
	\[
		(\partial_{1}z) \bigl(\, 0 , c(t) \,\bigr)(1)  = \left. \frac{\partial}{\partial \epsilon} \right\rvert_{\epsilon=0} z				\bigl(\, \epsilon, c(t) \,\bigr) = \tau \cdot \Gamma \bigl(\, c(t), z_{0} \,\bigr)
	\]
	and hence the tangent vector equals
	\[
		\left. \frac{\partial}{\partial \epsilon} \right\rvert_{\epsilon=0} \Bigl(\,  \theta_{\epsilon} \bigl (c(t) \bigr), z \bigl(			\epsilon, \theta_{\epsilon} \bigl( c(t) \bigr) \bigr) \, \Bigr) = \left(\, \left. \frac{\partial}{\partial \epsilon} \right				\rvert_{\epsilon=0} \theta_{\epsilon} \bigl( c(t) \bigr), \ \tau \cdot \Gamma \bigl(\, c(t), z_{0} \,\bigr) \right).
	\]
	The computation of the derivative in \eqref{ES5.2:drvcompsptfnc1} can be continued as
	\begin{align*}
		\Rightarrow \left. \frac{\partial}{\partial \epsilon} \right\rvert_{\epsilon=0} (\Gamma_{\epsilon} \circ 					\theta_{\epsilon}) \bigl[ c(t) \bigr] &= (\partial_{1}\Gamma) \bigl( c(t), z_{0} \bigr) \left( \left. \frac{\partial}{\partial 		\epsilon} \right\rvert_{\epsilon=0} \theta_{\epsilon} \bigl[ c(t) \bigr] \right)\\
		& \phantom{(\partial_{1}\Gamma) \bigl[ c(t), z_{0} \bigr] + }+ (\partial_{2} \Gamma) \bigl( c(t), z_{0} \bigr) \bigl(\, 		\tau \cdot \Gamma \bigl( c(t), z_{0} \bigr) \,\bigr)
	\end{align*}
	\begin{align*}
		\Rightarrow 0 &= \left. \frac{\partial}{\partial t} \right\rvert_{t=0} \left\{ u \cdot (\partial_{1}\Gamma) \bigl( c(t), 			z_{0} \bigr) \left. \frac{\partial}{\partial \epsilon} \right\rvert_{\epsilon=0} \theta_{\epsilon} \bigl[ c(t) \bigr] \right\}\\
		 &\phantom{\left. \frac{\partial}{\partial t} \right\rvert_{t=0} \left\{ u \cdot (\partial_{1}\Gamma) \bigl( c(t) \right\}} + 		\left. \frac{\partial}{\partial t} \right\rvert_{t=0} \Bigl\{ u \cdot (\partial_{2} \Gamma) \bigl( c(t), z_{0} \bigr) \bigl( \tau 		\cdot \Gamma \bigl( c(t), z_{0} \bigr) \bigr) \Bigr\}\\
		 &= u \cdot \left. \frac{\partial}{\partial t} \right\rvert_{t=0} \left\{ \derv \Gamma \bigl( c(t), z_{0} \bigr) \left( \left. 			\frac{\partial}{\partial \epsilon} \right\rvert_{\epsilon=0} \bigl( \theta_{\epsilon} \bigl[ c(t) \bigr], z_{0} \bigr) \right) 		\right\}\\
		 &\phantom{\left. \frac{\partial}{\partial t} \right\rvert_{t=0} \left\{ u \cdot (\partial_{1}\Gamma) \bigl( c(t) \right\}} + u 		\cdot \left. \frac{\partial}{\partial t} \right\rvert_{t=0} \left\{ \left( \frac{\partial}{\partial z} \right \rvert_{z=z_{0}} 			\Gamma \bigl( c(t), z \bigr) \right) \bigl( \tau \cdot \Gamma \bigl( c(t), z_{0} \bigr) \bigr)\\
		&= u \cdot \left. \frac{\partial}{\partial t} \right\rvert_{t=0} \left. \frac{\partial}{\partial \epsilon} \right					\rvert_{\epsilon=0} \Gamma \bigl(\, \theta_{\epsilon} \bigl( c(t) \bigr), z_{0} \,\bigr)\\
		 &\phantom{\left. \frac{\partial}{\partial t} \right\rvert_{t=0} \left\{ u \cdot (\partial_{1}\Gamma) \bigl( c(t) \right\}} + u \cdot \left\{ \left( \left. \frac{\partial}{\partial t} \right\rvert_{t=0} \left. \frac{\partial}{\partial z} 			\right\rvert_{z=z_{0}} \Gamma \bigl( c(t), z \bigr) \right)  \bigl(\tau \cdot \Gamma (u, z_{0}) \bigr) \right\}\\
		 &\phantom{\left. \frac{\partial}{\partial t} \right\rvert_{t=0} \left\{ u \cdot (\partial_{1}\Gamma) 						\bigl( c(t) \right\}} + u \cdot \left\{ \left( \left. \frac{\partial}{\partial z} \right\rvert_{z=z_{0}} 				\Gamma(u, z) \right) \left( \tau \cdot \left. \frac{\partial}{\partial t} \right\rvert_{t=0} \Gamma \bigl( c(t), z_{0} \bigr) 		\right) \right\}
	\end{align*}
	\[
		\text{Since} \  \left. \frac{\partial}{\partial t} \right\rvert_{t=0} \Gamma \bigl( c(t), z_{0} \bigr) = \left. \frac{\partial}			{\partial t} \right\rvert_{t=0} \Gamma_{0} \bigl[ c(t) \bigr] = \derv \Gamma_{0}(u)v, \quad \text{we continue as}
	\]
	\begin{align*}
		= u \cdot \left. \frac{\partial}{\partial \epsilon} \right\rvert_{\epsilon=0} \left. \frac{\partial}{\partial t} \right				\rvert_{t=0} \Gamma \bigl( \theta_{\epsilon} \bigl( c(t) \bigr), z_{0} \bigr) &+ \left. \frac{\partial}{\partial z} \right			\rvert_{z=z_{0}} \left\{ u \cdot \left. \frac{\partial}{\partial t} \right\rvert_{t=0} \Gamma \bigl( c(t), z \bigr) \right\} 			\bigl( \tau \cdot \Gamma (u, z_{0}) \bigr)\\
		&+ \left[ u \cdot \frac{\partial \Gamma}{\partial z} (u, z_{0}) \right] \bigl( \tau \cdot \derv \Gamma_{0}(u)v \bigr).
	\end{align*}
	For each $\lvert z \rvert < \delta$, $\Gamma(\cdot, z)$ support parameterizes $\mathcal{O}_{\tau}(z, 0) = \mathcal{O}(z)$, 	and 
	\begin{align*}
		&\left. \frac{\partial}{\partial t} \right\rvert_{t=0} \Gamma \bigl[ c(t), z \bigr] \in \tng_{\Gamma(u, z)} \mathcal{O}			    (z) = \tng_{u} \mathbb{S}^{n-2}\\
		&\Rightarrow u \cdot \left. \frac{\partial}{\partial t} \right\rvert_{t=0} \Gamma \bigl( c(t), z\bigr) = 0
	\end{align*}
	Therefore, we can continue the computation as
	\begin{align*}
		&=  u \cdot \left. \frac{\partial}{\partial \epsilon} \right\rvert_{\epsilon=0} \left. \frac{\partial}{\partial t} \right				\rvert_{t=0} \Gamma_{0} \bigl(\, \theta_{\epsilon} \bigl( c(t) \bigr) \,\bigr) + \left[ u \cdot \frac{\partial \Gamma}			{\partial z} (u, z_{0}) \right] \bigl( \tau \cdot \derv \Gamma_{0}(u)v \bigr)\\
		&= u \cdot \left. \frac{\partial}{\partial \epsilon} \right\rvert_{\epsilon=0} \derv \Gamma_{0} \bigl[ \theta_{\epsilon}			(u) \bigr] \bigl( \derv \theta_{\epsilon}(u)v \bigr) +\left[ u \cdot \frac{\partial \Gamma}{\partial z} (u, z_{0}) \right] 			\bigl( \tau \cdot \derv \Gamma_{0}(u)v \bigr).
	\end{align*}
	In summary, we have the equality
	\begin{equation} \label{ES5.2:drvcompsptfnc2}
		\begin{aligned}
			0 &= u \cdot \derv \, \bigl(\Gamma_{\epsilon} \circ \theta_{\epsilon}\bigr)(u)v\\
			   &= u \cdot \left. \frac{\partial}{\partial \epsilon} \right\rvert_{\epsilon=0} \derv \Gamma_{0} 							\bigl[ \theta_{\epsilon}(u) \bigr] \bigl( \derv \theta_{\epsilon}(u)v \bigr) +\left[ u \cdot \frac{\partial 						\Gamma}{\partial z} (u, z_{0}) \right] \bigl( \tau \cdot \derv \Gamma_{0}(u)v \bigr).
		\end{aligned}
	\end{equation}
	Note that $\Gamma_{0}(\cdot) = \Gamma(\cdot, z_{0})$ support parameterizes $\mathcal{O}(z_{0})$ and so for every $		\lvert \epsilon \rvert < \delta$
	\begin{equation} \label{ES5.2:drvcompsptfnc3}
		\theta_{\epsilon}(u) \cdot \derv \Gamma_{0} \bigl[ \theta_{\epsilon}(u) \bigr] \bigl( \derv \theta_{\epsilon}(u)v \bigr) 		= 0.
	\end{equation}
	When we take the $\epsilon$-derivative of the equation \eqref{ES5.2:drvcompsptfnc3} at $\epsilon = 0$ we get
	\begin{align*}
		0 &= \left. \frac{\partial}{\partial \epsilon} \right\rvert_{\epsilon=0} \Bigl\{ \theta_{\epsilon}(u) \cdot \derv 					\Gamma_{0} \bigl[ \theta_{\epsilon}(u) \bigr] \bigl( \derv \theta_{\epsilon}(u)v \bigr) \Bigr\}\\
		   &= \left[ \left. \frac{\partial}{\partial \epsilon} \right\rvert_{\epsilon=0} \theta_{\epsilon}(u) \right] \cdot \derv 				\Gamma_{0}(u)v + u \cdot  \left. \frac{\partial}{\partial \epsilon} \right\rvert_{\epsilon=0} \derv \Gamma_{0} 				\bigl[ \theta_{\epsilon}(u) \bigr] \bigl( \derv \theta_{\epsilon}(u)v \bigr)
	\end{align*}
	\begin{equation} \label{ES5.2:drvcompsptfnc4}
		\Rightarrow u \cdot \left. \frac{\partial}{\partial \epsilon} \right\rvert_{\epsilon=0} \derv \Gamma_{0} 					\bigl[ \theta_{\epsilon}(u) \bigr] \bigl(\derv \theta_{\epsilon}(u)v \bigr) = - \left[ \left. \frac{\partial}{\partial \epsilon} 		\right	\rvert_{\epsilon=0} \theta_{\epsilon}(u) \right] \cdot \derv \Gamma_{0}(u)v.
	\end{equation}
	Using the equality \eqref{ES5.2:drvcompsptfnc4} in the equation \eqref{ES5.2:drvcompsptfnc2}, we get
	\begin{align*}
		0 &= - \left[ \left. \frac{\partial}{\partial \epsilon} \right\rvert_{\epsilon=0} \theta_{\epsilon}(u) \right] \cdot \derv 			 	\Gamma_{0}(u)v + \left[ u \cdot \frac{\partial \Gamma}{\partial z}(u, z_{0}) \right] \bigl( \tau \cdot \derv 				\Gamma_{0}(u)v \bigr)\\
		   &= \left[ \left( u \cdot \frac{\partial \Gamma}{\partial z}(u, z_{0}) \right) \tau - \left. \frac{\partial}{\partial 					 \epsilon} \right\rvert_{\epsilon=0} \theta_{\epsilon}(u) \right] \cdot \derv \Gamma_{0}(u)v\\
		   \intertext{since $\derv \Gamma_{0}(u)v \in \tng_{\Gamma(u, z_{0})} \mathcal{O}(z_{0}) = \tng_{u}\mathbb{S}					 ^{n-2}$ we can replace $\tau$ with $\tau^{\tng}(u)$}
		   &= \left[ \left( u \cdot \frac{\partial \Gamma}{\partial z}(u, z_{0}) \right) \tau^{\tng}(u) - \left. \frac{\partial}					{\partial \epsilon} \right\rvert_{\epsilon=0} \theta_{\epsilon}(u) \right] \cdot \derv \Gamma_{0}(u)v
	\end{align*}
	holds for every $v \in \tng_{u}\mathbb{S}^{n-2}$. Since $\derv \Gamma_{0}(u) \colon \tng_{u}\mathbb{S}^{n-2} \to 		\tng_{u} \mathbb{S}^{n-2}$ is an isomorphism we can conclude that
	\[
		\left. \frac{\partial}{\partial \epsilon} \right\rvert_{\epsilon=0} \theta_{\epsilon}(u) =  \left( u \cdot \frac{\partial 			\Gamma}{\partial z}(u, z_{0}) \right) \tau^{\tng}(u)
	\]
	holds for every $u \in \mathbb{S}^{n-2}$.
\end{proof}	

\subsubsection{The symmetry obstruction}

Our goal is to get an analytic formulation of central symmetry. Since the constancy of the centrix is equivalent to central symmetry, an appropriate derivative of the centrix, which we call the symmetry obstruction, will provide the needed formulation.

\begin{definition}
	Given any tilt direction $\tau \in \mathbb{S}^{n-2}$, $z_{0} \in I$, and $\delta > 0$ as in Proposition					\ref{PS5.2:sptreparam}, define the $1$-parameter family of smooth maps
	\begin{align*}
		c_{\epsilon} \colon \mathbb{S}^{n-2} \to \mathbb{S}^{n-2}\\
		\intertext{for each $\lvert \epsilon \rvert < \delta$ as}
		c_{\epsilon}(u) = c_{\epsilon}(u, z_{0}) = \frac{\Gamma_{\epsilon} \bigl[ \theta_{\epsilon}(u) \bigr] + \Gamma 			\bigl[ \theta_{\epsilon}(-u) \bigr]}{2} 
	\end{align*}
	Namely, for each $\lvert \epsilon \rvert < \delta$, $c_{\epsilon}(\cdot, z_{0})$ is the centrix of the projected ovaloid $		\mathcal{O}_{\tau}(z_{0}, \epsilon)$.
\end{definition}
	
Let $\Gamma^{-}( \cdot, z)$ denote the odd part of $\Gamma(\cdot, z)$ for each $-1 < z < 1$ and define the maps
\begin{equation} \label{NS5.2:sptfnc}
	\begin{aligned}
		&\psi \colon \mathbb{S}^{n-2} \to \mathbb{R}^{n-1} &\hspace{2cm}  & \rho \colon \mathbb{R}^{n-1} \to 				\mathbb{R}^{n-1}\\
		&\psi(u) = \Gamma^{-}(u, z_{0})       		        	       &                          & \rho(x) = -x\\[0.5cm]
		&\bar{\theta}_{\epsilon} \colon \mathbb{S}^{n-2} \to \mathbb{S}^{n-2} &&\\
		&\bar{\theta}_{\epsilon} = (\theta_{\epsilon} \circ \rho)(u).                         &&  
	\end{aligned}
\end{equation}
Then we can compute for each $u \in \mathbb{S}^{n-2}$
\begin{align*}
	(\psi \circ \rho)(u) 			          &= \psi(-u) = \Gamma^{-}(-u, z_{0}) = \frac{\Gamma(-u, z_{0}) -\Gamma(u, 										z_{0})}{2}\\
				  			          &= - \Gamma^{-}(u, z_{0}) = - \psi(u) = (\rho \circ \psi)(u)
\end{align*}
\[
	\Rightarrow \derv \psi(-u) = - \derv \, (\psi \circ \rho)(u)   = - \derv \, (\rho \circ \psi)(u) =  \derv \psi(u)
\]
which can be equivalently written as
\begin{equation} \label{ES5.2:dervoddsptparam}
	(\partial_{1} \Gamma^{-})(-u, z_{0})  = (\partial_{1} \Gamma^{-})(u, z_{0}).
\end{equation}

\begin{lemma}
	Supose the horizontal cross-section $\mathcal{O}(z)$ of a transversely convex tube $\mathcal{T}$ is central about $(c(z), 		z) $ for each $-1 < z < 1$. Then for any tilt direction $\tau \in \mathbb{S}^{n-2}$ and $z_{0} \in I$, we have
	\begin{equation} \label{ES5.2:dervctx}
		\begin{aligned}
				\left. \frac{\partial}{\partial \epsilon} \right\rvert_{\epsilon=0} c_{\epsilon}(u, z_{0}) &=							\left( \frac{\partial \Gamma^{-}}{\partial z}(u, z_{0}) \cdot u \right) (\partial_{1}\Gamma^{-})(u, z_{0}) 					\bigl[\tau^{\tng}(u)\bigr]\\
				&+ \bigl[ \tau \cdot c(z_{0}) \bigr] c'(z_{0}) + \bigl[ \tau \cdot \Gamma^{-}(u, z_{0})\bigr] \frac{\partial 					\Gamma^{-}}{\partial z} (u, z_{0}).
		\end{aligned}
	\end{equation}
\end{lemma}

\begin{proof}
	Using the notation in \eqref{NS5.2:sptfnc} the centrix can be written as
	\[
		c_{\epsilon}(u, z_{0}) = \frac{1}{2} \Bigl\{ \Gamma \Bigl(\, \theta_{\epsilon}(u), z\bigl( \epsilon, \theta_{\epsilon}			(u) \bigr) \,\Bigr) + \Gamma \Bigl(\, \bar{\theta}_{\epsilon}(u), z \bigl( \bar{\theta}_{\epsilon}(u), z 						\bigl(\epsilon, \bar{\theta}_{\epsilon}(u) \bigr) \,\Bigr) \Bigr\}
	\]
	and the $\epsilon$-derivative at $\epsilon = 0$ of the first summand equals
	\begin{align}
		&\left. \frac{\partial}{\partial \epsilon} \right\rvert_{\epsilon=0} \Gamma \Bigl(\, \theta_{\epsilon}(u), z 					\bigl( \epsilon, \theta_{\epsilon}(u) \bigr) \,\Bigr) \notag\\
		&= \derv \Gamma(u, z_{0}) \left( \left. \frac{\partial}{\partial \epsilon} \right\rvert_{\epsilon=0} \theta_{\epsilon}(u), 		\frac{\partial z}{\partial \epsilon} (0, u) + (\partial_{2}z) (0, u) \left[ \left. \frac{\partial}{\partial \epsilon} \right			\rvert_{\epsilon=0} \theta_{\epsilon}(u) \right] \right). \label{ES5.2:dervcompcnt1}
	\end{align}
	Since $z(0, \cdot) \equiv z_{0}$, it follows that $(\partial_{2}z) (0, u) = 0$ and using the equation						\eqref{ES5.2:drvofimpfnc} the derivative in \eqref{ES5.2:dervcompcnt1} simplifies to
	\[
		(\partial_{1} \Gamma) (u, z_{0}) \left[ \left. \frac{\partial}{\partial \epsilon} \right\rvert_{\epsilon=0} 					\theta_{\epsilon}(u) \right] + \frac{\partial \Gamma}{\partial z}(u, z_{0}) \bigl[ \tau \cdot \Gamma(u, z_{0}) \bigr]
	\]
	For each $z \in I$, $\Gamma(\cdot, z)$ support parameterizes the projected ovaloid $\mathcal{O}(z)$ with center of 			symmetry at $c(z)$ , $\Gamma(\cdot, z) = c(z) + \Gamma^{-}(\cdot, z)$, and according to Lemma 						\ref{LS5.1:oddsptparam}, $\Gamma^{-}(\cdot, z)$ support parameterizes the ovaloid $\mathcal{O}(z) - c(z)$. Using the 		previous equality and the equation \eqref{ES5.2:dervoddsptparam} we can conclude that
	\[
		(\partial_{1} \Gamma)(u, z) = (\partial_{1} \Gamma^{-})(u, z) = (\partial_{1} \Gamma^{-})(-u, z) = (\partial_{1}			\Gamma)(-u, z)
	\]
	for each $u \in \mathbb{S}^{n-2}$ and $z \in I$. Using these observations we can compute the $\epsilon$-derivative of 		$c_{\epsilon}(u, z_{0})$ at $\epsilon = 0$ as follows
	\begin{align*}
		&\left. \frac{\partial}{\partial \epsilon} \right\rvert_{\epsilon=0} c_{\epsilon}(u, z_{0}) = \frac{1}{2} 					\left\{ (\partial_{1} \Gamma) (u,  z_{0}) \left[ \left. \frac{\partial}{\partial \epsilon} \right\rvert_{\epsilon=0} 				\theta_{\epsilon}(u) \right] + \frac{\partial \Gamma}{\partial z}(u, z_{0}) \bigl[ \tau \cdot \Gamma(u, z_{0}) \bigr] 			\right\}\\
		&\quad+ \frac{1}{2} \left\{(\partial_{1} \Gamma) (-u, z_{0}) \left[ \left. \frac{\partial}{\partial \epsilon} \right				\rvert_{\epsilon=0} \bar{\theta}_{\epsilon}(u) \right] + \frac{\partial \Gamma}{\partial z}(-u, z_{0}) \bigl[ \tau \cdot 			\Gamma(-u, z_{0}) \bigr] \right\}\\
		&= \frac{1}{2} \left\{ (\partial_{1} \Gamma^{-}) (u, z_{0}) \left[ \left. \frac{\partial}{\partial \epsilon} \right				\rvert_{\epsilon=0} \theta_{\epsilon}(u) \right] \right\}\\
		 &\quad + \frac{1}{2} \left\{ \left( c'(z_{0}) + \frac{\partial \Gamma^{-}}{\partial z} (u, z_{0}) \right) \bigl( \tau 			\cdot 	 \bigl[ c(z_{0}) + \Gamma^{-}(u, z_{0}) \bigr] \bigr) \right\}\\
		&\qquad + \frac{1}{2} \left\{ (\partial_{1} \Gamma) (u, z_{0}) \left[ \left. \frac{\partial}{\partial \epsilon} \right			\rvert_{\epsilon=0} \bar{\theta}_{\epsilon}(u) \right] \right\}\\
		&\quad \qquad + \frac{1}{2} \left\{ \left( c'(z_{0}) - \frac{\partial \Gamma^{-}}{\partial z} (u, z_{0}) \right) 				\bigl( \tau \cdot \bigl[ c(z_{0}) - \Gamma^{-}(u, z_{0}) \bigr] \bigr) \right\}
	\end{align*}
	\begin{equation} \label{ES5.2:dervcompcnt2}
		\begin{aligned}
			&= \frac{1}{2} \left\{  (\partial_{1} \Gamma^{-}) (u, z_{0}) \left[ \left. \frac{\partial}{\partial \epsilon} \right				\rvert_{\epsilon=0} \theta_{\epsilon}(u) +  \left. \frac{\partial}{\partial \epsilon} \right \rvert_{\epsilon=0} 				\bar{\theta}_{\epsilon}(u) \right]\right\} + \bigl[ \tau \cdot c(z_{0}) \bigr] c'(z_{0})\\
			&\qquad + \bigl[ \tau \cdot \Gamma^{-}(u, z_{0}) \bigr] \frac{\partial \Gamma^{-}}{\partial z}(u, z_{0})
		\end{aligned}
	\end{equation}
	The Proposition \ref{PS5.2:sptreparam} gives the $\epsilon$-derivatives of $\theta_{\epsilon}$ and $\bar{\theta}			_{\epsilon}$ as
	\begin{equation} \label{ES5.2:dervcompcnt3}
		\left. \frac{\partial}{\partial \epsilon} \right\rvert_{\epsilon=0} \theta_{\epsilon}(u) =  \left( u \cdot 						\frac{\partial \Gamma}{\partial z}(u, z_{0})\right) \tau^{\tng}(u) = u \cdot \left[ c'(z_{0}) + \frac{\partial 					\Gamma^{-}}{\partial z}(u, z_{0})\right] \tau^{\tng}(u)  
	\end{equation}
	\begin{equation} \label{ES5.2:dervcompcnt4}
		\begin{aligned} 
			\left. \frac{\partial}{\partial \epsilon} \right\rvert_{\epsilon=0} \bar{\theta}_{\epsilon}(u) &= \left. 						\frac{\partial}{\partial \epsilon} \right\rvert_{\epsilon=0} \theta_{\epsilon}(-u) = (-u) \cdot \left[ c'(z_{0}) + 				\frac{\partial \Gamma^{-}}{\partial z}(-u, z_{0})\right] \tau^{\tng}(-u)\\
			&= u \cdot \left[ \frac{\partial \Gamma^{-}}{\partial z}(u, z_{0}) - c'(z_{0}) \right] \tau^{\tng}(u)		
		\end{aligned}
	\end{equation}
	Plugging the derivatives \eqref{ES5.2:dervcompcnt3} and \eqref{ES5.2:dervcompcnt4} into the equation					\eqref{ES5.2:dervcompcnt2} we get 
	\begin{align*}
		\left. \frac{\partial}{\partial \epsilon} \right\rvert_{\epsilon=0} c_{\epsilon}(u, z_{0}) &= \left( \frac{\partial 				\Gamma^{-}}{\partial z}(u, z_{0}) \cdot u \right) (\partial_{1}\Gamma^{-})(u, z_{0}) \bigl[\tau^{\tng}(u)\bigr]\\
		&+ \bigl[ \tau \cdot c(z_{0}) \bigr] c'(z_{0}) + \bigl[ \tau \cdot \Gamma^{-}(u, z_{0})\bigr] \frac{\partial 					\Gamma^{-}}{\partial z} (u, z_{0}). \qedhere
	\end{align*}
\end{proof}

\begin{definition} [The symmetry obstruction ]
	Given any tilt direction $\tau \in \mathbb{S}^{n-2}$ and $z_{0} \in I$, $\mathcal{O}_{\tau}(z_{0}, \epsilon)$ is central if 	and only if $c_{\epsilon} (\cdot, z_{0})$ is constant, or equivalently,
	\[
		\overline{\nabla}_{v} c_{\epsilon} \, \bigr\rvert_{u} = 0
	\]
	for every $u \in \mathbb{S}^{n-2}$ and $v \in \tng_{u} \mathbb{S}^{n-2}$, where the vector field $c_{\epsilon}$ along $		\mathbb{S}^{n-2}$ is extended radially constant in an open neighborhood of $\mathbb{S}^{n-2}$ in $\mathbb{R}^{n-1}$ 	to a vector field in an open subset of $\mathbb{R}^{n-1}$. When $\mathcal{O}_{\tau}(z_{0}, \epsilon)$ has central 			symmetry for all sufficiently small $\lvert \epsilon \rvert > 0$, which \emph{cop} requires, we have
	\begin{equation} \label{ES5.2:symobstr}																	0 = \left. \frac{\partial}{\partial \epsilon} \right\rvert_{\epsilon=0} \overline{\nabla}_{v} c_{\epsilon} \, \bigr				\rvert_{u} = \overline{\nabla}_{v} \left[ \left. \frac{\partial}{\partial \epsilon} \right\rvert_{\epsilon=0} c_{\epsilon}			(\cdot, z_{0}) \right]_{u}
	\end{equation}
	for every $u \in \mathbb{S}^{n-2}$ and $v \in \tng_{u} \mathbb{S}^{n-2}$. Therefore, the derivative in					\eqref{ES5.2:symobstr} forms an obstruction to central symmetry and hence to \emph{cop}.
\end{definition}

The spherical divergence of the vectorfield
\[
	\left. \frac{\partial}{\partial \epsilon} \right\rvert_{\epsilon=0} c_{\epsilon}(\cdot, z_{0})
\]
equals
\begin{align*}
	\divg_{\mathbb{S}^{n-2}} \left[ \left. \frac{\partial}{\partial \epsilon} \right\rvert_{\epsilon=0} c_{\epsilon}(\cdot, 			z_{0}) \right](u) &= \sum_{j=1}^{n-2} E_{j}(u) \cdot \nabla_{E_{j}(u)}^{\mathbb{S}^{n-2}} \left[ \left. \frac{\partial}		{\partial \epsilon} \right\rvert_{\epsilon=0} c_{\epsilon}(\cdot, z_{0}) \right]_{u}\\ 
				 &= \sum_{j=1}^{n-2} E_{j}(u) \cdot \overline{\nabla}_{E_{j}(u)} \left[ \left. \frac{\partial}{\partial 		\epsilon} \right\rvert_{\epsilon=0} c_{\epsilon}(\cdot, z_{0}) \right]_{u}
\end{align*}
where $\{E_{1}, \dotsc, E_{n-2}\}$ is the local orthonormal field constructed in Definition \ref{DS3.2:orthnmfield}. Since the equality above holds for all $(u, z) \in \mathbb{S}^{n-2} \times I$ we can define a function for each tilt direction $\tau \in \mathbb{S}^{n-2}$ as
\begin{align*}
	&f_{\tau} \colon \mathbb{S}^{n-2} \times I \to \mathbb{R}\\
	&f_{\tau}(u, z) = \divg_{\mathbb{S}^{n-2}} \left[ \left. \frac{\partial}{\partial \epsilon} \right\rvert_{\epsilon=0} 			c_{\epsilon}(\cdot, z) \right](u).
\end{align*}
The central ovaloid property of the tube $\mathcal{T}$ implies $f_{\tau}(u, z) = 0$ for every $u \in \mathbb{S}^{n-2}$, $z \in I$, and $\tau \in \mathbb{S}^{n-2}$.

\subsubsection{The partial differential equations}

Our goal is to construct two partial differential equations satisfied by the support function of the rectified tube $\mathcal{T}^{-}$. These equations will have important analytical and geometric consequences.

\begin{definition}[Transverse support function]
	For each $\lvert z \rvert < 1$, $\Gamma^{-}(\cdot, z)$ support parameterizes the horizontal cross-section $\mathcal{O}(z) - 	c(z)$ of the rectified tube $\mathcal{T}^{-}$. Therefore, there exists a smooth function
	\begin{align}
		&h \colon \mathbb{S}^{n-2} \times I \to \mathbb{R} \notag\\
		&h(u, z) = u \cdot \Gamma^{-}(u, z) \label{DS5.2:trsvsptfnct}
	\end{align}
	which, for each $\lvert z \rvert < 1$, yields the support function of the ovaloid $\mathcal{O}(z) - c(z)$. We call $h$ the 		transverse support function of $\mathcal{T}^{-}$.
\end{definition}

\begin{proposition} \label{PS5.2:pde}
	On a transversely convex tube $\mathcal{T}^{-}$ with central ovaloid property, the transverse support function $h$ of $		\mathcal{T}^{-}$ satisfies the two partial differential equations:
	\begin{equation} \label{PS5.2:pde1}
		\begin{aligned}
			&\frac{\partial h}{\partial z}(u, z) \nabla_{\mathbb{S}^{n-2}} \bigl[\, \Delta_{\mathbb{S}^{n-2}} h(\cdot, z) + 			(n-2) h(\cdot, z) \,\bigr] \, \Bigr\rvert_{u}\\
			&\ + \frac{\partial}{\partial z} \bigl[\, \Delta_{\mathbb{S}^{n-2}} h(u, z) + (n-2) h(u, z) \,\bigr] 						\nabla_{\mathbb{S}^{n-2}} h(\cdot, z) \, \Bigr\rvert_{u}\\
			&\ + 2 \left\{ h(u, z) \nabla_{\mathbb{S}^{n-2}} \frac{\partial h}{\partial z}(\cdot, z) \, \Bigr\rvert_{u} + 					\nabla_{\textstyle \nabla_{\mathbb{S}^{n-2}} \left. \frac{\partial h}{\partial z} (\cdot, z) \,\right\rvert_{u}}				^{\mathbb{S}^{n-2}} \nabla_{\mathbb{S}^{n-2}} h( \cdot, z) \, \Bigr\rvert_{u} \right\} = 0
		\end{aligned}
	\end{equation}
	\begin{equation} \label{PS5.2:pde2}
		\begin{aligned}
			h(u, z) \frac{\partial}{\partial z} \bigl[ \Delta_{\mathbb{S}^{n-2}} h(u, z) &+ (n-2)h(u, z) \bigr]\\
		 														    &- \bigl[ \Delta_{\mathbb{S}					^{n-2}} h(u, z) + (n-2) h(u, z) \bigr] \frac{\partial h}{\partial z}(u, z) = 0
		 \end{aligned}
	\end{equation}
\end{proposition}

\begin{notation}
	Unless stated otherwise, we will use the following abbreviations in the rest of the section:
	\begin{alignat*}{2}
		\nabla &= \nabla_{\mathbb{S}^{n-2}} & \quad \Delta &= \Delta_{\mathbb{S}^{n-2}}\\
		\nabla_{X} &= \nabla_{X}^{\mathbb{S}^{n-2}} \text{ for every $X \in \varkappa(\mathbb{S}^{n-2})$}& \divg &= 		\divg_{\mathbb{S}^{n-2}}\\
		N &= N_{\mathbb{S}^{n-2}} \text{ the outer unit normal field along $\mathbb{S}^{n-2}$} &&
	\end{alignat*}
\end{notation}

\begin{proof}
	Our aim is to show that for an given $u \in \mathbb{S}^{n-2}$, $\lvert z_{0} \rvert < 1$, and $\tau \in \mathbb{S}^{n-2}$
	\begin{align}
	0 &= f_{\tau}(u, z_{0}) = \tau^{\tng}(u) \cdot \left\{ \frac{\partial h}{\partial z}(u, z_{0}) \nabla \bigl[ \Delta h(\cdot, 			z_{0}) + (n-2) h(\cdot, z_{0}) \bigr] \Bigr\rvert_{u} \right.\\
	   &\phantom{= f_{\tau}(u, z_{0}) = \tau^{\tng}(u) \cdot \ } + \frac{\partial}{\partial z} \left[ \Delta h(u, z_{0}) + 		      (n-2)h(u, z_{0}) \right] \nabla h(\cdot, z_{0}) \, \bigr\rvert_{u} \notag\\
	   &\phantom{= f_{\tau}(u, z_{0}) = \tau^{\tng}(u) \cdot \ }+ 2 \left. \left( h(u, z_{0}) \nabla \frac{\partial h}{\partial 		      z}(\cdot, z_{0}) \, \bigr\rvert_{u} + \nabla_{\textstyle \nabla \left. \frac{\partial h}{\partial z} (\cdot, z_{0}) \, \right		      \rvert_{u}} \nabla h( \cdot, z_{0}) \, \bigr\rvert_{u} \right) \right\} \notag\\
	   &\phantom{= f_{\tau}(u, z_{0})} + (\tau \cdot u) \left\{ h(u, z_{0}) \frac{\partial}{\partial z} \bigl[ \Delta h(u, z_{0}) + 		      (n-2) h(u, z_{0}) \bigr] \right.\\
	   & \phantom{= f_{\tau}(u, z_{0}) = \tau^{\tng}(u) \quad}\left. - \bigl[ \Delta h(u, z_{0}) + (n-2) h(u, z_{0}) \bigr] 		       	       \frac{\partial h}{\partial z}(u, z_{0}) \right\}\notag
	\end{align}
	Since $\tau \in \mathbb{S}^{n-2}$ is arbitrary and the expressions in curly brackets do not depend on $\tau$ we can 			conclude that each one of the two expressions must vanish and thus we get the two partial differential equations 				\eqref{PS5.2:pde1} and \eqref{PS5.2:pde2}.
	
	Let $\{e_{1}, \dotsc, e_{n-2}\}$ be an orthonormal basis of $\tng _{u} \mathbb{S}^{n-2}$. Extend them to a local 			\hypertarget{ortfrmfld}{orthonormal frame field} $\{E_{1}, \dotsc, E_{n-2}\}$ such that $E_{i} \cdot E_{j} = \delta_{ij}$ 	and $\nabla_{E_{i}} E_{j} \, \bigl \lvert_{u} = 0$ hold for every $i$, $j = 1, \ldots, n-2$.
	
	Recall that for every $z_{0} \in I$, $\Gamma(\cdot, z_{0}) \colon \mathbb{S}^{n-2} \to \mathbb{R}^{n-1}$ support 		parameterizes $\mathcal{O}(z_{0})$ and $\Gamma^{-}(\cdot, z_{0})$ support parameterizes $\mathcal{O}(z_{0}) - 		c(z_{0})$ with support function $h(\cdot, z_{0})$.
	
	Throughout the proof assume that each $E_{j}$, $j=1, \dotsc, n-2$, $\Gamma(\cdot, z_{0})$, $\Gamma^{-}(\cdot, z_{0})$ 	and $N$ is extended radially constant in a neighborhood of $\mathbb{S}^{n-2}$ in $\mathbb{R}^{n-1}$ to a vectorfield in 	an open subset of $\mathbb{R}^{n-1}$. 
	
	Using the equation \eqref{ES5.2:dervctx} the function $f_{\tau}$ can be expanded as
	\begin{align*}
		0 &= f_{\tau}(u, z_{0}) = \divg \left[ \left. \frac{\partial}{\partial \epsilon} \right\rvert_{\epsilon=0} c_{\epsilon}			(\cdot, z_{0}) \right](u)\\
		& = \divg \left[ v \mapsto \bigl( \tau \cdot \Gamma^{-}(v, z_{0}) \bigr) \frac{\partial \Gamma^{-}}{\partial z}(v, 			z_{0}) \right](u) + \divg \bigl[ v \mapsto \bigl( \tau \cdot c(z_{0}) \bigr) c'(z_{0}) \bigr](u)\\
		&\quad + \divg \left[ v \mapsto \left( \frac{\partial \Gamma^{-}}{\partial z}(v, z_{0}) \cdot v \right) (\partial_{1} 			\Gamma^{-})(v, z_{0}) \bigl[ \tau^{\tng}(v) \bigr] \right](u).
	\end{align*}
	Clearly the second summand vanishes and $f_{\tau}(u, z_{0})$ is equal to
	\begin{equation} \label{ES5.2:pde}
		\begin{aligned}
			 0 = f_{\tau}(u, z_{0}) &= \underbrace{\divg \left[ v \mapsto \bigl( \tau \cdot \Gamma^{-}(v, z_{0}) \bigr) 									\frac{\partial \Gamma^{-}}{\partial z}(v, z_{0}) \right](u)}_{\crci}\\
			 				  &\quad + \underbrace{\divg \left[ v \mapsto \left( \frac{\partial \Gamma^{-}}{\partial z}			     (v, z_{0}) \cdot v \right) (\partial_{1} \Gamma^{-})(v, z_{0}) \bigl[ \tau^{\tng}(v) \bigr] \right](u)}_{\crcii}	
		\end{aligned}
	\end{equation}
	The computation of the part \crci \ of the equation \eqref{ES5.2:pde} is carried out as follows:
	\begin{equation} \label{ES5.2:pde1}
		\begin{aligned}
			&\divg \left[ v \mapsto \bigl( \tau \cdot \Gamma^{-}(v, z_{0}) \bigr) \frac{\partial \Gamma^{-}}{\partial z}(v, 				z_{0}) \right](u)\\
			&= \sum_{j=1}^{n-2} E_{j}(u) \cdot \overline{\nabla}_{E_{j}(u)} \left[ v \mapsto \bigl( \tau \cdot 						       \Gamma^{-} (v, z_{0})\bigr) \frac{\partial \Gamma^{-}}{\partial z}(v, z_{0}) \right]_{u}\\
			&= \underbrace{\sum_{j=1}^{n-2} E_{j}(u) \cdot \left\{ \underbrace{E_{j}\bigl[ v \mapsto \tau \cdot 					      \Gamma^{-}(v, z_{0})\bigr](u)}_{\crci} \frac{\partial \Gamma^{-}}{\partial z}(u, z_{0}) \right\}}_{\crcii}\\
			&\quad + \underbrace{\sum_{j=1}^{n-2}\bigl[ \tau \cdot \Gamma^{-}(u, z_{0}) \bigr] E_{j}(u) \cdot 					    \underbrace{\overline{\nabla}_{E_{j}(u)} \left[ v \mapsto \frac{\partial \Gamma^{-}}{\partial z}(v, z_{0}) 				     \right]_{u}}_{\crciii}}_{\crciv}
		\end{aligned}
	\end{equation}
	The part \crci \ of the equation \eqref{ES5.2:pde1} can be expanded as
	\begin{equation}
		\begin{aligned}
			&E_{j}\bigl[ v \mapsto \tau \cdot \Gamma^{-}(v, z_{0}) \bigr](u) = \tau \cdot \overline{\nabla}_{E_{j}(u)} 				\bigl[ v \mapsto \Gamma^{-}(v, z_{0}) \bigr] \, \bigr\rvert_{u}\\
			&=\tau \cdot \overline{\nabla}_{E_{j}(u)} \bigl[ v \mapsto \nabla h(\cdot, z_{0}) \, \bigr\rvert_{v} + h(v, 					z_{0})v \bigr] \, \bigr\rvert_{u}\\
			&= \tau \cdot \overline{\nabla}_{E_{j}(u)} \nabla h(\cdot, z_{0}) \, \bigr\rvert_{u} + \tau \cdot 						\overline{\nabla}_{E_{j}(u)} \bigl[v \mapsto h(v, z_{0})v \bigr] \, \bigr\rvert_{u}\\
			&=\tau^{\tng}(u) \cdot \nabla_{E_{j}(u)} \nabla h(\cdot, z_{0}) \, \bigr\rvert_{u} + \bigl[(\tau \cdot u)u \bigr] 				\cdot \overline{\nabla}_{E_{j}(u)} \nabla h(\cdot, z_{0}) \, \bigr\rvert_{u}\\
			&\qquad + \tau \cdot \left\{ E_{j}\bigr[h(\cdot, z_{0})\bigr](u)u + h(u, z_{0}) \overline{\nabla}_{E_{j}(u)}N \, 			\bigr\rvert_{u} \right\}
		\end{aligned}
	\end{equation}
	Recalling the fact that $\overline{\nabla} N \, \bigl\lvert_{u} = \mathrm{id}_{\tng_{u}\mathbb{S}^{n-2}}$ for 		each $u \in \mathbb{S}^{n-2}$ we obtain
	\begin{equation} \label{ES5.2:pde1;1}
		\begin{aligned}
			\underbrace{E_{j} \bigl[ v \mapsto \tau \cdot \Gamma^{-}(v, z_{0}) \bigr](u)}_{\crci} &= \tau^{\tng}(u) \cdot  			\underbrace{\nabla_{E_{j}(u)} \nabla h(\cdot, z_{0}) \, \bigr\rvert_{u}}_{\crcii}\\
			 &+(\tau \cdot u)  \underbrace{u \cdot \overline{\nabla}_{E_{j}(u)} \nabla h(\cdot, z_{0}) \, \bigr\rvert_{u}}				     _{\crciii}\\
			 & + \tau \cdot \bigl\{ E_{j}\bigl[h(\cdot, z_{0})\bigr](u)u + h(u, z_{0})E_{j}(u)\bigr\}.
		\end{aligned}
	\end{equation}
	The part \crcii \ of the equation \eqref{ES5.2:pde1;1} can be written as
	\begin{equation}
		\begin{aligned}
			&\nabla_{E_{j}(u)} \nabla h(\cdot, z_{0}) \, \bigr\rvert_{u} = \nabla_{E_{j}(u)} \left[ \sum_{k=1}^{n-2} 				E_{k} \bigl[ h(\cdot, z_{0}) \bigr] E_{k} \right]_{u}\\
													     &= \sum_{k=1}^{n-2} \nabla_{E_{j}(u)} 					\Bigl( E_{k} \bigl[ h(\cdot, z_{0}) \bigr] E_{k} \Bigr) \Bigr\rvert_{u}\\
													     &= \sum_{k=1}^{n-2} E_{j}\Bigl( E_{k} 					\bigl[ h(\cdot, z_{0}) \bigr] \Bigr)(u) E_{k}(u) + E_{k} \bigl[ h(\cdot, z_{0}) \bigr](u) 								\underbrace{\nabla_{E_{j}(u)}E_{k} \, \bigr\rvert_{u}}_{=0}\\
													     &=  \sum_{k=1}^{n-2} E_{j}\Bigl( E_{k} 					\bigl[ h(\cdot, z_{0}) \bigr] \Bigr)(u) E_{k}(u).		
		\end{aligned}
	\end{equation}
	The part \crciii \ of the equation \eqref{ES5.2:pde1;1} can be simplified as
	\begin{equation}
		\begin{aligned}
			u \cdot \overline{\nabla}_{E_{j}(u)} \nabla h(\cdot, z_{0}) \, \bigr\rvert_{u} &= \overline{\nabla}_{E_{j}(u)} 				\bigl[ N \cdot \nabla h(\cdot, z_{0}) \bigr] \, \big\rvert_{u} - \overline{\nabla}_{E_{j}(u)} N \, \bigr\rvert_{u} 				\cdot \nabla h(\cdot, z_{0}) \, \bigr\rvert_{u}\\
			&= - E_{j}(u) \cdot \nabla h(\cdot, z_{0}) \, \bigr\rvert_{u}.
		\end{aligned}
	\end{equation}
	Since $\tau \cdot E_{j}(u) = \tau^{\tng}(u) \cdot E_{j}(u)$ the part \crci \ of the equation \eqref{ES5.2:pde1;1} can be 		written as
	\begin{equation} \label{ES5.2:pde1;2}
	 	\begin{aligned}
	 		&E_{j} \bigl[ v \mapsto \tau \cdot \Gamma^{-}(v, z_{0}) \bigr](u)\\
			&= \tau^{\tng}(u) \cdot \left\{ \sum_{k=1}^{n-2} E_{j} \Bigl( E_{k} \bigl[ h(\cdot, z_{0}) \bigr] \Bigr)(u) 				     E_{k}(u) + h(u, z_{0}) E_{j}(u) \right\}\\
			&\phantom{\tau^{\tng}(u) \cdot} + (\tau \cdot u) \Bigl\{ E_{j} \bigl[ h(\cdot, z_{0}) \bigr](u) - E_{j}(u) \cdot 				   \nabla h(\cdot, z_{0}) \, \bigr\rvert_{u}\Bigr\}.
		\end{aligned}
	\end{equation}
	In order to compute the part \crcii \ of the equation \eqref{ES5.2:pde1} we first observe that
	\begin{align*}
		E_{j}(u) \cdot \frac{\partial \Gamma^{-}}{\partial z}(u, z_{0}) &= E_{j}(u) \cdot \left. \frac{\partial}{\partial z} 			\right	\rvert_{z=z_{0}} \bigl[ \nabla h(u, z) + h(u, z)u \bigr] \notag\\
		&= E_{j}(u) \cdot \left[ \nabla \frac{\partial h}{\partial z}(\cdot, z_{0}) \, \bigr\rvert_{u} + \frac{\partial h }{\partial 			z}(u, z_{0})u \right] = E_{j}(u) \cdot \nabla \frac{\partial h}{\partial z}(\cdot, z_{0}) \, \Bigr\rvert_{u}
	\end{align*}
	and then by using the equation \eqref{ES5.2:pde1;2} we can write
	\begin{align*}
		&\sum_{j=1}^{n-2} E_{j} \bigl[ v \mapsto \tau \cdot \Gamma^{-}(v, z_{0}) \bigr](u) \left[ E_{j}(u) \cdot \nabla 			\frac{\partial h}{\partial z}(\cdot, z_{0}) \, \Bigr\rvert_{u} \right]\\
		&= \sum_{j=1}^{n-2} \left\{ \tau^{\tng}(u) \cdot \left( \sum_{k=1}^{n-2} E_{j} \Bigl( E_{k} \bigl[ h(\cdot, z_{0}) 			\bigr] \Bigr)(u) E_{k}(u) + h(u, z_{0}) E_{j}(u) \right) \right.\\\
		&\hspace{1cm} + (\tau \cdot u) \Bigl( E_{j} \bigl[ h(\cdot, z_{0}) \bigr](u) - E_{j}(u) \cdot \nabla h(\cdot, z_{0}) \, 			\bigr\rvert_{u}\Bigr) \Biggr\} \left[ E_{j}(u) \cdot \nabla \frac{\partial h}{\partial z}(\cdot, z_{0}) \, \Bigr\rvert_{u} 			\right]
	\end{align*}
	\begin{align*}
		&=\sum_{k=1}^{n-2} \tau^{\tng}(u) \cdot \left( \sum_{j=1}^{n-2} \left[ E_{j} \cdot \nabla \frac{\partial h}{\partial 			z} (\cdot, z_{0}) \, \Bigr\rvert_{u} \right] E_{j} \Bigl( E_{k} \bigl[ h(\cdot, z_{0}) \bigr] \Bigr)(u) \right) E_{k}(u)\\
		&+ \tau^{\tng}(u) \cdot \sum_{j=1}^{n-2} \left[ E_{j}(u) \cdot \nabla \frac{\partial h}{\partial z}(\cdot, z_{0}) \, 			\Bigr\rvert_{u} \right] h(u, z_{0}) E_{j}(u)\\
		&+ (\tau \cdot u) \left( \sum_{j=1}^{n-2} \left[ E_{j}(u) \cdot \nabla \frac{\partial h}{\partial z}(\cdot, z_{0}) \, \Bigr			\rvert_{u} \right] E_{j} \bigl[ h(\cdot, z_{0}) \bigr](u) - \nabla h(\cdot, z_{0}) \, \bigr\rvert_{u} \cdot \nabla 				\frac{\partial h}{\partial z}(\cdot, z_{0}) \, \Bigr\rvert_{u} \right)
	\end{align*}
	where the last dot product is obtained as
	\begin{align*}
		&\sum_{j=1}^{n-2} \left[ E_{j}(u) \cdot \nabla \frac{\partial h}{\partial z} (\cdot, z_{0}) \, \Bigr\rvert_{u} \right] 			E_{j}(u) \cdot \nabla h(\cdot, z_{0}) \, \bigr\rvert_{u}\\
		&= \nabla h(\cdot, z_{0}) \, \bigr\rvert_{u} \cdot \sum_{j=1}^{n-2} \left[ E_{j}(u) \cdot \nabla \frac{\partial h}				{\partial z}(\cdot, z_{0}) \, \Bigr\rvert_{u} \right] E_{j}(u)\\
		&= \nabla h(\cdot, z_{0}) \, \bigr\rvert_{u} \cdot \nabla \frac{\partial h}{\partial z} (\cdot, z_{0}) \, \Bigr\rvert_{u}
	\end{align*}
	So we can conclude that the part \crcii \ of the equation \eqref{ES5.2:pde1} equals
	\begin{equation}\label{ES5.2:pde1;3}
		\begin{aligned}
			&\sum_{j=1}^{n-2} E_{j} \bigl[ v \mapsto \tau \cdot \Gamma^{-}(v, z_{0}) \bigr](u) E_{j}(u) \cdot 					\frac{\partial \Gamma^{-}}{\partial z}(u, z_{0}) \\
			&=\tau^{\tng}(u) \cdot \Biggl\{ \sum_{k=1}^{n-2} \Biggl( \sum_{j=1}^{n-2} \Biggl[ E_{j}(u) \cdot \nabla 				\frac{\partial h}{\partial z}(\cdot, z_{0}) \, \Bigr\rvert_{u} \Biggr] E_{j} \Bigl( E_{k} \bigl[ h(\cdot, z_{0}) 				\bigr] \Bigr)(u) \Biggr) E_{k}(u) \\
			&\hspace{2.5cm} + h(u, z_{0}) \nabla \frac{\partial h}{\partial z}(\cdot, z_{0}) \, \Bigr\rvert_{u} \Biggr\} 			\end{aligned}
	\end{equation}
	Note that $(\tau \cdot u)$ terms vanishes and \eqref{ES5.2:pde1;3} is obtained because
	\[
		\sum_{j=1}^{n-2} \left[ E_{j}(u) \cdot \nabla \frac{\partial h}{\partial z}(\cdot, z_{0}) \, \Bigr\rvert_{u} \right] 			E_{j}\bigl[ h(\cdot, z_{0}) \bigr](u) = \nabla h(\cdot, z_{0}) \, \bigr\rvert_{u} \cdot \nabla \frac{\partial h}{\partial z}		(\cdot, z_{0}) \, \Bigr\rvert_{u}.
	\]
	In oder to compute the part \crciv \ of the equation \eqref{ES5.2:pde1}, which is
	\[
		\bigl[ \tau \cdot \Gamma^{-}(u, z_{0}) \bigr] \sum_{j=1}^{n-2} E_{j}(u) \cdot \overline{\nabla}_{E_{j}(u)}\left[ v 			\mapsto \frac{\partial \Gamma^{-}}{\partial z}(v, z_{0}) \right]_{u}
	\]
	we first expand the part \crciii \ of the same equation as
	\begin{align*}
		&\overline{\nabla}_{E_{j}(u)} \left[ v \mapsto \frac{\partial \Gamma^{-}}{\partial z}(v, z_{0}) \right]_{u} = 				\overline{\nabla}_{E_{j}(u)} \left[ v \mapsto \nabla \frac{\partial h}{\partial z}(v, z_{0}) + \frac{\partial h}{\partial 			z}(v, z_{0})v \right]_{u}\\
		&= \overline{\nabla}_{E_{j}(u)} \left[ v \mapsto \nabla \frac{\partial h}{\partial z}(v, z_{0}) \right]_{u} + 				\overline{\nabla}_{E_{j}(u)} \left[ v \mapsto \frac{\partial h}{\partial z}(v, z_{0}) N_{v} \right]_{u}\\
		&= \overline{\nabla}_{E_{j}(u)} \left[ \nabla \frac{\partial h}{\partial z}(\cdot, z_{0}) \right]_{u} + E_{j} 				\left[ \frac{\partial h}{\partial z}(\cdot , z_{0}) \right] (u) \, u + \frac{\partial h}{\partial z} (u, z_{0}) 					\overline{\nabla}_{E_{j}(u)} N \, \bigr\rvert_{u} 
	\end{align*}
	and when we sum over $1 \leq j \leq n - 2$ the dot product with $E_{j}(u)$ we get
	\begin{align*}
		&\sum_{j=1}^{n-2} E_{j}(u) \cdot \overline{\nabla}_{E_{j}(u)} \left[ v \mapsto \frac{\partial \Gamma^{-}}{\partial 		z}(v, z_{0}) \right]_{u}\\
		&= \sum_{j-1}^{n-2} E_{j}(u) \cdot \overline{\nabla}_{E_{j}(u)} \left[ \nabla \frac{\partial h}{\partial z} (\cdot, 			z_{0}) \right]_{u} + \sum_{j=1}^{n-2} \frac{\partial h}{\partial z}(u, z_{0})\\
		&= \divg \left[ \nabla \frac{\partial h}{\partial z}(\cdot, z_{0}) \right](u) + (n - 2) \frac{\partial h}{\partial z}(u, 			z_{0})
		\\
		&= \Delta \frac{\partial h}{\partial z}(\cdot, z_{0}) \, \Bigr\rvert_{u} + (n - 2) \frac{\partial h}{\partial z}(u, z_{0}) = 		\left. \frac{\partial}{\partial z} \right\rvert_{z=z_{0}} \bigl[ \Delta h(\cdot, z) \, \bigr\rvert_{u} + (n - 2)h(u, z) \bigr].
	\end{align*}	
	Meanwhile, 
	\begin{align*}
		\tau \cdot \Gamma^{-}(u, z_{0}) &= \tau \cdot \bigl[ \nabla h(\cdot, z_{0}) \, \bigr\rvert_{u} + h(u, z_{0})u \bigr]\\
								&= \tau^{\tng}(u) \cdot \nabla h(\cdot, z_{0}) \, \bigr\rvert_{u} + (\tau \cdot u) h(u, 								      z_{0}).
	\end{align*}
	Therefore, the part \crciv \ of the equation \eqref{ES5.2:pde1} equals
	\begin{equation} \label{ES5.2:pde1;4}
	\begin{aligned}
			&\bigl[ \tau \cdot \Gamma^{-}(u, z_{0}) \bigr] \sum_{j=1}^{n-2} E_{j}(u) \cdot \overline{\nabla}_{E_{j}(u)}				\left[ v \mapsto \frac{\partial \Gamma^{-}}{\partial z}(v, z_{0}) \right]_{u}\\
			&= \tau^{\tng}(u) \cdot \left\{ \left. \frac{\partial}{\partial z} \right\rvert_{z=z_{0}} \bigl[ \Delta h(\cdot, z) \, 				\bigr\rvert_{u} + (n - 2)h(u, z) \bigr] \nabla h( \cdot, z_{0}) \, \bigr\rvert_{u}\right\}\\
			&\qquad +(\tau \cdot u) h(u, z_{0}) \left. \frac{\partial}{\partial z} \right\rvert_{z=z_{0}} \bigl[ \Delta h(\cdot, 				z) \, \bigr\rvert_{u} + (n - 2)h(u, z) \bigr] 
		\end{aligned}
	\end{equation}
	Then adding the equations \eqref{ES5.2:pde1;3} and \eqref{ES5.2:pde1;4} we can conclude that the part \crci \ of the 		equation \eqref{ES5.2:pde} equals
	
	\begin{equation} \label{ES5.2:pde1;fnl}
		\begin{aligned}
			&\divg \left[ v \mapsto \bigl( \tau \cdot \Gamma^{-}(v, z_{0}) \bigr) \frac{\partial \Gamma^{-}}{\partial z}(v, 				z_{0}) \right](u)\\
			&= \tau^{\tng}(u) \cdot \Biggl\{ \left. \frac{\partial}{\partial z} \right\rvert_{z=z_{0}} \bigl[ \Delta h(\cdot, z) \, 			\bigr\rvert_{u} + (n - 2) h(u, z) \bigr] \nabla h(\cdot, z_{0}) \, \bigr\rvert_{u}\\
			& \hspace{0.8cm}+ h(u, z_{0}) \nabla \frac{\partial h}{\partial z}( \cdot, z_{0}) \, \bigr\rvert_{u}\\
			&\hspace{0.8cm}+ \sum_{k=1}^{n-2} \left(\sum_{j=1}^{n-2} \left[ E_{j}(u) \cdot \nabla \frac{\partial h}				{\partial z}	(\cdot, z_{0}) \, \Bigr\rvert_{u} \right] E_{j}\Bigl( E_{k}\bigl[ h (\cdot, z_{0}) \bigr] \Bigr)(u)\right) 			E_{k}(u) \Biggr\}\\
			&+ (\tau \cdot u) \Biggl\{ h(u, z_{0}) \left. \frac{\partial}{\partial z} \right\rvert_{z=z_{0}} \bigl[ \Delta h(\cdot, 			z) \, \bigr\rvert_{u} + (n - 2) h(u, z) \bigr] \Biggr\}
		\end{aligned}
	\end{equation}
	
	Now we want to compute the part \crcii \ of the equation \eqref{ES5.2:pde}
	\begin{equation} \label{ES5.2:pde2}
		\divg \left[ v \mapsto \left( \frac{\partial \Gamma^{-}}{\partial z}(v, z_{0}) \cdot v \right) (\partial_{1}\Gamma^{-})			(v, z_{0}) \bigl[ \tau^{\tng}(v) \bigr] \right](u).
	\end{equation}
	Note that $v \in \mathbb{S}^{n-2} \mapsto \tau^{\tng}(v)$ is a vector field on $\mathbb{S}^{n-2}$ and $v \in \mathbb{S}	^{n-2} \mapsto \Gamma^{-}(v, z_{0})$ support parameterizes the ovaloid $\mathcal{O}(z_{0}) - c(z_{0})$, where for 		every $v \in \mathbb{S}^{n-2}$
	\[
		\tng_{\Gamma^{-}(v, z_{0})} \mathcal{O}(z_{0}) - c(z_{0}) = \tng_{v} \mathbb{S}^{n-2}.
	\]
	Thus we can conclude that $v \mapsto (\partial_{1}\Gamma^{-})(v, z_{0}) \bigl[ \tau^{\tng}(v) \bigr]$ is a vector field on $	\mathbb{S}^{n-2}$.
	\begin{align*}
		(\partial_{1}\Gamma^{-})(v, z_{0}) \bigl[ \tau^{\tng}(v) \bigr] &= \overline{\nabla}_{\tau^{\tng}(v)} \Gamma^{-}			(\cdot, z_{0}) \, \bigr\rvert_{v} = \overline{\nabla}_{\tau^{\tng}(v)} \bigl[ \nabla h(\cdot, z_{0}) + h(\cdot, z_{0}) N			\bigr] \, \bigr\rvert_{v}\\
		&= \overline{\nabla}_{\tau^{\tng}(v)} \nabla h(\cdot, z_{0}) \, \bigr\rvert_{v} + \overline{\nabla}_{\tau^{\tng}(v)} 			h(\cdot, z_{0})N \, \bigr\rvert_{v}\\
		&= \overline{\nabla}_{\tau^{\tng}(v)} \nabla h(\cdot, z_{0}) \, \bigr\rvert_{v} + \bigl[ \tau^{\tng}h(\cdot, z_{0}) 			\bigr](v) v + h(v, z_{0}) \overline{\nabla}_{\tau^{\tng}(v)} N \, \bigr\rvert_{v}\\
		&= \overline{\nabla}_{\tau^{\tng}(v)} \nabla h(\cdot, z_{0}) \, \bigr\rvert_{v} + \bigl[ \tau^{\tng}h(\cdot, z_{0}) 			\bigr](v) v + h(v, z_{0}) \tau^{\tng}(v).
	\end{align*}
	Since $(\partial_{1}\Gamma^{-})(v, z_{0}) \bigl[ \tau^{\tng}(v) \bigr] \in \tng_{v} \mathbb{S}^{n-2}$, $	\overline{\nabla}	_{\tau^{\tng}(v)} \nabla h(\cdot, z_{0}) \, \bigr\rvert_{v}$ decomposes as
	\[
		\overline{\nabla}_{\tau^{\tng}(v)} \nabla h(\cdot, z_{0}) \, \bigr\rvert_{v} = \nabla_{\tau^{\tng}(v)} \nabla h(\cdot, 			z_{0}) \, \bigr\rvert_{v} - \bigl[ \tau^{\tng}h(\cdot, z_{0}) \bigr](v)v
	\]
	and
	\begin{equation}
		\begin{aligned}
			(\partial_{1}\Gamma^{-})(v, z_{0}) \bigl[ \tau^{\tng}(v) \bigr] &= \Bigl( \nabla_{\tau^{\tng}(v)} \nabla h(\cdot, 			z_{0}) \, \bigr\rvert_{v} - \bigl[ \tau^{\tng}h(\cdot, z_{0}) \bigr](v)v \Bigr)\\ 
			&\quad +\bigl[ \tau^{\tng}h(\cdot, z_{0}) \bigr](v) v + h(v, z_{0}) \tau^{\tng}(v)\\
			&= \nabla_{\tau^{\tng}(v)} \nabla h(\cdot, z_{0}) \, \bigr\rvert_{v} + h(v, z_{0}) \tau^{\tng}(v).	
		\end{aligned}
	\end{equation}
	Therefore, the divergence in \eqref{ES5.2:pde2} can be expanded as
	\begin{equation} \label{ES5.2:pde2;1}
		\begin{aligned}
			&\sum_{j=1}^{n-2} E_{j}(u) \cdot \overline{\nabla}_{E_{j}(u)} \left[ v \mapsto \left( \frac{\partial 					\Gamma^{-}}{\partial z}(v, z_{0}) \cdot v \right) \Bigl(\nabla_{\tau^{\tng}(v)} \nabla h(\cdot, z_{0}) \, \bigr				\rvert_{v} + h(v, z_{0}) \tau^{\tng}(v) \Bigr) \right]_{u}\\
		 	&= \underbrace{\sum_{j=1}^{n-2} E_{j}(u) \cdot \Biggl\{ E_{j} \left[ v \mapsto \frac{\partial 						\Gamma^{-}}{\partial z} (v, z_{0}) \cdot v \right](u) \Bigl( \nabla_{\tau^{\tng}(u)} \nabla h(\cdot, z_{0}) \, 				\bigr\rvert_{u}+ h(u, z_{0}) \tau^{\tng}(u) \Bigr) \Biggr\}}_{\textstyle \crcvi}\\
			&+ \underbrace{\sum_{j=1}^{n-2} E_{j}(u) \cdot \Biggl\{ \left( \frac{\partial \Gamma^{-}}{\partial z}(u, 				z_{0}) \cdot u \right) \underbrace{\nabla_{E_{j}(u)} \nabla_{\tau^{\tng}} \nabla h(\cdot, z_{0}) \, \bigr					\rvert_{u}}_{\textstyle \crci} \Biggr\}}_{\textstyle \crcii + \crciii + \crciv + \crcv}\\
			&+ \underbrace{\sum_{j=1}^{n-2} E_{j}(u) \cdot \Biggl\{ \left( \frac{\partial \Gamma^{-}}{\partial z}(u, 				z_{0}) \cdot u\right) \underbrace{\overline{\nabla}_{E_{j}(u)} \bigl[ v \mapsto h(v, z_{0}) \tau^{\tng}(v) \bigr] 			\, \bigr\rvert_{u}}_{\textstyle \crcvii} \Biggr\}}_{\textstyle \crcviii}
		\end{aligned}
	\end{equation}
	Using the fact $\nabla N \, \bigr\rvert_{u} = \mathrm{id}_{\tng_{u}\mathbb{S}^{n-2}}$ and Gauss equations \cite[p.194]		{jJ11} we get
	\begin{equation}
		\begin{aligned}
			\nabla_{X} \nabla_{Y} Z - \nabla_{Y} \nabla_{X} Z - \nabla_{[X, Y]} Z &= - (S_{N}X  \cdot Z) S_{N}Y + 																		   (S_{N}Y \cdot Z) S_{N}X\\
																   &= - (X \cdot Z)Y + (Y \cdot Z)X
		\end{aligned}
	\end{equation}
	for all $X$, $Y$, $Z \in \varkappa(\mathbb{S}^{n-2})$. So if we let $X = E_{j}$, $Y = \tau^{\tng}$, and $Z = \nabla 		h(\cdot, z_{0})$ then the part \crci \ of the equation \eqref{ES5.2:pde2;1} can be written as
	\begin{equation} \label{ES5.2:pde2;2}
		\begin{split}
			\nabla_{E_{j}} \nabla_{\tau^{\tng}} \nabla h( \cdot, z_{0}) &= \underbrace{\nabla_{\tau^{\tng}} 						\nabla_{E_{j}} \nabla h(\cdot, z_{0})}_{\textstyle \crciii} + \underbrace{\nabla_{\underbrace{[E_{j}, 					\tau^{\tng}]}_{\textstyle \crci}} \nabla h(\cdot, z_{0})}_{\textstyle \crcii}\\
			&\underbrace{- \bigl[ S_{N}E_{j} \cdot \nabla h(\cdot, z_{0}) \bigr] S_{N}\tau^{\tng}}_{\textstyle \crciv}{} + 			\underbrace{\bigl[S_{N}\tau^{\tng} \cdot \nabla h(\cdot, z_{0})\bigr] S_{N}E_{j}}_{\textstyle \crcv}
		\end{split}
	\end{equation}
	
	The parts \crcii, \crciii, \crciv, and \crcv \ of the equation \eqref{ES5.2:pde2;2} yield the corresponding parts of the equation 	\eqref{ES5.2:pde2;1}. We will first compute the parts of the equation \eqref{ES5.2:pde2;2}. The part \crci \ of the equation 		\eqref{ES5.2:pde2;2} can be calculated as
	\begin{align*}
		&[E_{j}, \tau^{\tng}](u) = \nabla_{E_{j}(u)} \tau^{\tng} \, \bigr\rvert_{u} - \nabla_{\tau^{\tng}(u)} E_{j} \, \bigr			\rvert_{u} = \overline{\nabla}_{E_{j}(u)} \tau^{\tng} \, \bigr\rvert_{u} - \overline{\nabla}_{\tau^{\tng}(u)} E_{j} 			\bigr\rvert_{u}\\
		&= \overline{\nabla}_{E_{j}(u)} \bigl[ v \mapsto \tau - (\tau \cdot v)v \bigr] \, \bigr\rvert_{u} -						\bigl[ \underbrace{\nabla_{\tau^{\tng}(u)} E_{j} \, \bigr\rvert_{u}}_{= 0}{} + \bigl( \overline{\nabla}_{\tau^{\tng}			(u)} E_{j} \, \bigr\rvert_{u} \cdot u\bigr)u \bigr]
	\end{align*}
	\begin{align*}
		&= - \overline{\nabla}_{E_{j}(u)} \bigl( \tau \cdot N \bigr)N \bigr\rvert_{u} - \bigl( \overline{\nabla}_{\tau^{\tng}			(u)} E_{j} \, \bigr\rvert_{u} \cdot u \bigr)u\\
		&= - \Bigl( \bigl[ \tau \cdot \overline{\nabla}_{E_{j}(u)} N \, \bigr\rvert_{u}\bigr]u + (\tau \cdot u) 						\overline{\nabla}_{E_{j}(u)} N \, \bigr\rvert_{u} \Bigr)\\
		&\hspace{5.5cm} - \Bigl( \underbrace{\overline{\nabla}_{\tau^{\tng}(u)} E_{j} \cdot N \, \bigr\rvert_{u}}_{=0}{} - 		E_{j}(u) \cdot \overline{\nabla}_{\tau^{\tng}(u)} N \, \bigr\rvert_{u} \Bigr)u\\
		&= - \Bigl( \bigl[ \tau \cdot E_{j}(u) \bigr] u + (\tau \cdot u) E_{j}(u) \Bigr) + \bigl[ E_{j}(u) \cdot \tau^{\tng}(u)
 		\bigr]u =  - (\tau \cdot u) E_{j}(u).	
 	\end{align*}
	The parts \crciv \ and \crcv \ of the equation \eqref{ES5.2:pde2;2} can be written as
	\begin{align*}
		- \Bigl( S_{N} \bigl[ E_{j}(u) \bigr] \cdot \nabla h(\cdot, z_{0}) \, \bigr\rvert_{u} \Bigr) S_{N} \bigl[ \tau^{\tng}(u) 			\bigr] &= - \bigl( E_{j}(u) \cdot \nabla h(\cdot, z_{0}) \, \bigr\rvert_{u} \bigr) \tau^{\tng}(u)\\
		\Big( S_{N} \bigl[ \tau^{\tng}(u) \bigr] \cdot \nabla h(\cdot, z_{0}) \, \bigr\rvert_{u} \Bigr) S_{N} \bigl[ E_{j}(u)
 		\bigr] &= \bigl( \tau^{\tng}(u) \cdot \nabla h(\cdot, z_{0}) \, \bigr\rvert_{u} \bigr) E_{j}(u).	
	\end{align*}
	Using the part  \crci, the part \crcii \ of the equation \eqref{ES5.2:pde2;2} can be written as
	\begin{align*}
		\nabla_{[E_{j}, \tau^{\tng}]} \nabla h(\cdot, z_{0}) \, \bigr\rvert_{u} &= - (\tau \cdot u) \nabla_{E_{j}(u)} \nabla 			h(\cdot, z_{0}) \, \bigr\rvert_{u}\\
														     &= - (\tau \cdot u) \nabla_{E_{j}(u)} 				\left( \sum_{k=1}^{n-2} E_{k} \bigl[ h(\cdot, z_{0}) \bigr] E_{k} \right)_{u}.
	\end{align*}
	Recall that $u \in \mathbb{S}^{n-2} \mapsto \Gamma^{-}(u, z_{0})$ is the support parameterization of the ovaloid $			\mathcal{O}(z_{0}) - c(z_{0})$ given by
	\begin{equation} \label{RS5.2:pde}
		\Gamma^{-}(u, z_{0}) = \nabla h(\cdot, z_{0}) \, \bigr\rvert_{u} + h(u, z_{0})u
	\end{equation}
	and hence the part \crcii \ of the equation \eqref{ES5.2:pde2;1} equals
	\begin{align}
		&\sum_{j=1}^{n-2} E_{j}(u) \cdot \left[ \frac{\partial h}{\partial z}(u, z_{0}) \nabla_{[E_{j}, \tau^{\tng}](u)} \nabla 		h(\cdot, z_{0}) \, \bigr\rvert_{u} \right] \notag\\
		&= \frac{\partial h}{\partial z}(u, z_{0}) \sum_{j=1}^{n-2} E_{j}(u) \cdot \nabla_{-(\tau \cdot u)E_{j}(u)} \nabla 			h(\cdot, z_{0}) \, \bigr\rvert_{u} \notag\\
		&= - (\tau \cdot u) \frac{\partial h}{\partial z}(u, z_{0}) \sum_{j=1}^{n-2} E_{j}(u) \cdot \nabla_{E_{j}(u)} \nabla 			h(\cdot, z_{0}) \, \bigr\rvert_{u} \notag\\
		&= (\tau \cdot u) \left[ - \Delta h(\cdot, z_{0}) \, \bigr\rvert_{u} \frac{\partial h}{\partial z}(u, z_{0})\right]. 				\label{ES5.2:pde2;1p2}
	\end{align}
	Using the equation \eqref{RS5.2:pde} we can easily compute
	\begin{equation}
		\begin{aligned}
			\frac{\partial \Gamma^{-}}{\partial z}(u, z_{0}) \cdot u &= \frac{\partial h}{\partial z}(u, z_{0})\\
			E_{j} \left[ v \mapsto \frac{\partial \Gamma^{-}}{\partial z}(v, z_{0}) \cdot v \right](u) &= E_{j}						\left[ \frac{\partial h}{\partial z}(\cdot, z_{0}) \right](u).
		\end{aligned}
	\end{equation}
	The part \crcvii \ of the equation \eqref{ES5.2:pde2;1} is equal to
	\begin{align*}
		&\overline{\nabla}_{E_{j}(u)} \bigl[ v \mapsto h(v, z_{0}) \tau^{\tng}(v) \bigr] \, \bigr\rvert_{u}\\
		& = E_{j} \bigl[ h(\cdot, z_{0}) \bigr](u) \tau^{\tng}(u) + h(u, z_{0}) \overline{\nabla}_{E_{j}(u)} \bigl[v \mapsto 			\tau - (\tau \cdot v)v \bigr] \, \bigr\rvert_{u}\notag\\
		&= E_{j} \bigl[ h(\cdot, z_{0}) \bigr](u) \tau^{\tng}(u) + h(u, z_{0}) \bigl[ - \overline{\nabla}_{E_{j}(u)} (\tau \cdot 			N)N \, \bigr\rvert_{u} \bigr] \notag\\
		&= E_{j} \bigl[ h(\cdot, z_{0}) \bigr](u) \tau^{\tng}(u) - h(u, z_{0}) \Bigl[ \bigl( \tau \cdot \overline{\nabla}_{E_{j}			(u)}N \, \bigr\rvert_{u} \bigr)u + (\tau \cdot u) \overline{\nabla}_{E_{j}(u)} N \, \bigr\rvert_{u} \Bigr] \notag\\
		&= E_{j} \bigl[ h(\cdot, z_{0}) \bigr](u) \tau^{\tng}(u) - h(u, z_{0})  \bigl[ \tau \cdot E_{j}(u) \bigr]u - h(u, z_{0})			(\tau \cdot u) E_{j}(u).	
	\end{align*}	
	Using the part \crcvii, part \crcviii \ of the equation \eqref{ES5.2:pde2;1} can be computed as
	\begin{align}
		&\sum_{j=1}^{n-2} E_{j}(u) \cdot \left\{ \left( \frac{\partial \Gamma^{-}}{\partial z}(u, z_{0}) \cdot u \right) 				\overline{\nabla}_{E_{j}(u)} \bigl[ v \mapsto h(v, z_{0}) \tau^{\tng}(v) \bigr] \, \bigr\rvert_{u} \right\} \notag\\
		&= \frac{\partial h}{\partial z}(u, z_{0}) \left\{ \sum_{j=1}^{n-2} E_{j} \bigl[ h(\cdot, z_{0}) \bigr](u) 
		\bigl( \tau^{\tng}(u) \cdot E_{j}(u) \bigr) - (n - 2)(\tau \cdot u) h(u, z_{0}) \right\} \notag\\
		&= \tau^{\tng}(u) \cdot \left\{ \sum_{j=1}^{n-2} E_{j} \bigl[ h(\cdot, z_{0}) \bigr](u) \frac{\partial h}{\partial z}(u, 			z_{0}) E_{j}(u) \right\} \notag\\
		&\hspace{6cm} + (\tau \cdot u) \left[ - (n - 2) h(u, z_{0}) \frac{\partial h}{\partial z}(u, z_{0}) \right] \notag\\
		&= \tau^{\tng}(u) \left\{ \frac{\partial h}{\partial z}(u, z_{0}) \nabla h(\cdot, z_{0}) \, \bigr\rvert_{u} \right\} + (\tau 		\cdot u) \left[ - (n - 2) h(u, z_{0}) \frac{\partial h}{\partial z}(u, z_{0}) \right]. \label{ES5.2:pde2;1P3}
	\end{align}
	The part \crciii \ of the equation \eqref{ES5.2:pde2;1} equals
	\begin{align}
		&\sum_{j=1}^{n-2} E_{j}(u) \cdot \left\{ \frac{\partial h}{\partial z}(u, z_{0}) \nabla_{\tau^{\tng}(u)} 					\nabla_{E_{j}}\nabla h( \cdot, z_{0}) \, \bigr\rvert_{u} \right\} \notag\\
		&= \frac{\partial h}{\partial z}(u, z_{0}) \sum_{j=1}^{n-2} E_{j}(u) \cdot \nabla_{\tau^{\tng}(u)} \nabla_{E_{j}} 			\nabla h(\cdot, z_{0}) \, \bigr\rvert_{u} \notag\\
		\intertext{since $\nabla_{\tau^{\tng}(u)} E_{j} \, \bigr\rvert_{u} = 0$ and $\nabla_{\tau^{\tng}(u)} E_{j} \, \bigr			\rvert_{u} \cdot \nabla_{E_{j}(u)}\nabla h(\cdot, z_{0}) \, \bigr\rvert_{u} = 0$ we have}
		&= \frac{\partial h}{\partial z}(u, z_{0}) \sum_{j=1}^{n-2} \nabla_{\tau^{\tng}(u)} \bigl[ E_{j} \cdot 					\nabla_{E_{j}} \nabla h(\cdot, z_{0}) \bigr] \, \bigr\rvert_{u} \notag\\
		&= \frac{\partial h}{\partial z}(u, z_{0}) \nabla_{\tau^{\tng}(u)} \left( \sum_{j=1}^{n-2} E_{j} \cdot 					\overline{\nabla}_{E_{j}} \nabla h(\cdot, z_{0}) \right)_{u} \notag\\
		&= \frac{\partial h}{\partial z}(u, z_{0}) \nabla_{\tau^{\tng}(u)} \Delta h(\cdot, z_{0}) \, \bigr\rvert_{u} \notag\\
		&= \frac{\partial h}{\partial z}(u, z_{0}) \nabla_{\tau^{\tng}(u)} \bigr[ v \mapsto \Delta h(v, z_{0}) \bigr] \, \bigr			\rvert_{u} \notag\\ 		
		&= \tau^{\tng}(u) \cdot \left\{ \frac{\partial h}{\partial z}(u, z_{0}) \nabla  \bigl[ \Delta h(\cdot, z_{0}) \bigr] \, \bigr			\rvert_{u} \right\}. \label{ES5.2:pde2;1p3}
	\end{align}
	The part \crciv \ of the equation \eqref{ES5.2:pde2;1} can be written as
	\begin{align}
		&\sum_{j=1}^{n-2} E_{j}(u) \cdot \left\{ - \frac{\partial h}{\partial z}(u, z_{0}) \bigl[ E_{j}(u) \cdot \nabla h(\cdot, 		z_{0}) \, \bigr\rvert_{u} \bigr] \tau^{\tng}(u) \right\} \notag\\
		&= - \frac{\partial h}{\partial z}(u, z_{0}) \sum_{j=1}^{n-2} \bigl[ E_{j}(u) \cdot \nabla h(\cdot, z_{0}) \, \bigr			\rvert_{u} \bigr] E_{j}(u) \cdot \tau^{\tng}(u) \notag\\
		&= \tau^{\tng}(u) \cdot  \left\{ - \frac{\partial h}{\partial z}(u, z_{0}) \nabla h(\cdot, z_{0}) \, \bigr\rvert_{u} \right\}. 		\label{ES5.2:pde2;1p4}
	\end{align}
	The part \crcv \ of the equation \eqref{ES5.2:pde2;1} is equal to
	\begin{align}
		&\sum_{j=1}^{n-2} E_{j}(u) \cdot \left\{ \frac{\partial h}{\partial z}(u, z_{0}) \bigl[ \tau^{\tng}(u) \cdot \nabla 			h(\cdot, z_{0}) \, \bigr\rvert_{u} \bigr] E_{j}(u)\right\} \notag\\
		&= \frac{\partial h}{\partial z}(u, z_{0}) \bigl[ \tau^{\tng}(u) \cdot \nabla h(\cdot, z_{0}) \, \bigr\rvert_{u} \bigr] 			\sum_{j=1}^{n-2} E_{j}(u) \cdot E_{j}(u) \notag\\
		&= \tau^{\tng}(u) \cdot \left\{(n - 2) \frac{\partial h}{\partial z}(u, z_{0}) \nabla h(\cdot, z_{0}) \, \bigr\rvert_{u} 			\right\}. \label{ES5.2:pde2;1p5}
	\end{align}
	By summing the equations \eqref{ES5.2:pde2;1p2}, \eqref{ES5.2:pde2;1p3}, \eqref{ES5.2:pde2;1p4}, and 				\eqref{ES5.2:pde2;1p5} we get
	\begin{equation} \label{ES5.2:pde2;1P2}
		\begin{aligned}
		 	&\sum_{j=1}^{n-2} E_{j}(u) \cdot \left\{ \left( \frac{\partial \Gamma^{-}}{\partial z}(u, z_{0}) \cdot u\right) 				\nabla_{E_{j}(u)}\nabla_{\tau^{\tng}} \nabla h(\cdot, z_{0}) \, \bigr\rvert_{u} \right\}\\
			&= \tau^{\tng}(u) \cdot \Bigl\{ \frac{\partial h}{\partial z}(u, z_{0}) \nabla \bigl[ \Delta h(\cdot, z_{0}) \bigr] \, 			\bigr\rvert_{u} - \frac{\partial h}{\partial z}(u, z_{0}) \nabla h(\cdot, z_{0}) \, \bigr\rvert_{u}\\
			&\phantom{\tau^{\tng}(u) \cdot \Bigl\{ \frac{\partial h}{\partial z}(u, z_{0}) \nabla \bigl[ \Delta h(\cdot, z_{0}) 			\bigr] \, \bigr\rvert_{u} - {}}  + (n -2) \frac{\partial h}{\partial z}(u, z_{0}) \nabla h(\cdot, z_{0}) \, \bigr					\rvert_{u} \Bigr\}\\
			& \qquad + (\tau \cdot u) \left\{ - \frac{\partial h}{\partial z}(u, z_{0}) \Delta h(\cdot, z_{0}) \, \bigr\rvert_{u} 				\right\}
		\end{aligned}
	\end{equation}	
	The part \crcvi  \ of the equation \eqref{ES5.2:pde2;1} can be computed as
	\begin{align}
		&\sum_{j=1}^{n-2} E_{j}(u) \cdot \Biggl\{ E_{j} \left[ v \mapsto \frac{\partial 									\Gamma^{-}}{\partial z} (v, z_{0}) \cdot v \right](u) \Bigl( \nabla_{\tau^{\tng}(u)} \nabla h(\cdot, z_{0}) \, 				\bigr\rvert_{u}+ h(u, z_{0}) \tau^{\tng}(u) \Bigr) \Biggr\} \notag\\
		&= \sum_{j=1}^{n-2} E_{j}(u) \cdot \left\{ E_{j}\left[  \frac{\partial h}{\partial z}(\cdot, z_{0}) \right](u)  \ 				\nabla_{\tau^{\tng}(u)} \nabla h(\cdot, z_{0}) \, \bigr\rvert_{u}\right\}\notag\\
		& \hspace{4cm} + \sum_{j=1}^{n-2} E_{j}(u) \cdot \left\{ E_{j} \left[ \frac{\partial h}{\partial z}(\cdot, z_{0}) 			\right](u) \ 	h(u, z_{0}) \tau^{\tng}(u) \right\}\notag
	\end{align}
	\begin{align}
		&= \sum_{j=1}^{n-2} E_{j} \left[ \frac{\partial h}{\partial z}(\cdot, z_{0}) \right](u)  \ \Bigl\{ E_{j}(u) \cdot 				\nabla_{\tau^{\tng}(u)} \nabla h(\cdot, z_{0}) \, \bigr\rvert_{u} \Bigr\}\notag\\
		&\hspace{4cm} + \tau^{\tng}(u) \cdot \sum_{j=1}^{n-2} h(u, z_{0}) E_{j} \left[ \frac{\partial h}{\partial z}(\cdot, 			z_{0}) \right](u) \ E_{j}(u) \notag\\
		&= \sum_{j=1}^{n-2} E_{j} \left[ \frac{\partial h}{\partial z}(\cdot, z_{0}) \right](u)  \ \nabla_{\tau^{\tng}(u)} 			\bigl[ E_{j} \cdot \nabla h(\cdot, z_{0}) \bigr] \, \bigr\rvert_{u}  + \tau^{\tng}(u) \cdot h(u, z_{0}) \nabla \frac{\partial 		h}{\partial z}(\cdot, z_{0}) \, \bigr\rvert_{u} \notag\\
		&= \sum_{j=1}^{n-2} E_{j} \left[ \frac{\partial h}{\partial z}(\cdot, z_{0}) \right](u) \ \Bigl\{ \tau^{\tng}(u) \cdot 			\nabla \bigl[ E_{j} \cdot \nabla h(\cdot, z_{0}) \bigr] \, \bigr\rvert_{u} \Bigr\} \notag\\
		&\hspace{4cm} + \tau^{\tng}(u) \cdot h(u, z_{0}) \nabla \frac{\partial h}{\partial z}(\cdot, z_{0}) \, \bigr\rvert_{u}			\notag\\		
		&= \tau^{\tng}(u) \cdot \left\{ \sum_{j=1}^{n-2} E_{j} \left[ \frac{\partial h}{\partial z}(\cdot, z_{0}) \right](u) \ 			\nabla \bigl[ E_{j} \cdot \nabla h(\cdot, z_{0}) \bigr] \, \bigr\rvert_{u} + h(u, z_{0}) \nabla \frac{\partial h}{\partial z}		(\cdot, z_{0}) \, \bigr\rvert_{u} \right\}. \label{ES5.2:pde2;1P1}
	\end{align}
	By summing the equations \eqref{ES5.2:pde2;1P1}, \eqref{ES5.2:pde2;1P2} and \eqref{ES5.2:pde2;1P3} we find that the 		divergence at $u \in \mathbb{S}^{n-2}$
	\[
		\divg \left[ v \mapsto \left( \frac{\partial \Gamma^{-}}{\partial z}(v, z_{0}) \cdot v \right) (\partial_{1}\Gamma^{-})			(v, z_{0}) \bigl[ \tau^{\tng}(v) \bigr] \right] (u)
	\]
	equals
	\begin{equation} \label{ES5.2:pde2;fnl}
		\begin{aligned}
			&\tau^{\tng}(u) \cdot \biggl\{ \frac{\partial h}{\partial z}(u, z_{0}) \nabla \bigl[ \Delta h(\cdot, z_{0}) + (n-2) 				h(\cdot, z_{0}) \bigr] \, \bigr\rvert_{u} + h(u, z_{0}) \nabla \frac{\partial h}{\partial z}(\cdot, z_{0}) \, \bigr				\rvert_{u}\\
			&\phantom{\tau^{\tng}(u) \cdot \biggl\{ \frac{\partial h}{\partial z}(u, z_{0}) \nabla \bigl[ \Delta h(\cdot, 					z_{0})}+ \sum_{j=1}^{n-2} E_{j} \left[ \frac{\partial h}{\partial z}(\cdot, z_{0}) \right](u) \ \nabla \bigl[ E_{j} 			\cdot \nabla h(\cdot, z_{0}) \bigr] \, \bigr\rvert_{u} \biggr\}\\
			&\quad + (\tau \cdot u) \left\{  - \bigl[ \Delta h(\cdot, z_{0}) \, \bigr\rvert_{u} + (n -2)h(u, z_{0}) \bigr] 					\frac{\partial h}{\partial z}(u, z_{0})\right\}	.	
		\end{aligned}
	\end{equation}
	Finally the value of the function $f_{\tau}$ at $(u, z_{0})$ can be found by adding the equations \eqref{ES5.2:pde1;fnl} 		and \eqref{ES5.2:pde2;fnl}
	\begin{equation} \label{ES5.2:pde;fnl}
		\begin{aligned}
			&f_{\tau}(u, z_{0}) = \divg \left[ \left. \frac{\partial}{\partial \epsilon} \right\rvert_{\epsilon=0} c_{\epsilon}				( \cdot, z_{0}) \right](u)\\
			&= \divg\left[ v \mapsto \bigl( \tau \cdot \Gamma^{-}(v, z_{0}) \bigr) \frac{\partial \Gamma^{-}}{\partial z}(v, 			z_{0}) \right](u)\\
			&\hspace{4cm} + \divg \left[ v \mapsto \left( \frac{\partial \Gamma^{-}}{\partial z}(v, z_{0}) \cdot v \right)				(\partial_{1}\Gamma^{-})(v, z_{0}) \bigl[ \tau^{\tng}(u) \bigr] \right]
		\end{aligned}
	\end{equation}
	\begin{equation*}
		\begin{aligned}
			&= \tau^{\tng}(u) \cdot \Biggl\{ \left. \frac{\partial}{\partial z} \right\rvert_{z=z_{0}} \bigl[ \Delta h(\cdot, 				z_{0}) + (n -2)h(u, z) \bigr] \nabla h(\cdot, z_{0}) \, \bigr\rvert_{u}\\
			&\hspace{1.8cm} +2 h(u, z_{0}) \nabla \frac{\partial h}{\partial z}(u, z_{0}) \, \Bigr\rvert_{u} + 						\underbrace{\sum_{j=1}^{n-2} E_{j} \left[ \frac{\partial h}{\partial z}(\cdot, z_{0}) \right](u) \ \nabla 					\bigl[ E_{j} \cdot \nabla h(\cdot, z_{0}) \bigr](u)}_{\textstyle \crci} \\
			&\hspace{3.8cm} + \underbrace{\sum_{k=1}^{n-2} \left( \sum_{j=1}^{n-2} E_{j} \left[ \frac{\partial h}					{\partial z}(\cdot, z_{0})\right](u) \ E_{j} \bigl[ E_{k}h(\cdot, z_{0}) \bigr](u) \right)E_{k}(u)}_{\textstyle 				\crcii}\\
			&\hspace{3.8cm} + \frac{\partial h}{\partial z}(u, z_{0}) \nabla \bigl[ \Delta h(\cdot, z_{0}) + (n -2) h(\cdot, 				z_{0}) \bigr] \, \bigr\rvert_{u} \Biggr\}
		\end{aligned}
	\end{equation*}
	\begin{multline*}
		+ (\tau \cdot u)  \, \Biggl\{h(u, z_{0}) \left. \frac{\partial}{\partial z} \right\rvert_{z=z_{0}} \bigl[ \Delta 					h(\cdot, z_{0}) \, \bigr\rvert_{u} + (n -2) h(u, z) \bigr]\\
		 - \bigl[ \Delta h(\cdot, z_{0}) \, \bigr\rvert_{u} + (n - 2) h(u, z_{0}) \bigr] \frac{\partial h}{\partial z}(u, z_{0}) 			\Biggr\}
	\end{multline*}
	
	Now we want to expand the part \crci \ of the equation \eqref{ES5.2:pde;fnl} and show that it equals the part \crcii.
	
	\begin{align*}
		&\sum_{j=1}^{n-2} E_{j} \left[ \frac{\partial h}{\partial z}(\cdot, z_{0}) \right](u)  \ \nabla \bigl[ E_{j} \cdot \nabla 			h(\cdot, z_{0}) \bigr](u)\\
		&= \sum_{j=1}^{n-2} E_{j} \left[ \frac{\partial h}{\partial z}(\cdot, z_{0}) \right](u) \sum_{k=1}^{n-2} E_{k} 			\bigl[ E_{j} \cdot \nabla h(\cdot, z_{0}) \bigr](u) E_{k}(u)\\
		&= \sum_{j,k=1}^{n-2} E_{j} \left[ \frac{\partial h}{\partial z}(\cdot, z_{0}) \right](u) \Bigl\{							\underbrace{\nabla_{E_{k}(u)}	E_{j} \, \bigr\rvert_{u}}_{=0} \cdot \nabla h(\cdot, z_{0})\\
		&\hspace{5cm} + E_{j}(u) \cdot \nabla_{E_{k}(u)} \nabla h(\cdot, z_{0}) \, \bigr\rvert_{u}\Bigr\} E_{k}(u)\\
		&= \sum_{j,k=1}^{n-2} E_{j} \left[ \frac{\partial h}{\partial z}(\cdot, z_{0}) \right](u) \left\{ E_{j}(u) \cdot 				\sum_{l=1}^{n-2} \nabla_{E_{k}(u)} \Bigl( E_{l} \bigl[ h(\cdot, z_{0}) \bigr] E_{l} \Bigr)_{u} \right\} E_{k}(u)\\
		&= \sum_{j,k,l=1}^{n-2} E_{j} \left[ \frac{\partial h}{\partial z}(\cdot, z_{0})\right](u) \biggl\{ E_{j}(u) \cdot 			\Bigl[ E_{k} \bigl[ E_{l} (h (\cdot, z_{0})) \bigr](u) \ E_{l}(u)\\
		&\hspace{5cm} + E_{l} \bigl[h(\cdot, z_{0})\bigr](u) \underbrace{\nabla_{E_{k}(u)}E_{l} \, \bigr\rvert_{u}}				_{=0}\Bigr] \biggr\} E_{k}(u)\\
		&=\sum_{j,k=1}^{n-2} E_{j} \left[ \frac{\partial h}{\partial z}(\cdot, z_{0}) \right](u) \left\{ \sum_{l=1}^{n-2} 			E_{j}(u) \cdot \Bigl[ E_{k} \bigl[ E_{l}(h (\cdot, z_{0})) \bigr](u) \Bigr] E_{l}(u)\right\} E_{k}(u)
	\end{align*}
	\begin{align*}
		&\mspace{-58mu} = \sum_{j,k=1}^{n-2} E_{j} \left[ \frac{\partial h}{\partial z}(\cdot, z_{0}) \right](u) \ E_{k} 			\bigl[ E_{j}(h(\cdot, z_{0})) \bigr](u) \, E_{k}(u)\\
		&\mspace{-58mu} = \sum_{k=1}^{n-2} \left\{ \sum_{j=1}^{n-2} E_{j} \left[ \frac{\partial h}{\partial z}( \cdot, 			z_{0}) \right](u) \ E_{k} \bigl[ E_{j}(h(\cdot, z_{0})) \bigr](u) \right\} E_{k}(u)\\
		& \mspace{-58mu} = \sum_{k=1}^{n-2} \left\{ \sum_{j=1}^{n-2} E_{j} \left[ \frac{\partial h}{\partial z}( \cdot, 			z_{0}) \right](u) \ E_{j} \bigl[ E_{k}(h(\cdot, z_{0})) \bigr](u) \right\} E_{k}(u).
	\end{align*}
	The change of order of differentiation at the last step is permissible because for every smooth function $f \in C^{\infty}		(S^{n-2})$, $u \in \mathbb{S}^{n-2}$ and for every $j$, $k = 1, \dotsc n-2$
	\begin{equation} \label{ES5.2:orddif}
		\begin{aligned}
			\bigl( [E_{j}, E_{k}] f \bigr)(u) &= \bigl[ \nabla_{E_{j}(u)} E_{k} \, \bigr\rvert_{u} - \nabla_{E_{k}(u)} E_{j} 			\, \bigr\rvert_{u} \bigr] f = 0 \\
			E_{j}(E_{k}f)(u) &- E_{k}(E_{j}f)(u) = 0
		\end{aligned}
	\end{equation}
	So we can conclude the parts \crci \ and \crcii \ of the equation \eqref{ES5.2:pde;fnl} are equal and now we want to show 		that they are also equal to 
	\begin{equation} \label{ES5.2:pde2;3}
		\nabla_{\textstyle \nabla \left. \frac{\partial h}{\partial z}(\cdot, z_{0}) \right\rvert_{u}} \nabla h(\cdot, z_{0}) \, \bigr		\rvert_{u}.
	\end{equation}
	\begin{align*}
		&\nabla_{\textstyle \nabla \left. \frac{\partial h}{\partial z}(\cdot, z_{0}) \right\rvert_{u}} \nabla h(\cdot, z_{0}) \, 			\bigr	\rvert_{u} = \sum_{j=1}^{n-2} \left[ E_{j}(u) \cdot \nabla \frac{\partial h}{\partial z}(\cdot, z_{0}) \, \Bigr			\rvert_{u} \right] \nabla_{E_{j}(u)} \bigl[ \nabla h(\cdot, z_{0}) \bigr] \, \bigr\rvert_{u}\\
		&= \sum_{j=1}^{n-2}E_{j} \left[ \frac{\partial h}{\partial z}(\cdot, z_{0}) \right] (u) \sum_{k=1}^{n-2} 					\nabla_{E_{j}(u)} \Bigl( v \mapsto E_{k} \bigl[ h(\cdot, z_{0}) \bigr](v)  \, E_{k}(v) \Bigr)_{u} 
	\end{align*}
	\begin{align*}
		&= \sum_{j,k=1}^{n-2} E_{j} \left[ \frac{\partial h}{\partial z}(\cdot, z_{0}) \right](u) \Bigl\{ E_{j} \Bigl[ E_{k} 			\bigl[ h(\cdot, z_{0}) \bigr] \Big](u)  \, E_{k}(u) + E_{k} \bigl[ h(\cdot, z_{0}) \bigr](u)  \underbrace{\nabla_{E_{j}			(u)} E_{k} \, \bigr\rvert_{u}}_{=0}\Bigr\}\\
		&= \sum_{j,k=1}^{n-2} E_{j} \left[ \frac{\partial h}{\partial z}(\cdot, z_{0}) \right](u) \ E_{j} \Bigl[ E_{k} 				\bigl[ h(\cdot, z_{0}) \bigr] \Big](u) \, E_{k}(u)\\
		&= \sum_{k=1}^{n-2} \left\{ \sum_{j=1}^{n-2} E_{j} \left[ \frac{\partial h}{\partial z}(\cdot, z_{0}) \right] (u) \  			E_{j} \Bigl[ E_{k} \bigl[ h(\cdot, z_{0}) \bigr] \Bigr] (u) \right\} E_{k}(u).
	\end{align*}
	
	Therefore, by replacing the parts \crci \ and \crcii \ of the equation \eqref{ES5.2:pde;fnl} by their common value 			\eqref{ES5.2:pde2;3} we get 
	
	\begin{equation} \label{ES5.2:pde;fnll}
		\begin{aligned}
			0 &= f_{\tau}(u, z_{0}) = \tau^{\tng}(u)  \cdot \biggl\{ \frac{\partial h}{\partial z}(u, z_{0}) \nabla \bigl[ \Delta 				h(\cdot, z_{0}) + (n -2) h(\cdot, z_{0}) \bigr] \, \bigr\rvert_{u}\\
			   &\hspace{3.3cm}+ \left. \frac{\partial}{\partial z} \right\rvert_{z=z_{0}} \bigl[ \Delta h(\cdot, 						       z_{0}) + (n -2)h(u, z) \bigr] \nabla h(\cdot, z_{0}) \, \bigr\rvert_{u}\\
			   &\hspace{3.3cm}+2 \left( h(u, z_{0}) \nabla \frac{\partial h}{\partial z}(\cdot, z_{0}) \, \Bigr\rvert_{u} + 				      \nabla_{\textstyle \nabla \left. \frac{\partial h}{\partial z}(\cdot, z_{0}) \, \right\rvert_{u}} \nabla h (\cdot, 				      z_{0}) \, \bigr\rvert_{u}\right) \bigg\}\\
		&\phantom{f_{\tau}(u, z_{0}) =} + (\tau \cdot u)  \, \Biggl\{h(u, z_{0}) \left. \frac{\partial}{\partial z} \right				\rvert_{z=z_{0}} \bigl[ \Delta h(\cdot, z_{0}) \, \bigr\rvert_{u} + (n -2) h(u, z) \bigr]\\
		 &\phantom{f_{\tau}(u, z_{0}) =} \hspace{3.8cm} - \bigl[ \Delta h(\cdot, z_{0}) \, \bigr\rvert_{u} + (n - 2) h(u, 			z_{0}) \bigr] \frac{\partial h}{\partial z}(u, z_{0}) \Biggr\}.
		\end{aligned}
	\end{equation}
	As we mentioned in the beginning the expressions in the curly brackets do not depend on $\tau \in \mathbb{S}^{n-2}$,		$f_{\tau}(u, z_{0}) = 0$ for all $\tau \in \mathbb{S}^{n-2}$ and hence these expressions must individually vanish for each 	$u \in \mathbb{S}^{n-2}$ and $z_{0} \in I$.
\end{proof}

\begin{proposition}[Splitting Lemma] \label{PS5.2:split}
	If a transversely convex tube $\mathcal{T}$ in \emph{standard position} has cop, then its rectification $\mathcal{T}^{-}$ splits.
\end{proposition}

\begin{proof}
	Let $h \colon \mathbb{S}^{n-2} \times (-1, 1) \to \mathbb{R}$ be the transverse support function of the rectified tube $		\mathcal{T}^{-}$ and $\Gamma^{-} \colon \mathbb{S}^{n-2} \times I \to \mathbb{R}^{n}$ be the support map of $\mathcal{T}		^{-}$ given by
	\[
		\Gamma^{-}(u, z) = \nabla_{\mathbb{S}^{n-2}} h(u, z) + h(u, z)u.
	\]
	If we can show that $h(u, z) = r(z)h_{0}(u)$ is the product of $r \colon (-1, 1) \to (0, \infty)$ and the support function $h 		\colon \mathbb{S}^{n-2} \to \mathbb{R}$ of a fixed ovaloid, then 
	\begin{align*}
		\Gamma^{-}(u, z) &= \nabla_{\mathbb{S}^{n-2}} \bigl[ r(z)h_{0}(u) \bigr] + \bigl[ r(z)h_{0}(u) \bigr]u\\
					   &= r(z) \bigl[ \nabla_{\mathbb{S}^{n-2}} h_{0}(u) + h_{0}(u)u \bigr] = r(z) \Gamma^{\circ}(u)
	\end{align*}
	where $\Gamma^{\circ}$ is the support parameterization of a fixed ovaloid.
	
	According to Proposition \ref{PS5.2:pde}, the transverse support function $h$ satisfies the partial differential equation 		\eqref{PS5.2:pde2}, which is equivalent to
	\begin{equation} \label{ES5.2:splteqvcnd1}
		\frac{\partial}{\partial z} \left( \frac{\Delta h(u, z) + (n -2) h(u, z)}{h(u, z)} \right) = 0 \text{ for every $u \in 				\mathbb{S}^{n-2}$ and $\lvert z \rvert < 1$.}
	\end{equation}
	At height $z = 0$, using \cite[p.119]{rS14}, we have
	\[
		 \Delta h(u, 0) + (n -2) h(u, 0) = \sum_{i=1}^{n-2} \frac{1}{k_{i}(u)}	
	\]
	for every $u \in \mathbb{S}^{n-2}$, where $k_{i}$ is the $\mathrm{i}^{\mathrm{th}}$ principal curvature of the ovaloid $	\mathcal{O}(0)$, $i = 1, \dotsc, n-2$. Therefore, the equation \eqref{ES5.2:splteqvcnd1} is equivalent to 
	\begin{equation} \label{ES5.2:splteqvcnd2}
		\begin{aligned}
		\frac{\Delta h(u, z) + (n -2) h(u, z)}{h(u, z)} &= \frac{\Delta h(u, 0) + (n -2) h(u, 0)}{h(u, 0)}\\
										&= \frac{1}{h(u, 0)} \sum_{i=1}^{n-2} \frac{1}{k_{i}(u)} =: q^{2}(u)
		\end{aligned}
	\end{equation}
	 for every $u \in \mathbb{S}^{n-2}$ and $\lvert z \rvert < 1$. And the condition \eqref{ES5.2:splteqvcnd2} is equivalent to
	\begin{equation} \label{ES5.2:splteqvcnd3}
		\Delta h(u, z) + (n - 2) h(u, z) = q^{2}(u) \, h(u, z) \text{ for every $u \in \mathbb{S}^{n-2}$ and $\lvert z \rvert < 1$.}
	\end{equation} 
	
	Define the operator $L = \Delta + (n - 2 - q^{2})$, then the transverse support function h satisfies
	\[
		L h (\cdot, z) = 0 \text{ for every $\lvert z \rvert < 1$.}	
	\]
	Let $h_{0} = h(\cdot, 0)$ be the support function of the ovaloid $\mathcal{O}(0)$. Our aim is to show that $h(u, z) = r(z) 		h_{0}(u)$ for some positive function $r \colon (-1, 1) \to (0, \infty)$ of height $z$. 
	
	By construction $L h_{0}(u) = 0$ for each $u \in \mathbb{S}^{n-2}$ and assume that there exists $g \in C^{\infty}			(\mathbb{S}^{n-2}) \setminus \{ 0 \}$ satisfying $g \geq 0$ and $L g(u) = 0$ for each $u \in \mathbb{S}^{n-2}$. Since 		$h_{0}$ is a strictly positive function, $h_{0}$ and $g$ are continuous on a compact set, we can fix $\lambda > 0$ so that $	\lambda h_{0} - g > 0$ on $\mathbb{S}^{n-2}$. Define the positive constants $\lambda_{0} > 0$ and $t_{0} > 0$ as follows:
	\[
		0 < \max_{u \in \mathbb{S}^{n-2}} \frac{g(u)}{\lambda h_{0}(u)} = t_{0} \qquad \lambda_{0} = \lambda t_{0}.
	\]
	The subset $U \subseteq \mathbb{S}^{n-2}$ defined as
	\[
		U = \bigl\{ u \in \mathbb{S}^{n-2} \colon \lambda_{0}h_{0}(u) - g(u) = 0 \bigr\}
	\]
	is nonempty because there exists $u \in \mathbb{S}^{n-2}$ with
	\[
		\frac{g(u)}{\lambda h_{0}(u)} = t_{0} \text{ and hence } \lambda_{0}h_{0}(u) - g(u) = 0.
	\]
	The subset $U$ is also closed because $\lambda_{0}h_{0} - g$ is a continuous function. Now we want to show that $U$ is 	also open and then conclude that $\lambda_{0} h_{0} = g$.
	
	Let $u \in U$, and consider a local coordinate system $(V, \phi)$ of $\mathbb{S}^{n-2}$ about $u$
	\begin{align*}
		&\phi \colon V \subseteq \mathbb{S}^{n-2} \to \mathbb{R}^{n-2}\\
		&\phi = (x_{1}, \dotsc, x_{n-2}), \ \phi(u) = 0.
	\end{align*}
	Consider the function $(\lambda_{0}h_{0} - g) \circ \phi^{-1} \colon \phi(V) \subseteq \mathbb{R}^{n-2} \to \mathbb{R}		$, which is smooth and satisfies $(\lambda_{0}h_{0} - g) \circ \phi^{-1} \geq 0$ because
	\[
		\lambda_{0}h_{0}(v) - g(v) = \lambda h_{0}(v) \left( t_{0} - \frac{g(v)}{\lambda h_{0}(v)} \right) \geq 0
	\]
	for every $v \in \mathbb{S}^{n-2}$.
	
	The operator $L = \Delta + (n - 2 -q^{2})$, under the local coordinate system $(V, \phi)$, is strictly elliptic in the domain 		$\phi(W) \subseteq \mathbb{R}^{n-2}$ for some $u \in W \Subset V$. Therefore, the operator $L$ and the smooth function $	(\lambda_{0}h_{0} - g) \circ \phi^{-1}$ in $\phi(W) \subseteq \mathbb{R}^{n-2}$ satisfy the hypotheses of Harnack's 		inequality \cite[p.199]{GT98} and hence exists a positive constant $C > 0$ depending only on $\phi(W)$ and $L$ so that for 	some closed  ball $\mathbb{B}^{n}(0, R) \subseteq \phi(W)$
	\[
		0 \leq \sup_{\mathbb{B}^{n}(0, R)} (\lambda_{0}h_{0} - g) \circ \phi^{-1} \leq C \inf_{\mathbb{B}^{n}(0, R)} 			(\lambda_{0}h_{0} - g) \circ \phi^{-1} = 0
	\]
	and the last equality holds because $\phi^{-1}(0) = u$ and $\lambda_{0}h_{0}(u) - g(u) = 0$. So we can conclude that
	\[
		(\lambda_{0}h_{0} - g) \circ \phi^{-1} \, \Bigr\rvert_{\mathbb{B}^{n}(0, R)} = 0 \qquad (\lambda_{0}h_{0} - g) \, 			\Bigr\rvert_{\phi^{-1}\bigl[ \mathbb{B}^{n}(0, R) \bigr]} = 0.
	\]
	For the open ball $\mathbb{U}^{n}(0, R)$ of radius $R$ and centered at the origin, $\phi^{-1} \bigl[ \mathbb{U}^{n}(0, 		R) \bigr]$ is an open neighborhood of $u \in U$ in which $\lambda_{0}h_{0} - g$ vanishes identically and hence	 
	\[
		u \in \phi^{-1} \bigl[ \mathbb{U}^{n}(0, R) \bigr] \subseteq U.
	\]
	So we can conclude that $\lambda_{0}h_{0} = g$ on $\mathbb{S}^{n-2}$.
	
	Now for every $\lvert z \rvert < 1$, $h(\cdot, z) \colon \mathbb{S}^{n-2} \to \mathbb{R}$ is a positive, smooth function 		satisfying $\tilde{L}h(\cdot, z) \equiv 0$. So we can conclude that for every $\lvert z \rvert < 1$ there exists a positive 		number $r(z)$ with
	\[
		h(u, z) = r(z) h_{0}(u) \text{ for every $u \in \mathbb{S}^{n-2}$.} \qedhere
	\]
\end{proof}

\begin{proposition} \label{PS5.2:classofrcttube}
	Let $\mathcal{T}$ be a transversely convex tube in \emph{standard position} with cop. Then its rectification $\mathcal{T}^{-}$ is 		either
	\begin{enumerate}
		\item a cylinder over a central ovaloid, or
		\item affinely isomorphic to a hypersurface of revolution.
	\end{enumerate}
\end{proposition}

\begin{proof}
	We show that when $\mathcal{T}$ is a transversely convex tube in \emph{standard position}, and $\mathcal{T}^{-}$ is not a 		cylinder, then there exists a linear isomorphism that fixes the $e_{n}$-axis while making each horizontal cross-section $		\mathcal{O}(z) - c(z)$ of $\mathcal{T}^{-}$ simultaneously spherical. 
	
	According to Lemma \ref{PS5.2:split} (Splitting lemma), the transverse support function $h$ can be written as $h(u, z) = 		r(z) h_{0}(u)$ for each $\lvert z \rvert < 1$ and $u \in \mathbb{S}^{n-2}$. The partial differential equation 				\eqref{PS5.2:pde1} can therefore be written as
	\begin{equation} 
		\begin{aligned}
			0 &= r(z)r'(z)h_{0}(u) \Bigl\{ \nabla \bigl[ \Delta h_{0} \bigr] (u) + (n -2) \nabla h_{0}(u)\Bigr\}\\
			   &\quad +r(z)r'(z) \Bigl\{ \bigl[ \Delta h_{0}(u) \bigr] \nabla h_{0}(u) + (n -2) h_{0}(u) \nabla h_{0}(u) \Bigr					\}\\
			   &\quad + 2 \left\{ r(z)r'(z) h_{0}(u) \nabla h_{0} (u)+ \nabla_{\textstyle \nabla r'(z) h_{0}(u)} \nabla 					   r(z)h_{0}(u)\right\}\\
			   &= r(z)r'(z) \Bigl\{ h_{0}(u) \nabla \bigl[ \Delta h_{0}\bigr](u) + 2(n -1)h_{0}(u) \nabla h_{0}(u)\\
			   & \qquad \phantom{r(z)r'(z) \Bigl\{} + \bigl[ \Delta h_{0}(u) \bigr] \nabla h_{0}(u) + 2 \nabla_{\textstyle 				    \nabla h_{0}(u)} \nabla h_{0}(u) \Bigr\}.
		\end{aligned}
	\end{equation}
	If $\mathcal{T}^{-}$ is not cylindrical, then there exists $\lvert z_{0} \rvert < 1$, $r'(z_{0}) \ne 0$ and since $r(z_{0}) \ne 	0$ we must have
	\begin{equation} \label{ES5.2:pde1spltfnc}
		\begin{aligned}
			0 &= \bigl[ \Delta h_{0}(u) \bigr] \nabla h_{0}(u) + h_{0}(u) \nabla \bigl[ \Delta h_{0} \bigr](u) + 2(n-1) h_{0}			(u) \nabla h_{0}(u) \\
			& \phantom{{}= \bigl[ \Delta h_{0}(u) \bigr] \nabla h_{0}(u) + h_{0}(u) \nabla \bigl[ \Delta h_{0} \bigr](u)} + 			\underbrace{2 \nabla_{\textstyle \nabla h_{0}(u)} \nabla h_{0}(u)}_{\textstyle \crci}
		\end{aligned}
	\end{equation}
	for each $u \in \mathbb{S}^{n-2}$. 
	
	Let's consider the smooth function $\Delta h_{0}^{2} + 2 (n - 1)h_{0}^{2}$ on $\mathbb{S}^{n-2}$ and calculate its 		spherical gradient
	\begin{equation} \label{ES5.2:sphgrad}
		\begin{aligned}
		&\nabla \bigl[ \Delta h_{0}^{2} + 2 (n -1)h_{0}^{2} \bigr](u) = \nabla \bigl[ \Delta h_{0}^{2} \bigr](u) + 2 (n -1) 			\nabla h_{0}^{2}(u)\\
		&= \nabla \bigl[ 2 h_{0} \Delta h_{0} + 2 \lvert \nabla h_{0} \rvert^{2} \bigr](u) + 4(n - 1) h_{0}(u) \nabla h_{0}(u)			\\
		&= 2 \bigl[ \Delta h_{0}(u) \bigr] \nabla h_{0}(u) + 2 h_{0}(u) \nabla \bigl[ \Delta h_{0}\bigr](u) + 4 (n -1)h_{0}(u) 		\nabla h_{0}(u)\\
		& \phantom{{}= 2 \bigl[ \Delta h_{0}(u) \bigr] \nabla h_{0}(u) + 2 h_{0}(u) \nabla \bigl[ \Delta h_{0}\bigr](u)}+ 			\underbrace{2 \nabla \bigl[ \nabla h_{0} \cdot \nabla h_{0} \bigr]}_{\textstyle \crci}. 
		\end{aligned}
	\end{equation}
	The part \crci \ of the equation \eqref{ES5.2:pde1spltfnc} can be expanded using the \hyperlink{ortfrmfld}{orthonormal frame 	field} $\{ E_{1}, \dotsc, E_{n-2} \}$ as follows:
 	\begin{align}
		2 \nabla_{\textstyle \nabla h_{0}(u)} \nabla h_{0}(u) &= 2 \sum_{k=1}^{n-2} (E_{k}h_{0})(u) \nabla_{E_{k}(u)} 			\left[ \sum_{j=1}^{n-2} (E_{j}h_{0}) \, E_{j} \right](u) \notag\\
											     &= 2 \sum_{j, k=1}^{n-2} \Bigl[ (E_{k}h_{0})(u) \  E_{k}			(E_{j}h_{0})(u) \Bigr] E_{j}(u) \notag\\
		\phantom{2 \nabla_{\textstyle \nabla h_{0}(u)} \nabla h_{0}(u)} &= 2 \sum_{j=1}^{n-2} \left[ \sum_{k=1}^{n-2} 			(E_{k}h_{0}) (u) \ E_{k}(E_{j} h_{0})(u) \right] E_{j}(u) \notag\\
											     		      &= 2 \sum_{j=1}^{n-2} \left[ \sum_{k=1}^{n-2} 			(E_{k}h_{0}) (u) \ E_{j}(E_{k} h_{0})(u) \right] E_{j}(u)
	\end{align}
	where the change of the order of differentiation at the last equality is possible because of the observation 					\eqref{ES5.2:orddif}. Again by using the \hyperlink{ortfrmfld}{orthonormal frame field} $\{ E_{1}, \dotsc, E_{n-2}\}$, the 	part \crci \ of the equation \eqref{ES5.2:sphgrad} can expanded as follows:
	\begin{align*}
		&2 \nabla \bigl[ \nabla h_{0} \cdot \nabla h_{0}\bigr](u) = 2 \sum_{j=1}^{n-2} E_{j} \bigl[ \nabla h_{0} \cdot \nabla 												       h_{0} \bigr](u) E_{j}(u)\\
		&= 2 \sum_{j=1}^{n-2} 2 \bigl[ \nabla_{E_{j}(u)} \nabla  h_{0}(u) \cdot \nabla h_{0}(u) \bigr] E_{j}(u) \\
		&= 4 \sum_{j=1}^{n-2} \left\{ \nabla_{E_{j}(u)} \left[ \sum_{k=1}^{n-2} (E_{k}h_{0}) \, E_{k}\right](u) \cdot 			\nabla h_{0}(u) \right\} E_{j}(u)\\	
	\end{align*}
	\begin{align}
		&= 4 \sum_{j, k=1}^{n-2} \Bigl\{ \bigl[E_{j}(E_{k}h_{0})(u) \,  E_{k}(u)+ (E_{k}h_{0})(u) 						\underbrace{\nabla_{E_{j}(u)} E_{k} \, \bigr\rvert_{u}}_{=0} \bigr] \cdot \nabla h_{0}(u) \Bigr\} E_{j}(u) \notag\\
		&= 4 \sum_{j,k=1}^{n-2} E_{j}(E_{k}h_{0})(u) \left\{ E_{k}(u) \cdot \sum_{l=1}^{n-2} (E_{l}h_{0})(u) \, E_{l}			(u) \right\} E_{j}(u) \notag\\
		&= 4 \sum_{j=1}^{n-2} \left[ \sum_{k=1}^{n-2} (E_{k}h_{0})(u) \ E_{j}(E_{k}h_{0})(u) \right] E_{j}(u).
	\end{align}
	Therefore we can conclude that 
	\[
		\nabla \bigl[ \nabla h_{0} \cdot \nabla h_{0} \bigr](u) = 2 \nabla_{\textstyle \nabla h_{0}(u)} \nabla h_{0}(u) 
	\]
	and hence
	\begin{align*}
		2^{-1} \nabla \bigl[ \Delta h_{0}^{2} + 2 (n - 1)h_{0}^{2} \bigr](u) &= \bigl[ \Delta h_{0}(u) \bigr] \nabla h_{0}(u) 		+ h_{0}(u) \nabla \bigl[ \Delta h_{0} \bigr](u)\\
														   &\qquad + 2(n-1) h_{0}(u) \nabla h_{0}(u) + 2 			\nabla_{\textstyle \nabla h_{0}(u)} \nabla h_{0}(u)\\
														   &=0.
	\end{align*}
	for each $u \in \mathbb{S}^{n-2}$. Using Lemma \ref{LS5.1:orgcntellipsoid} we can conclude that the positive 			function $h_{0}$ on $\mathbb{S}^{n-2}$ is the support function of an origin centered ellipsoid. Since the support function 	$h$ satisfies $h(u, z) = r(z) h_{0}(u)$, the item \eqref{LS2.2:spfnpr3} of Lemma \ref{LS2.2:spfnpr} implies that every 		horizontal cross 	section $\mathcal{O}	(z) - c(z)$ of $\mathcal{T}^{-}$ is a homothetic copy of $\mathcal{O}(0) - c(0)$. 		Therefore, there exists an affine isomorphism of $ \mathbb{R}^{n-1}$ mapping $\mathcal{O}(0) - c(0)$ onto the unit 		sphere $\mathbb{S}^{n-2} \subseteq \mathbb{R}^{n-1}$. Extend this affine isomorphism to $\mathbb{R}^{n}$ by fixing 		the last coordinate and hence $\mathcal{T}^{-}$ is affinely isomorphic to a hypersurface of revolution in $\mathbb{R}^{n}	$. 
\end{proof}

\subsection{Straightening the central curve}

In this subsection, we show that the central curve of a transversely convex tube $\mathcal{T}$ with \emph{cop} is affine and therefore $\mathcal{T}$ is affinely isomorphic to its rectification $\mathcal{T}^{-}$.

The following auxiliary lemma is an important component in the proof of Proposition \ref{LS5.3:affcrv} (Axis Lemma). For a short proof see \cite{bS12}.

\begin{lemma} [Linearity criterion]  \label{LS5.3:affcrv}
	A curve $c \colon I \to \mathbb{R}^{n}$ defined on an open interval is affine if and only if it is locally odd, i.e., for each $b 	\in I$ there exists $t(b) > 0$ so that 
	\[
		c(b + t) - c(b) = c(b) - c(b-t)
	\]
	for each $\lvert t \rvert < t(b)$.
\end{lemma}

\begin{proposition} [Axis lemma] \label{PS5.3:axslemm}
	Suppose $\mathcal{T}$ is a transversely convex tube with cop. Then its central curve is affine, making $\mathcal{T}$ 		affinely isomorphic to its rectification $\mathcal{T}^{-}$. 
\end{proposition}

\begin{proof}
	Without loss of generality we can assume that $\mathcal{T}$ lies in \emph{standard position} because a general transversely convex 	tube, by definition, is affinely isomorphic to the one in \emph{standard position} and the central curve is preserved under affine 		isomorphisms. Using Proposition \ref{PS5.2:classofrcttube} and the invariance of the central curve under affine 				isomorphisms we can further assume that $\mathcal{T}^{-}$ is either a cylinder or a hypersurface of revolution with 			$\mathbb{R}e_{n}$ as its axis and show that the \emph{cop} assumption forces the central curve of $\mathcal{T}$ to be a 	straight line. 
	
	\textbf{Cylindrical case:} Assume that $\mathcal{T}^{-}$ is a cylinder over a central ovaloid $\mathcal{O}(0)$ at height 		$0$. Each horizontal cross-section $\mathcal{O}(z)$ at height $z \in (-1, 1)$ translates to $\mathcal{O}(0)$. Let $\Gamma$ 	be the support parameterization of $\mathcal{O}(0)$, then a parameterization of $\mathcal{T}$ can be obtained as
	\begin{align*}
		&X \colon \mathbb{S}^{n-2} \times (-1, 1) \to \mathcal{T}\\
		&X(u, z) = \bigl( c(z) + \Gamma(u), z \bigr).
	\end{align*}
	Fix $b \in (-1, 1)$, $u \in \mathbb{S}^{n-2}$, consider the $u$-diameter of $\mathcal{O}(b)$, which is the line segment 		joining the unique two points at which $\mathcal{O}(b)$ has normal line $\mathbb{R}u$. Since $\Gamma$ support 			parameterizes $\mathcal{O}(0)$, the endpoints of the $u$-diamater are $X(u, b)$ and $X(-u, b)$. The tangent hyperplanes 		of $\mathcal{T}$ at $X(u, b)$ and $X(-u, b)$ are
	\begin{align*}	
		\tng_{X(u, b)} \mathcal{T} &= \mathrm{span} \Bigl\{ \partial_{1}X(u, b) \bigl[ u^{\perp} \bigr], \, \partial_{2}X(u, 			b)(1) \Bigr\}\\
		&= u^{\perp} \oplus \mathbb{R}(c'(b), 1)\\
		&= (-u)^{\perp} \oplus \mathbb{R}(c'(b), 1)\\
		&= \mathrm{span} \Bigl\{ \partial_{1}X(-u, b) \bigl[ (-u)^{\perp} \bigr], \, \partial_{2}X(-u, b)(1) \Bigr\}\\
		&= \tng_{X(-u, b)} \mathcal{T}.
	\end{align*}
	
	Fix the $u$-diameter of $\mathcal{O}(b)$ as axis and tilt the horizontal hyperplane $H_{e_{n}, b}$ in a direction 			perpendicular to the axis with some small slope $\epsilon > 0$ to get a new hyperplane $P_{\epsilon}(u)$. For sufficiently 		smal $\epsilon > 0$, the intersection $\mathcal{O}(b, u, \epsilon) := \mathcal{T} \cap P_{\epsilon}(u)$ will remain a 			central ovaloid because $\mathcal{T}$ has \emph{cop}.
	
	Since $P_{\epsilon}(u)$ contains the $u$-diameter of $\mathcal{O}(b)$, the endpoints $X(u, b)$ and $X(-u, b)$ of that 		diameter remain on $\mathcal{O}(b, \epsilon, u)$ independently on $\epsilon > 0$. Since the tangent hyperplanes to $		\mathcal{T}$ at these points are parallel and their intersections with $P_{\epsilon}(u)$ form codimension one planes 			tangent to $\mathcal{O}(b, u, \epsilon)$ at $X(u, b)$ and $X(-u, b)$, those planes are also parallel.
	
	We can conclude that the $u$-diameter of $\mathcal{O}(b)$ remains a diameter of $\mathcal{O}(b, u, \epsilon)$ 			independently on $\epsilon > 0$, and hence $(c(b), b)$ remains the center of symmetry of $\mathcal{O}(b, u, \epsilon)$ for 	each $u \in \mathbb{S}^{n-2}$ and for all sufficiently small $\epsilon > 0$. The center of $\mathcal{O}(b, u, \epsilon)$ 		remains fixed as we vary $\epsilon > 0$.
	
	Every point sufficiently close to $\mathcal{O}(b)$ on $\mathcal{T}$ belongs to $\mathcal{O}(b, u, \epsilon)$ for some $u 	\in \mathbb{S}^{n-2}$ and small enough $\epsilon > 0$, so that by \emph{cop}, its reflection through $(c(b), b)$ also lies 		on $\mathcal{T}$. It follows that an entire neighborhood $U$ of $\mathcal{O}(b)$ in $\mathcal{T}$ has reflection 			symmetry through $(c(b), b)$. In this neighborhood, pick two points $p$, $q$, which are the endpoints of a diameter of 		horizontal cross-section and denote the center of symmetry $(c(b), b)$ as $r$. 
	
	We can define the reflection through $r$ as
	\begin{align*}
		A_{r} &\colon U \to U\\
		x	  &\mapsto 2r - x
	\end{align*}
	which maps $p$ and $q$ respectively to $p'$ and $q'$, which belong to a horizontal cross-section in the neighborhood $U$. 	Since the differentials satisfy $\derv A_{r}(p) = - \mathrm{id}_{\tng_{p} \mathcal{T}}$ and $\derv A_{r}(q) = - 			\mathrm{id}_{\tng_{q} \mathcal{T}}$ and $\tng_{p} \mathcal{T} = \tng_{q} \mathcal{T}$ we can conclude that $			\tng_{p'} \mathcal{T} = \tng_{q'} \mathcal{T}$. This implies that the segment $[p', q']$ is a diameter and the midpoint of $	[p, q]$ is mapped to the midpoint of $[p', q']$. Since this holds in the neighborhood $U$ of $\mathcal{O}(b)$ we can 			choose $t(b) > 0$ so that
	\begin{align*}
		c(b + t)  - c(b) &= - \bigl[ c(b - t) - c(b) \bigr]\\
				     &= c(b) - c(b - t)
	\end{align*}	
	for each $0 \leq t < t(b)$. Since $b \in (-1, 1)$ is arbitrary then by using Lemma \ref{LS5.3:affcrv} (Linearity criterion) 		we can conclude that $c$ is affine.
		
	\textbf{Horizontally spherical case:} Each horizontal plane $H_{e_{n}, b}$ cuts the tube $\mathcal{T}$ in an $(n-2)$-		sphere centered at $(c(b), b)$, $b \in (-1, 1)$. Let $F(b)$ denote the square of the radius of this sphere and $c(b) = (c_{1}(b), 	\dotsc, c_{n-1}(b))$ for each $b \in (-1, 1)$.
	
	The tube $\mathcal{T}$ has a parameterization $X(u, z) = \bigl( c(z) + \Gamma(u, z), z \bigr)$ and for any fixed $u \in 		\mathbb{S}^{n-2}$ the function $F(z) = \lvert \Gamma(u, z) \rvert^{2}$ is smooth. Given the function $F$, the tube 			$\mathcal{T}$ can be described as
	\[
		\mathcal{T} = \left\{ (x, z) \in \mathbb{R}^{n-1} \times \mathbb{R} \biggm| x= (x_{1}, \dotsc, x_{n-1}), \, \lvert z 			\rvert < 1, \, \sum_{i=1}^{n-1} \lvert x_{i} - c_{i}(z) \rvert^{2} = F(z) \right\}.
	\]
	
	Let $\beta \in \mathbb{R}$ and $f$ be any function defined in a neighborhood of $\beta$. The $\beta$-translate of $f$ is 		defined as $\mu_{\beta}f(t) = f(\beta + t)$ and 
	
	\begin{alignat*}{2}
		\mu_{\beta}^{+}f(t) &= \frac{\mu_{\beta}f(t) + \mu_{\beta}f(-t)}{2} & \qquad &\text{is the even part of $				\mu_{\beta}f$,}\\
		\mu_{\beta}^{-}f(t) &= \frac{\mu_{\beta}f(t) - \mu_{\beta}f(-t)}{2}  & 	            &\text{is the odd part of $				\mu_{\beta} f$.}		
	\end{alignat*}
	Fix $\lvert \beta \rvert < 1$, $\tau \in \mathbb{S}^{n-2} \times \{0\} \subseteq \mathbb{R}^{n-1} \times \{0\} \subseteq 		\mathbb{R}^{n}$, let $ \tau^{\perp}$ denote the orthogonal complemet of $\tau$ in $\mathbb{R}^{n-1} \times \{0\}$ with 	basis $\{e_{1}, \ldots, e_{n-2}\}$, and $\bar{e}_{n} = (0, \ldots, 0, 1)$.	
	Since $\mathcal{T}$ has \emph{cop}, and lies in \emph{standard position} we can find a small slope $m > 0$, and an $\bar{e}_{n}	$-intercept $b = b(\beta)$ such that the plane $P$ given in the new coordinates by
	\[
		P = \left\{ x \tau + \sum_{i=1}^{n-2} y_{i}e_{i} + (mx + b) \bar{e}_{n} \biggm| x, y_{1}, \dotsc y_{n-2} \in 				\mathbb{R} \right\}
	\]
	intersects $\mathcal{T}$ in a central ovaloid $\mathcal{O}$ with center having $\bar{e}_{n}$ coordinate $\beta$. From 		now on we will identify the curve $c(I)$, the plane $P$, the tube $\mathcal{T}$, and the ovaloid $\mathcal{O}$ with their 		coordinate representations under the basis $\{ \tau, e_{1}, \dotsc, e_{n-2}, \bar{e}_{n}\}$ because such an identification 		constitutes an orthogonal isomorphism. Namely, we make the following identifications:
	
	For the curve $c \colon I \to \mathbb{R}^{n-1}$ given by 
	\[
		c(z) = \bar{c}_{1}(z) \tau + \sum_{i=1}^{n-2} \bar{c}_{i+1}(z) e_{i}
	\]
	we let $c(z) = (\bar{c}_{1}(z), c^{\perp}(z))$, where $c^{\perp}(z) = (\bar{c}_{2}(z), \dotsc, \bar{c}_{n-1}(z))$. The 		plane $P$ becomes
	\[
		P = \left\{ \left( \frac{z - b}{m}, y_{1}, \dotsc, y_{n-2}, z \right) \biggm| y_{1}, \dotsc, y_{n-2}, z \in \mathbb{R}			\right\}
	\]
	the tube is replaced with
	\[
		\mathcal{T} = \left\{ (x, y_{1}, \dotsc, y_{n-2}, z) \colon \lvert z \rvert < 1, \bigl[ x - \bar{c}_{1}(z) \bigr]^{2} + 			\sum_{i=1}^{n-2} \bigl[ y_{i} - \bar{c}_{i+1}(z) \bigr]^{2} = F(z) \right\}
	\]
	and lastly the ovaloid $\mathcal{O} = \mathcal{T} \cap P$ is replaced with 
	
	\begin{align}
		\mathcal{O} &= \left\{ (y, z) \in \mathbb{R}^{n-2} \times \mathbb{R} \biggm| \lvert z \rvert < 1,\, \left[ \frac{z-b}			{m} - \bar{c}_{1}(z) \right]^{2} + \sum_{i=1}^{n-2} \bigl[ y_{i} - \bar{c}_{i+1}(z) \bigr]^{2} = F(z) \right\} 				\label{ES5.3:ovaloidyz-crd}\\
				    &= \Bigl\{ (y, z) \in \mathbb{R}^{n-2} \times \mathbb{R} \colon \lvert z \rvert < 1, \, \lvert y - 				c^{\perp}(z) \rvert^{2} = \lambda(z) \Bigr\} \notag
	\end{align}
	where $\lambda(z)$ is defined as
	\[
		\lambda(z) = F(z) - \left[ \frac{z-b}{m} - \bar{c}_{1}(z) \right]^{2}.
	\]
	Therefore, the ovaloid $\mathcal{O}$ can be written as
	\begin{multline*}
		\mathcal{O} = \Bigl\{ (\bar{y}, y_{n-2}, z) \in \mathbb{R}^{n-3} \times \mathbb{R} \times \mathbb{R} \colon 			\lvert z \rvert < 1, \, \lvert \bar{y} - \bigl( \bar{c}_{2}(z), \dotsc, \bar{c}_{n-2}(z) \bigr) \rvert^{2} \leq \lambda(z)\\
		 y_{n-2} = \bar{c}_{n-1}(z) \pm \sqrt{\lambda(z) - \lvert \bar{y} - (\bar{c}_{2}(z), \dotsc, \bar{c}_{n-2}(z)) 			\rvert^{2}} \Bigr\}.
	\end{multline*}	
	
	\begin{figure}[h]
		\centering
		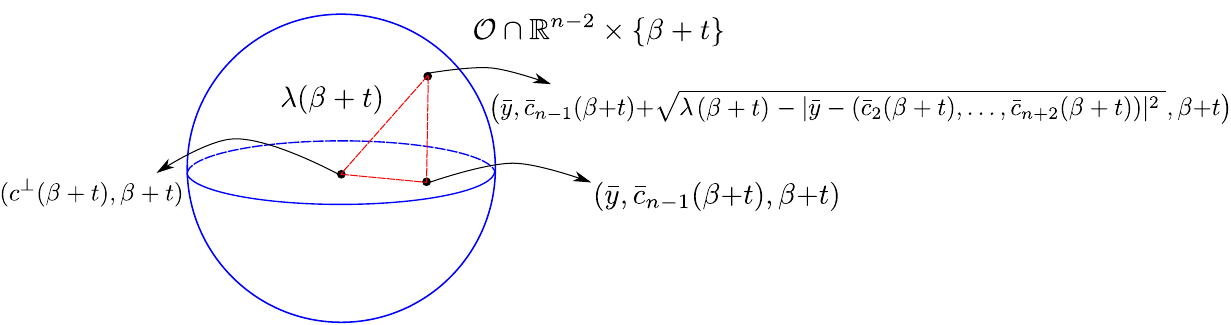
		\caption{Cross-section of the ovaloid $\mathcal{O}$}
	\end{figure}
		
	Let $L = \max \{z - \beta \colon (y, z) \in \mathcal{O}\}$, $\lvert t \rvert < L$, $t = z - \beta$, $\bar{\beta} := \beta - b$, $z 		= \bar{\beta} + b + t$, then the cross-section $\mathcal{O} \cap \mathbb{R}^{n-2} \times \{\beta + t\}$ of the ovaloid is an 	$(n-3)$-sphere centered at $(c^{\perp}(\beta + t), \beta + t)$ with radius $\lambda(\beta + t)$. When $t = 0$ and $z = \beta$, 	choose two antipodal points 
	\begin{align*}
		l^{+} &= \bigl( c^{\perp}(\beta) + (0, \lambda(\beta) \bigr), \beta) \in \mathcal{O} \cap (\mathbb{R}^{n-2} \times 			\{ \beta \})\\
		l^{-}  &= \bigl( c^{\perp}(\beta) - (0, \lambda(\beta) \bigr), \beta) \in \mathcal{O} \cap (\mathbb{R}^{n-2} \times 			\{ \beta \})
	\end{align*}
	
	\begin{figure}[h]
		\centering
		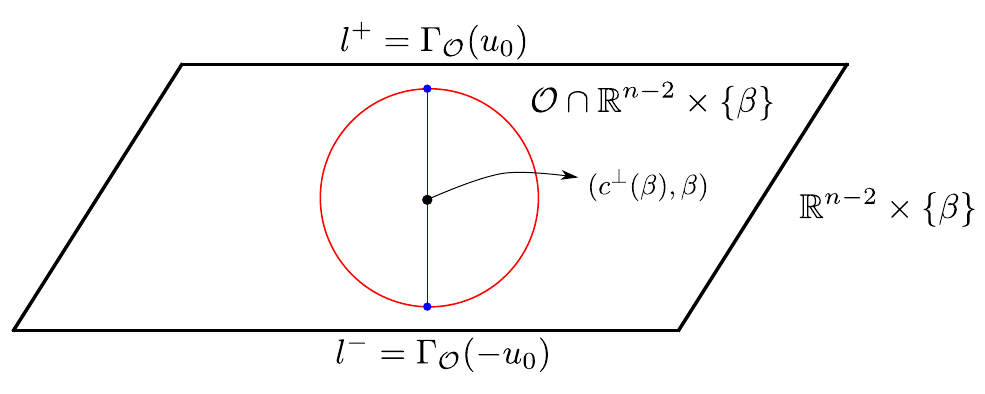
		\caption{Cross-section $\mathcal{O} \cap \mathbb{R}^{n-2} \times \{ \beta \}$}
	\end{figure}
	
	Pick the point $l^{+}$ and let $\Gamma_{\mathcal{O}}$ be the support parameterization of $\mathcal{O}$, then there 		exists a unique $u_{0} \in \mathbb{S}^{n-2}$ with $\Gamma_{\mathcal{O}}(u_{0}) = l^{+}$. Since $\mathcal{O}$ has 		the center of symmetry whose $\bar{e}_{n}$-coordinate is $\beta$, $\Gamma_{\mathcal{O}}(- u_{0})$ must also lie on		 $\mathcal{O} \cap \mathbb{R}^{n-2} \times \{\beta\}$. The ovaloid $\mathcal{O}$ has the same tangent hyperplane 		$u_{0}^{\perp} = (-u_{0})^{\perp}$ at the two endpoints of the diameter connecting $l^{+}$ to $\Gamma_{\mathcal{O}}		(-u_{0})$. This hyperplane, when intersected with $\mathbb{R}^{n-2} \times \{\beta\}$, becomes the tangent plane to $		\mathcal{O} \cap \mathbb{R}^{n-2} \times \{\beta\}$ at $l^{+}$ and $\Gamma_{\mathcal{O}}(-u_{0})$. Therefore, we 		can conclude that $l^{-} = \Gamma_{\mathcal{O}}(-u_{0})$, and the chord joining $l^{+}$ and $l^{-}$ passes through the 	center of $\mathcal{O}$ which has coordinates $(c^{\perp}(\beta), \beta)$ in the $(y, z)$ coordinate system.
	
	Using this fact we can express the central symmetry of $\mathcal{O}$ as follows: for every $y \in \mathbb{R}^{n-2}$ and 	$t \in \mathbb{R}$
	\begin{equation} \label{DS5.3:cntsymm}
		(c^{\perp}(\beta) + y, \beta + t) \in \mathcal{O} \text{ if and only if } (c^{\perp}(\beta) - y, \beta - t) \in \mathcal{O}.
	\end{equation}
	
	By using the expression of the ovaloid $\mathcal{O}$ in the $(y, z)$ coordinate system \eqref{ES5.3:ovaloidyz-crd} we get
	\begin{equation}\label{ES5.3:trnslfnc+}	
		\begin{aligned}
			& (c^{\perp}(\beta) + y, \beta + t) \in \mathcal{O} \iff  \mu_{\beta}^{+}F(t) + \mu_{\beta}^{-}F(t) = 					\mu_{\beta}F(t) = F(\beta + t) \\
			&= \left[ \frac{\bar{\beta} + t}{m} - \bar{c}_{1}(\beta + t) \right]^{2} + \bigl\lvert (c^{\perp}(\beta) + y) - 				c^{\perp}(\beta + t) \bigr\rvert^{2} \\
			&= \left[ \frac{\bar{\beta} + t}{m} - \mu_{\beta}^{+}\bar{c}_{1}(t) - \mu_{\beta}^{-}\bar{c}_{1}(t) 					\right]^{2} + \bigl\lvert c^{\perp}(\beta) + y - \mu_{\beta}^{+}c^{\perp}(t) - \mu_{\beta}^{-}c^{\perp}(t) \bigr				\rvert^{2}. 			
		\end{aligned}
	\end{equation}
	\begin{equation} \label{ES5.3:trnslfnc-}	
		\begin{aligned}
			&(c^{\perp}(\beta) - y, \beta - t) \in \mathcal{O} \iff \mu_{\beta}^{+}F(t) - \mu_{\beta}^{-}F(t) = \mu_{\beta}				F(-t) = F(\beta - t) \\
			&=\left[ \frac{\bar{\beta} - t}{m} - \bar{c}_{1}(\beta - t) \right]^{2} + \bigl\lvert (c^{\perp}(\beta) - y) - 					c^{\perp} (\beta - t) \bigr\rvert^{2}\\	
			&= \left[ \frac{\bar{\beta} - t}{m} - \mu_{\beta}^{+}\bar{c}_{1}(t) + \mu_{\beta}^{-}\bar{c}_{1}(t) 					\right]^{2} + \bigl\lvert c^{\perp}(\beta) - y - \mu_{\beta}^{+}c^{\perp}(t) + \mu_{\beta}^{-}c^{\perp}(t) \bigr				\rvert^{2}. 		
		\end{aligned}
	\end{equation}
	When we subtract the equation \eqref{ES5.3:trnslfnc-} from the equation \eqref{ES5.3:trnslfnc+} we get
	\begin{equation} \label{ES5.3:difftrnslfnc}
		\begin{aligned}
			&2 \mu_{\beta}^{-}F(t) = F(\beta + t) - F(\beta - t)\\
			&= \left[ \left( \frac{\bar{\beta}}{m} - \mu_{\beta}^{+}\bar{c}_{1}(t) \right) + \left( \frac{t}{m} - \mu_{\beta}			^{-}\bar{c}_{1}(t) \right) \right]^{2}\!+\! \bigl\lvert (c^{\perp}(\beta) - \mu_{\beta}^{+}c^{\perp}(t)) + (y - 				\mu_{\beta}^{-} c^{\perp}(t)) \bigr\rvert^{2}\\
			& - \left[ \left( \frac{\bar{\beta}}{m} - \mu_{\beta}^{+}\bar{c}_{1}(t) \right) - \left( \frac{t}{m} - \mu_{\beta}				^{-} \bar{c}_{1}(t) \right) \right]^{2}\! - \! \bigl\lvert (c^{\perp}(\beta) - \mu_{\beta}^{+}c^{\perp}(t)) - (y - 				\mu_{\beta}^{-}c^{\perp}(t)) \bigr\rvert^{2}\\
			&=4 \left( \frac{\bar{\beta}}{m} - \mu_{\beta}^{+}\bar{c}_{1}(t) \right)\left( \frac{t}{m} - \mu_{\beta}^{-} 				\bar{c}_{1}(t) \right) + 4 \bigl(c^{\perp}(\beta) - \mu_{\beta}^{+}c^{\perp}(t)\bigr) \cdot \bigl(y - \mu_{\beta}				^{-}c^{\perp}(t)\bigr)
		\end{aligned}
	\end{equation}
	Recall that 
	\begin{align*}
		\mathcal{O} \cap \mathbb{R}^{n-2} \times \{\beta + t\} &= \Bigl\{ (y, \beta + t) \in \mathbb{R}^{n-2} \times 				\mathbb{R} \colon \lvert y - c^{\perp}(\beta + t) \rvert^{2} = \lambda (\beta + t) \Bigr\}\\
		\intertext{where}
		\lambda(\beta + t) &= F(\beta + t) - \left[ \frac{\bar{\beta} + t}{m} - \bar{c}_{1}(\beta + t) \right]^{2} > 0
	\end{align*}
	for sufficiently small $\lvert t \rvert \geq 0$ because the ovaloid $\mathcal{O}$ can not contain its center of symmetry and 		hence $\lambda (\beta) > 0$. So we can conclude that $\mathcal{O} \cap \mathbb{R}^{n-2} \times \{\beta + t\}$ is a 			nondegenerate sphere
	\[
		\mathcal{O} \cap \mathbb{R}^{n-2} \times \{\beta + t\} = \bigl[ c^{\perp}(\beta + t) + \lambda(\beta + t) \mathbb{S}			^{n-3} \bigr] \times \{\beta + t\}
	\]
	for small enough $\lvert t \rvert \geq 0$.
	
	Pick any two distinct points $\bar{y}$, $\bar{y}' \in \mathcal{O} \cap \mathbb{R}^{n-2} \times \{\beta + t\}$,
	\begin{equation} \label{ES5.3:dstpnts}
		\begin{aligned}
			\bar{y}  = (y, \beta + t) &= (c^{\perp}(\beta + t) + \xi, \beta + t) \in \mathcal{O}\\
							  &= \bigl( c^{\perp}(\beta) + \bigl[ c^{\perp}(\beta + t) - c^{\perp}(\beta) + \xi \bigr], 								        \beta + t \bigr)\\
			\bar{y}' = (y', \beta + t) &= (c^{\perp}(\beta + t) + \xi', \beta + t) \in \mathcal{O}\\
							   &= \bigl( c^{\perp}(\beta) + \bigl[ c^{\perp}(\beta + t) - c^{\perp}(\beta) + \xi' \bigr], 								        \beta + t \bigr)
		\end{aligned}
	\end{equation}
	for two distinct points $\xi$, $\xi' \in \lambda(\beta + t) \mathbb{S}^{n-3}$. The central symmetry condition 				\eqref{DS5.3:cntsymm} of the ovaloid $\mathcal{O}$ and the equations \eqref{ES5.3:dstpnts}, \eqref{ES5.3:difftrnslfnc} 		allow us to write the following two equalities				
	\begin{equation} \label{ES5.3:eqybar}
		\begin{aligned}
			\mu_{\beta}^{-}F(t) &= 2 \left( \frac{\bar{\beta}}{m} - \mu_{\beta}^{+}\bar{c}_{1}(t) \right) \! \left( \frac{t}								  {m} - \mu_{\beta}^{-} \bar{c}_{1}(t) \right) \\
						      &\ \quad + 2 \bigl(c^{\perp}(\beta) - \mu_{\beta}^{+}c^{\perp}(t)\bigr) \cdot 								\Bigl( \bigl[ c^{\perp}(\beta + t) - c^{\perp}(\beta) + \xi \bigr] - \mu_{\beta}^{-}c^{\perp}(t)\Bigr)
		\end{aligned}
	\end{equation}
	using the coordinates of $\bar{y}$ and 
	\begin{equation} \label{ES5.3:eqybarprm}
		\begin{aligned}
			\mu_{\beta}^{-}F(t) &=2 \left( \frac{\bar{\beta}}{m} - \mu_{\beta}^{+}\bar{c}_{1}(t) \right) \left( \frac{t}								{m} - \mu_{\beta}^{-} \bar{c}_{1}(t) \right) \\
						       & \ \quad + 2 \bigl(c^{\perp}(\beta) - \mu_{\beta}^{+}c^{\perp}(t)\bigr) \cdot 								    \Bigl(\bigl[ c^{\perp}(\beta + t) - c^{\perp}(\beta) + \xi' \bigr] - \mu_{\beta}^{-}c^{\perp}(t)\Bigr)
	 	\end{aligned}
	\end{equation}
	using the coordinates of $\bar{y}'$. Subtracting the equation \eqref{ES5.3:eqybar} from the equation					\eqref{ES5.3:eqybarprm} we get 
	\begin{equation} \label{ES5.3:zerodtprd}
		0 = \bigl( c^{\perp}(\beta) - \mu_{\beta}^{+}c^{\perp}(t) \bigr) \cdot (\xi' - \xi).
	\end{equation}
	Since the equation \eqref{ES5.3:zerodtprd} holds for every $\xi$, $\xi' \in \lambda(\beta + t) \mathbb{S}^{n-3}$ we can 		conclude that 
	\begin{align*}
		c^{\perp}(\beta) &= \mu_{\beta}^{+}c^{\perp}(t) = \frac{\mu_{\beta}c^{\perp}(t) + \mu_{\beta}c^{\perp}(-t)}{2}\\
		\intertext{for small enough $\lvert t \rvert$ depending on $\beta \in (-1, 1)$.}
		\Rightarrow 2c^{\perp}(\beta) &= c^{\perp}(\beta + t) + c^{\perp}(\beta - t)\\
		\Rightarrow c^{\perp}(\beta + t) - c^{\perp}(\beta) &= - \bigl[ c^{\perp}(\beta - t) - c^{\perp}(\beta) \bigr]
	\end{align*}
	holds for small enough $\lvert t \rvert$ depending on $\beta \in (-1, 1)$ and hence using Lemma \ref{LS5.3:affcrv} 			(Linearity criterion) we can conclude that $c^{\perp}$ is affine.
	
	The curve $c \colon I \to \mathbb{R}^{n-1}$ has coordinates $c(z) = \bigl( \bar{c}_{1}(z), c^{\perp}(z) \bigr)$ with 			respect the the orthonormal basis $\{ \tau, e_{1}, \dotsc, e_{n-2}\}$ and since $c^{\perp}$ is affine it must be of the form
	\[
		c^{\perp}(z) = (a_{1} + b_{1}z, \dotsc, a_{n-2} + b_{n-2}z).
	\]
	Therefore, the curve $c$ must be of the form
	\begin{align*}
		c(z) &= (\bar{c}_{1}(z), a_{1} + b_{1}z, \dotsc, a_{n-2} + b_{n-2}z)\\
		\intertext{with respect to the basis $\{ \tau, e_{1}, \dotsc, e_{n-2}\}$, and}
		       &= (a_{1} + b_{1}z, \bar{c}_{1}(z), a_{2} + b_{2}z, \dotsc, a_{n-2} + b_{n-2}z)
	\end{align*}
	with respect to the basis $\{ e_{1}, \tau, e_{2}, \dotsc, e_{n-2}\}$. 	Using $e_{1}$ as the tilt direction and the argument 		above we get
	\[
		c(z) = (\bar{a}_{1} + \bar{b}_{1}z, \bar{a}_{2} + \bar{b}_{2}z, \dotsc, \bar{a}_{n-1} + \bar{b}_{n-1}z)
	\]
	with respect to the basis $\{ e_{1}, \tau, e_{2}, \dotsc, e_{n-2}\}$ and hence we can conclude that the curve $c \colon I \to 		\mathbb{R}^{n-1}$ is affine.
\end{proof}

\subsection{Main theorem}

\begin{proposition} [Local version] \label{PS5.4:lclvrs}
	A transversely convex tube with cop is either a cylinder over a central ovaloid or a quadric.
\end{proposition}

\begin{proof}
	Let $\mathcal{T}$ be a transversely convex tube with \emph{cop}. By definition $\mathcal{T}$ is affinely isomorphic to a 	transversely convex tube $\mathcal{T}_{0}$ in \emph{standard position}. As central ovaloids are mapped to central ovaloids by 		affine isomorphisms \eqref{LS3.2:Glnprsvovld}, $\mathcal{T}_{0}$ also has \emph{cop}. The central curve of $			\mathcal{T}_{0}$ is affine by Lemma \ref{PS5.3:axslemm} (Axis lemma) and hence $\mathcal{T}_{0}$ is affinely 			isomorphic to a rectified transversely convex tube $\mathcal{T}_{1}$ with \emph{cop} in \emph{standard position}. Using	 		Proposition \ref{PS5.2:classofrcttube} we can conclude that $\mathcal{T}_{1}$ is either a cylinder over a central ovaloid 		or affinely isomorphic to a hypersurface of revolution with \emph{cop}. As $\mathcal{T}$ is affinely isomorphic to $			\mathcal{T}_{1}$ we can conclude that $\mathcal{T}$ is either a cylinder over a central ovaloid or affinely isomorphic to a 	hypersurface of revolution with \emph{cop}. In the case when $\mathcal{T}$ is affinely isomorphic to a hypersurface of 		revolution with \emph{cop} we can use \cite{bS09} to conclude that $\mathcal{T}$ is affinely isomorphic to a quadric and 	hence it is itself a quadric.
\end{proof}

\begin{remark}
	Affine invariance of the problem means the following:
	
	For every affine isomorphism $G \colon \mathbb{R}^{n} \to \mathbb{R}^{n}$, $G \circ F$ is again a proper, complete, 		immersion with cop if $F$ has these properties and $(G \circ F)$ is again a cylinder or a quadric depending on whether 		$F(M)$ is a cylinder or a quadric.
\end{remark}

\begin{theorem}
	Let $M^{n-1}$ be a smooth, connected, manifold, and $F \colon M \to \mathbb{R}^{n}$, $n \geq 4$, be a proper, 			complete, immersion with cop, then $F(M)$ is either a cylinder over a central ovaloid or a quadric.
\end{theorem}

\begin{proof}
	According to the \emph{cop} Definition \ref{DS1.1:cop}, there exists a cross-cut $\Gamma \subseteq F^{-1}(H)$ so that 		$F(\Gamma)$ is an ovaloid. By of the affine invariance of the problem we can assume that $H = H_{e_{n}, 0}$ and use  		Lemma \ref{LS3.3:exsttrncvxtube} (Existence of transversely convex tube) to construct a transversely convex tube with 		\emph{cop} about $F(\Gamma)$ with the following properties:
	
	\begin{figure}[h]
		\centering
		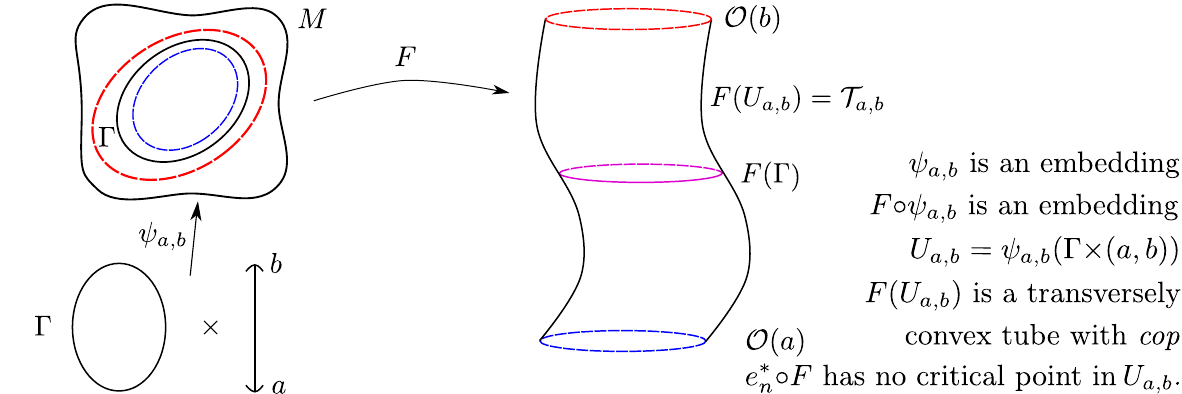
		\caption{Transversely convex tube about $F(\Gamma)$}
	\end{figure}

	There consequently exist minimal and maximal heights $-\infty \leq A < 0 < B \leq \infty$ so that 
	\begin{enumerate}
		\item $e_{n}^{*} \circ F$ has no critical point in $U_{a, b}$ \label{S5.4:cnd1}
		\item $F$ embeds $U_{a, b}$ into $\mathbb{R}^{n}$ as a transversely convex tube with \emph{cop} 					\label{S5.4:cdn2}
	\end{enumerate}
	for each $A < a < 0 < b < B$. 
	
	Now we consider three cases according to whether both, neither, or exactly one of the endpoints $A$ and $B$ are finite.
	
	\textbf{Case} $\mathbf{1 \colon (-\infty < A < B < \infty)}$ In this case, $e_{n}^{*} \circ F$ has no critical point in 			$U_{A, B}$, $\psi_{A, B}$ is an embedding and $F$ embeds $U_{A, B} = \psi( \Gamma \times (A, B))$ into $			\mathbb{R}^{n}$ as a transversely convex tube $\mathcal{T}_{A, B}$ with \emph{cop}. Therefore, using Proposition 		\ref{PS5.4:lclvrs} (Local version), we can conclude that $\mathcal{T}_{A, B}$ is either a cylinder 	over a central ovaloid, 	or a quadric, and because of the affine invariance of the problem we can further assume that $\mathbb{R}e_{n}$ is the axis 	of $\mathcal{T}_{A, B}$. Assume first that $\mathcal{T}_{A, B}$ is a cylinder. Our aim is to show that this is not possible 	by producing a contradiction to the choice of maximal and minimal heights.
	
	Since $\psi_{A, B}$ is an embedding and $\dim \bigl[ \Gamma \times (A, B) \bigr] = \dim M$, using the inverse function 		theorem we can conclude that it is a diffeomorphism onto its image, which is an open subset of $M$. Define the set $			\Gamma_{B}$ as 
	\[
		\Gamma_{B} = \Bigl\{ \lim_{k \to \infty} F^{-1}(p_{k}) \colon \{ p_{k}\}_{k \in \mathbb{N}} \subseteq 				\mathcal{T}_{A, B} \text{ Cauchy sequence}, \, e_{n}^{*}(p_{k}) \to B^{-} \Bigr\}.
	\]
	We eventually want to show that $\Gamma_{B}$ is a submanifold. But we first argue that the set $\Gamma_{B}$ exists, 		is compact and belongs to the boundary of $\psi_{A, B}(\Gamma \times (A, B))$. 
	
	\noindent
	\textbullet \quad The limit $\lim_{k} F^{-1}(p_{k})$ exists: \newline
	 \phantom{\textbullet \quad}Fix a Cauchy sequence $\{p_{k}\}_{k \in \mathbb{N}} \subseteq \mathcal{T}_{A, B}$ with 		$e_{n}^{*}(p_{k}) \to B^{-}$, $\epsilon > 0$, then there exist $I \in \mathbb{N}$, $\alpha_{ij} \colon J_{ij} \to 			\mathcal{T}_{A, B}$ smooth curve connecting $p_{i}$ to $p_{j}$ so that for each $i$ and $j \geq I$
	 \begin{align*}
	 	&\int_{J_{ij}} \lvert \dot{\alpha}_{ij}(t) \rvert \, \derv \mathcal{L}^{1}t < \epsilon \text{ and define for each $i$, $j 			\in \mathbb{N}$ the curve}\\
		&\beta_{ij} \colon J_{ij} \to M \quad \beta_{ij}(t) := F^{-1} \circ \alpha_{ij}(t).
	 \end{align*}
	Then the distance between the points $F^{-1}(p_{i})$ and $F^{-1}(p_{j})$ has a bound
	\[
		\dist_{M} \bigl[ F^{-1}(p_{i}), F^{-1}(p_{j}) \bigr] \leq \mathrm{length} \, \beta_{ij} = \mathrm{length} \, \alpha_{ij} 		< \epsilon
	\]	
	whenever $i$, $j \geq I$ because $M$ has the pullback metric. Therefore, $\{F^{-1}(p_{k})\}_{k \in \mathbb{N}}$ is a 		Cauchy sequence and the limit $\lim_{k} F^{-1}(p_{k})$ exists because $(M, \dist_{M})$ is a complete metric space, which 	follows because $F$ is a complete immersion.
	
	\noindent
	\textbullet \quad $\Gamma_{B} = \mathrm{Clos} \, \bigl[ \psi_{A, B} (\Gamma \times (A, B)) \bigr] \cap F^{-1} 			\bigl[ \mathcal{O}(B) \bigr]$ and it is compact. \newline
	\phantom{\textbullet \quad}Take $x \in \Gamma_{B}$, then $x = \lim_{k} F^{-1}(p_{k})$ and hence $x \in 				\mathrm{Clos} \, \bigl[ \psi_{A, B} (\Gamma \times (A, B)) \bigr]$, $F(x) = \lim_{k} p_{k}$, $e_{n}^{*} \bigl[ F(x) 		\bigr] = B$. This implies that $F(x) \in \mathcal{O}(B)$ and we get one inclusion. On the other hand, if we take $F(x) \in 		\mathcal{O}(B)$, $x = \lim_{k} x_{k}$ with $x_{k} = \psi_{A, B}(x_{k}', h_{k})$, $x_{k}' \in \Gamma$, $h_{k} \to 		B^{-}$, then $p_{k} = F \circ \psi_{A, B}(x_{k}', h_{k})$ is a Cauchy sequence, $e_{n}^{*}(p_{k}) = h_{k} \to B^{-}$ 		and $F^{-1}(p_{k}) = x_{k} \to x \in \Gamma_{B}$. Therefore, the equality holds and the set $\Gamma_{B}$ is compact 		because $F$ is a proper map. Moreover the definition of $\Gamma_{B}$ implies that $F(\Gamma_{B}) = \mathcal{O}(B)		$. A similar argument shows that
	\[
		\Gamma_{A} = \mathrm{Clos} \, \bigl[ \psi_{A, B} (\Gamma \times (A, B)) \bigr] \cap F^{-1} \bigl[ \mathcal{O}(A) 		\bigr]
	\]
	is compact and we can conclude that
	\begin{align*}
		\bdry \psi_{A, B} (\Gamma \times (A, B)) &= \mathrm{Clos} \, \bigl[ \psi_{A, B}(\Gamma \times (A, B)) \bigr]  											    \setminus \psi_{A, B}(\Gamma \times (A, B))\\
									      &= \Gamma_{A} \cup \Gamma_{B}.
	\end{align*}
	
	Now we want to show that $\Gamma_{B}$ is a submanifold of $M$ of dimension $n-2$ and conclude that $\Gamma_{B}$ 	is a cross-cut of $F$ relative to $H_{e_{n}, B}$. This conclusion, together with  Lemma \ref{LS3.3:exttrncvxtube} 			(Extension of transversely convex tube) will contradict the maximality of $B$. Because of the affine invariance of the 		problem we can assume that $B = 0$, $A < 0$ and denote the map $\psi_{A, 0}$ as $\psi$.
	
	Note that because $\mathcal{T}_{A,0}$ is a cylinder, for every $x \in \Gamma_{0}$
	\begin{equation*}
		\tng_{F(x)}H_{e_{n},0} + \derv F_{x} (\tng_{x} M) = \mathbb{R}^{n} \text{ and hence} 
		\begin{pmatrix}
			e_{n} \cdot \frac{\partial F}{\partial x_{1}}(x)\\
			\vdots\\
			e_{n} \cdot \frac{\partial F}{\partial x_{n-1}}(x) 
		\end{pmatrix}	
												  \ne \begin{pmatrix}
														0\\
														\vdots\\
														0
													\end{pmatrix}
	\end{equation*}
	hence $\nabla (e_{n}^{*} \circ F) \, \bigr\rvert_{x} \ne 0$. Therefore, we can conclude that $\nabla e_{n}^{*} \circ F$ does 	not vanish on the compact set $\Gamma_{0} = \psi(\Gamma \times \{0\})$ and we can define the vectorfield
	\[
		Y = \frac{\nabla e_{n}^{*} \circ F}{\lvert \nabla e_{n}^{*} \circ F \rvert^{2}}
	\]
	in an open neighborhood $V_{0}$ of $\Gamma_{0}$. Using the existence and uniqueness of flow box for $Y$ we obtain a 	triple $(V_{1}, a, \phi)$, where
	\begin{enumerate}
		\item $V_{1} \subseteq V_{0}$ is an open neighborhood of $\Gamma_{0}$ in $M$\\
		\item $\phi \colon (-a, a) \times V_{1} \to V_{0} \subseteq M$ is smooth\\
		\item For each $q \in V_{1}$, $t \to c_{q}(t) = \phi(t, q)$ is the integral curve of $Y$ and $c_{q}(0) = q$\\
		\item For each $\lvert t \rvert < a$, $\phi_{t} \colon V_{1} \to V_{0} \subseteq M$ is a diffeomorphism onto its 			image. 
	\end{enumerate}
	
	For each $q \in V_{1}$ the curve $c_{q}$ satisfies
	\begin{equation} \label{ES5.4:crvnolclext}
		\frac{\derv}{\derv t} (e_{n}^{*} \circ F \circ c_{q})(t) = \derv \, (e_{n}^{*} \circ F) \bigl[ c_{q}(t) \bigr] Y(c_{q}(t)) 		= 1
	\end{equation}
	for each $\lvert t \rvert < a$. Therefore, for each $q \in V_{1}$ the height function $e_{n}^{*} \circ F \circ c_{q}$ has no 		local maximum or minimum in $(-a, a)$.
	
	\begin{figure}[h]
		\centering
		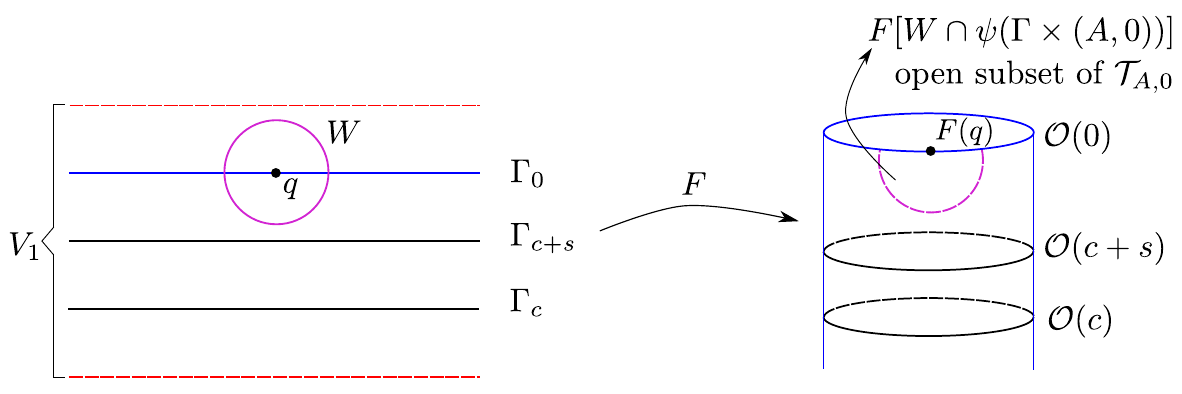
		\caption{Foliation by cross-cuts $\{\Gamma_{t}\}_{c \leq t \leq 0}$}
	\end{figure}
	
	Choose $c \in \mathbb{R}$ so that $-a < c < 0$, $A < c$, and $-c  = \lvert c \rvert < \dist_{M}(\Gamma_{0}, M \setminus 		V_{1})$. Define the submanifold $\Gamma_{c}$ of $M$, which is diffeomorphic to $\mathcal{O}(c)$ as 
	\[
		\Gamma_{c} = \bigl[ F \mid \psi(\Gamma \times (A, 0)) \bigr]^{-1} \mathcal{O}(c).
	\]
	Its $\dist_{M}$-distance to $\Gamma_{0}$ satisfies
	\[
		\dist_{M}( \Gamma_{0}, \Gamma_{c}) \leq \dist \bigl[ \mathcal{O}(0), \mathcal{O}(c) \bigr] = \lvert c \rvert < 			\dist_{M}(\Gamma_{0}, M \setminus V_{1})
	\]
	and hence we can conclude that $\Gamma_{c} \subseteq V_{1}$. Similarly for every $0 \leq s < \lvert c \rvert$ the 			submanifold $\Gamma_{c+s}$ satisfies $\Gamma_{c+s} \subseteq V_{1} \cap \psi(\Gamma \times (A, 0))$.
	
	\noindent
	\textbullet \quad For each $0 \leq s < \lvert c \rvert$, the diffeomorphism $\phi_{s}$ maps $\Gamma_{c}$ onto $			\Gamma_{c +s}$. First note that $\phi_{-s} = \phi^{-1}_{s}$ and hence it suffices to show that $\phi_{s}(\Gamma_{c}) 		\subseteq \Gamma_{c+s}$. Let $x \in \Gamma_{c}$ be arbitrary, then 
	
	\begin{align*}
		\phi_{s}(x) \in \Gamma_{c+s} &\iff \phi_{s}(x) \in \bigl[ F \mid \psi(\Gamma \times (A, 0)) \bigr]^{-1} \mathcal{O}			(c + s)\\
		&\iff \phi_{s}(x) \in \psi(\Gamma \times (A, 0)) \text{ and } (F \circ \phi_{s})(x) \in \mathcal{O}(c + s)\\
		&\iff \phi_{s}(x) \in \psi( \Gamma \times (A, 0)).
	\end{align*}
	The last equivalence is true because $\phi_{s}(x) \in \psi(\Gamma \times (A, 0))$ implies that its $F$ image $(F \circ 			\phi_{s})(x) \in \mathcal{T}_{A, 0}$ and $e_{n}^{*} \circ F \circ \phi_{s}(x) = c+s$ yields $\phi_{s}(x) \in \Gamma_{c		+s}$. By definition $\phi_{s}(x) \in \mathrm{im} \, \bigl( c_{x}\,  | \, [0, s] \bigr)$, and if this image is not a subset of $		\psi(\Gamma \times (A, 0))$ then there must exist $0 < s_{0} < s$ so that $c_{x}(s_{0}) \in \Gamma_{0} \cup 				\Gamma_{A}$, which implies the existence of a local extremum in $(0, s)$. This contradicts the observation in 
	\eqref{ES5.4:crvnolclext} and hence we can conclude that $\phi_{s}(x) \in \Gamma_{c+s}$.
	
	\noindent
	\textbullet \quad Assume that there exists $q \in \Gamma_{0} \setminus \phi_{-c}(\Gamma_{c})$, then there is an open 		neighborhood $W$ of $q$ with $\dist_{M}\bigl[ W, \phi_{-c}(\Gamma_{c}) \bigr] > 0$. Choose $0  < \epsilon < \lvert c 		\rvert$ so that for each $\lvert s \rvert < \epsilon$ $\dist_{M} \bigl[ W, \phi_{-c+s}(\Gamma_{c}) \bigr] > 0$. Then for 		each $- \epsilon < s < 0$
	\begin{align*}
		\bigl[\,  W \cap \psi( \Gamma \times (A, 0) ) \, \bigr] \cap \phi_{-c+s}(\Gamma_{c})  &= \emptyset\\
		\bigl[\, W \cap \psi( \Gamma \times (A, 0) \,\bigr] \cap \Gamma_{s} 			  &= \emptyset					\end{align*}
	and $F \bigl[ W \cap \psi( \Gamma \times (A, 0)) \bigr] \cap F(\Gamma_{s}) = \emptyset$ because $F \, | \, \psi(\Gamma 		\times (A, 0))$ is a bijection onto its image. Therefore, we must have for each $- \epsilon < s < 0$
	\[
		F \bigl[ W \cap \psi( \Gamma \times (A, 0)) \bigr] \cap \mathcal{O}(s) = \emptyset
	\]
	which is not possible. This implies that $\Gamma_{0} \subseteq \phi_{-c}(\Gamma_{c})$.
	
	\noindent
	\textbullet \quad Let $x \in \Gamma_{c}$ be arbitrary and consider the limit
	\begin{align*}
		\phi_{-c}(x) &= \lim_{k \to \infty} \phi_{-c - 1/k}(x) = \lim_{k \to \infty} \phi(-c - 1/k, x)\\
				   &= \lim_{k \to \infty} \psi(x_{k}, - 1/k)
	\end{align*}
	for some sequence $\{x_{k}\}_{k \in \mathbb{N}}$ in $\Gamma$ and hence we have
	\[
		\phi_{-c}(x) \in \mathrm{Clos} \bigl[\,  \psi(\Gamma \times (A, 0)) \,\bigr] \cap F^{-1} \bigl[ \mathcal{O}(0) \bigr] = 		\Gamma_{0}.
	\]
	So we get $\Gamma_{0} = \phi_{-c}(\Gamma_{c})$ and since $\phi_{-c}$ is a diffeomorphism and $\Gamma_{c}$ is an 		$n-2$ dimensional compact connected submanifold of $M$ we can conclude that $\Gamma_{0}$ is also an $n-2$ 			dimensional compact connected submanifold of  $M$. Therefore, using  Lemma \ref{LS3.3:exttubnbd} (Tubular 			neighborhood) we can conclude that $\Gamma_{0}$ is actually a cross-cut and Lemma \ref{LS3.3:exsttrncvxtube} 			(Existence of transversely convex tube) shows that there exists a transversely convex tube about $F(\Gamma_{0}) = 			\mathcal{O}(0)$. Finally, Lemma \ref{LS3.3:exttrncvxtube} (Extension of transversely convex tube) implies that tube about $	\mathcal{O}(0)$ extends $\mathcal{T}_{A, 0}$, contradicting the maximality of $B =0$.
	
	It follows that the transversely convex tube $\mathcal{T}_{A,B}$ is a quadric hypersurface of revolution about the $e_{n}		$-axis. Therefore there exist a vertical dilation $f_{1}$ and a vertical translation $f_{2}$ such that the affine isomorphism 		$f = f_{1} \circ f_{2}$ satisfies 
	\[
		f(\mathcal{T}_{A, B}) = \bigl\{ (x_{1}, \dotsc, x_{n}) \in \mathbb{R}^{n} \colon x_{1}^{2} + \cdots + x_{n-1}			^{2} \pm x_{n}^{2} = c \bigr\}
	\]
	for some $c \in \mathbb{R}$. Since $\mathcal{T}_{A, B}$ intersects some horizontal hyperplane along a compact set so 		does $f(\mathcal{T}_{A, B})$. After we remove the cylindrical case we can conclude that $f(\mathcal{T}_{A, B})$ is one 	of the following quadrics:
	\begin{description}
		\item[Nondegenerate] sphere, tube hyperboloid, convex hyperboloid, cone;
		\item[Degenerate] paraboloid.
	\end{description}
	
	Therefore, $\mathcal{T}_{A, B}$ is affinely isomorphic to one of the quadrics above. Discarding the cone, on all these 		hypersurfaces, horizontal cross-sections are spheres and hence the maximality and minimality of $B$ and $A$, respectively, 	must be dictated by the condition \eqref{S5.4:cnd1} alone. This implies that $e_{n}^{*} \circ F$ must have critical points 		on both boundaries of the open set $\psi_{A, B}(\Gamma \times (A, B))$. Namely, there exist $x \in \Gamma_{A}$, $y \in 	\Gamma_{B}$ so that
	\begin{align*}
		\nabla e_{n}^{*} \circ F \, \bigr\rvert_{x} = 0 \quad  &\Rightarrow e_{n} \cdot \frac{\partial F}{\partial x_{i}}(x) = 0 		\quad i = 1, \dotsc, n-1\\
		\nabla e_{n}^{*} \circ F \, \bigr\rvert_{y} = 0 \quad &\Rightarrow e_{n} \cdot \frac{\partial F}{\partial x_{i}}(y) = 0 		\quad i = 1, \dotsc, n-1.	
	\end{align*}
	
	But among the quadrics listed above, $e_{n}^{*}$ has multiple critical points only on the sphere, where it attains both a 		maximum and a minimum. So we can conclude that $F \bigl[ \mathrm{Clos} \, \psi_{A, B}(\Gamma \times (A, B)) \bigr]$ 		equals the complete ellipsoid $\mathrm{Clos} \, \mathcal{T}_{A, B}$. Since $M$ is connected we must have $F(M) = 		\mathrm{Clos} \, \mathcal{T}_{A, B}$.
	
	\textbf{Case} $\mathbf{2 \colon (- A = B = \infty)}$ In this case $\psi \colon \Gamma \times \mathbb{R} \to M$ is an 		 embedding and since $\dim (\Gamma \times \mathbb{R}) = \dim M$ the map $\psi$ must be a diffeomorphism onto 			its image, which is both open and closed in $M$. Since $M$ is connected $\psi(\Gamma \times \mathbb{R}) = M$.
	\begin{align*}
		F(M) = F \bigl[ \psi(\Gamma \times \mathbb{R}) \bigr] &= \bigcup_{r > 0} F \bigl[ \psi(\Gamma \times (-r, r)) \bigr]\\
								  			         &= \bigcup_{r > 0} F(U_{-r, r})
	\end{align*}
	where $U_{-r, r} = \psi (\Gamma \times (-r, r))$ for each $r > 0$.
	
	For each $r > 0$, $F(U_{-r, r})$ is a transversely convex tube with \emph{cop}. Using  Proposition \ref{PS5.4:lclvrs} 		(Local version) we can conclude that $F(U_{-r, r})$ is either a cylinder over a cental ovaloid with axis $\mathbb{R} e_{n}$ 	or a quadric hypersurface of revolution about $\mathbb{R}e_{n}$. Therefore, $F(U_{-r, r})$ must either be a cylinder over 	a central ovaloid or a tube hyperboloid for each $r > 0$.
	
	Assume that there exists $r_{0} > 0$ and $F(U_{-r_{0}, r_{0}})$ is a cylinder over a central ovaloid, then for every $r > 		r_{0}$
	\[
		F(U_{-r, r}) \cap \{ \lvert e_{n}^{*} \rvert < r_{0}\} = F(U_{-r_{0}, r_{0}})
	\]
	and hence $F(U_{-r, r})$ is a cylinder over a central ovaloid. So we can conclude that
	\[
		F(M) = \bigcup_{r>0} F(U_{-r, r})
	\]
	is a cylinder over a central ovaloid. Similarly, if there exists $r_{0} > 0$ so that $F(U_{-r_{0}, r_{0}})$ is a tube 			hyperboloid, then $F(M)$  is a again a tube hyperboloid.
	
	\textbf{Case} $\mathbf{3 \colon (\lvert A \rvert < B = \infty \text{ or } B < \lvert A \rvert  = \infty})$ Since the reflection 		about $H_{e_{n}, 0}$ is an affine isomorphism, these two cases are equivalent. So we assume $\lvert A \rvert < B = \infty$. 	For each $b > 0$
	\begin{align*}
		&\psi_{A, b} \colon \Gamma \times (A, b) \to M \text{ is an embedding}\\
		&U_{A, b} = \psi_{A, b} \bigl[ \Gamma \times (A, b) \bigr] \text{ is an open subset of $M$ and}\\
		&F(U_{A, b}) \text{ is a transversely convex tube with \emph{cop}.}
	\end{align*}
	
	$F(U_{A, b})$ can not be a cylinder over a central ovaloid because we can get, as in \textbf{Case 1}, a contradiction to the 	minimality of $A$. So we can conclude that for each $b > 0$, $F(U_{A, b})$ is either a spherical paraboloid or a convex 		hyperboloid. Assume that there exists $b_{0} > 0$ so that $F(U_{A, b_{0}})$ is a spherical paraboloid, then for each $b > 		b_{0}$, $F(U_{A, b})$ is again a spherical paraboloid and 
	\[
	 F \biggl( \mathrm{Clos} \bigcup_{b > 0} U_{A, b}\biggr)
	\]
	is a complete paraboloid. Since $M$ is connected, $F(M) = F(\mathrm{Clos} \cup_{b > 0} U_{A, b})$ and $F(M)$ is a 		complete paraboloid. Similarly, if there exists $b_{0} > 0$ so that $F(U_{A, b_{0}})$ is a convex hyperboloid, then $F(M)	$ is a convex hyperboloid.
\end{proof}

\end{document}